\documentclass[10pt, reqno]{amsart}
\usepackage{graphicx, amssymb, amsmath, amsthm}
\numberwithin{equation}{section}
\usepackage{xcolor}
\usepackage{amscd}
\usepackage{amssymb}
\usepackage{latexsym}

\usepackage{amsmath}
\usepackage{amscd}
\usepackage{cite}
\usepackage{color}
\usepackage{enumerate}
\usepackage{amsfonts}
\usepackage{graphicx}
\usepackage{mathrsfs}
\usepackage{setspace}
\usepackage{hyperref}
\def\pa{\partial}

\let\Re=\undefined\DeclareMathOperator*{\Re}{Re}
\let\Im=\undefined\DeclareMathOperator*{\Im}{Im}

\newcommand{\R}{\mathbb{R}}
\newcommand{\C}{\mathbb{C}}

\newcommand{\var}{\varepsilon}

\newtheorem{theorem}{Theorem}[section]

\newtheorem{lemma}[theorem]{Lemma}

\newtheorem{corollary}[theorem]{Corollary}

\newtheorem{proposition}[theorem]{Proposition}

\theoremstyle{definition}

\newtheorem{remark}[theorem]{Remark}
\newtheorem{assumption}[theorem]{Assumption}

\makeatletter
\newcommand{\Extend}[5]{\ext@arrow0099{\arrowfill@#1#2#3}{#4}{#5}}
\makeatother

\begin{document}
\title[Rough log--log blowup in $d\geq3$]{Rough log-log blowup solutions to mass-critical NLS in higher dimensions $d\geq 3$}

\author[Z. Ma]{Zuyu Ma}
	\address{Zuyu Ma
		\newline \indent Institute of Applied Physics and Computational Mathematics, Beijing, 100088, China}
	\email{mazuyu23@gscaep.ac.cn}
    
 \author[C. Miao]{Changxing Miao}
\address{Changxing Miao 
\newline \indent School of Mathematics and Physics, University of Science and Technology Beijing, Beijing 100083, China}
\email{miao\_changxing@ustb.edu.cn, miao\_changxing@iapcm.ac.cn}

\author[Y. Song]{Yilin Song}
	\address{Yilin Song
		\newline \indent Institute of Applied Physics and Computational Mathematics, Beijing, 100088, China}
	\email{songyilin21@gscaep.ac.cn}

	\author[J. Zheng]{Jiqiang Zheng}
	\address{Jiqiang Zheng
		\newline Institute of Applied Physics and Computational Mathematics, Beijing, 100088, China.
		\newline National Key Laboratory of Computational Physics, Beijing, 100088, China}
	\email{zheng\_jiqiang@iapcm.ac.cn, zhengjiqiang@gmail.com}


\begin{abstract}

We study the stability of the log--log blow-up regime for the
focusing mass-critical nonlinear Schr\"odinger equation under small
$H^s$ perturbations. Previously, stability was established for every
$0<s<1$ in dimension two by Colliander and Rapha\"el
[Math. Ann. (2009)] and subsequently extended to dimensions $d\geq3$
under the restriction $s>1/(1+\min\{1,4/d\})$ by Sun and the fourth
author [J. Math. Pures Appl. (2020)]. We remove this restriction and
establish stability throughout the full subcritical range $0<s<1$.
\end{abstract}

 \maketitle

\section{Introduction}

We consider finite-time singularity formation for the focusing
mass-critical nonlinear Schr\"odinger equation
\begin{equation}\label{equ:nls}\tag{NLS}
\left\{
\begin{aligned}
&i\partial_tu+\Delta u+|u|^{\frac4d}u=0,
\qquad (t,x)\in\mathbb R\times\mathbb R^d,\\
&u(0,x)=u_0(x).
\end{aligned}
\right.
\end{equation}
Its scaling symmetry is
\begin{equation}\label{equ:scalingintro}
u_\lambda(t,x)=\lambda^{\frac d2}u(\lambda^2t,\lambda x),
\qquad \lambda>0,
\end{equation}
which leaves both \eqref{equ:nls} and the $L^2$ norm invariant. Thus
$L^2$ is the critical space for the equation, while $H^s$ is subcritical
for every $s>0$. The Cauchy problem is locally well-posed in $H^s$ for
$0\leq s\leq1$; see \cite{CaW,GV79}.

In the energy space, the flow conserves the mass, energy, and momentum
\begin{align}
M(u)&=\int_{\mathbb R^d}|u(t,x)|^2\,dx,\label{equ:massintro}\\
E(u)&=\frac12\int_{\mathbb R^d}|\nabla u(t,x)|^2\,dx
-\frac{d}{2(d+2)}\int_{\mathbb R^d}|u(t,x)|^{2+\frac4d}\,dx,
\label{equ:energyintro}\\
P(u)&=\Im\int_{\mathbb R^d}\nabla u(t,x)\,\overline{u(t,x)}\,dx.
\label{equ:momentumintro}
\end{align}
The equation \eqref{equ:nls} admits the standing wave $e^{it}Q$, where the \emph{ground
state} $Q$ is known as the unique positive radial solution of
\begin{equation}\label{equ:groundstate}
\Delta Q-Q+Q^{1+\frac4d}=0
\end{equation}
(see \cite{Kw}).
If $u_0\in H^1$ satisfies
$\|u_0\|_{L^2}<\|Q\|_{L^2}$, then the corresponding solution is global in
$H^1$. This sharp criterion follows from conservation of mass and energy
and the sharp Gagliardo--Nirenberg inequality \cite{wen82},
\begin{equation}\label{equ:GNintro}
\|f\|_{L^{2+\frac4d}}^{2+\frac4d}
\leq \frac{d+2}{d}
\left(\frac{\|f\|_{L^2}}{\|Q\|_{L^2}}\right)^{\frac4d}
\|\nabla f\|_{L^2}^2.
\end{equation}
Global well-posedness and scattering in $L^2$ under the same mass condition
were subsequently established in \cite{KVZ2008,Dodson15}.

On the other hand, the equation is not globally well-posed for
arbitrary initial data. Indeed, the focusing nonlinearity allows
the energy to be negative, and for
$u_0\in\dot H^1(\mathbb R^d)\cap L^2(|x|^2dx)$ with $E(u_0)<0$,
the virial identity
\begin{equation}
\frac{d^2}{dt^2}\int_{\mathbb R^d}|x|^2|u(t,x)|^2\,dx
=16E(u_0)<0
\end{equation}
implies finite-time blow-up. This argument, however, does not yield any specific quantitative rate of blow-up.

At the threshold mass $M(Q)$, the pseudoconformal transform of the standing wave
$e^{it}Q$ gives the explicit solution
\begin{equation}\label{equ:pseudoconformalintro}
S(t,x)=\frac1{|t|^{\frac d2}}Q\left(\frac{x}{t}\right)
\exp\left(i\frac{|x|^2}{4t}-\frac{i}{t}\right),
\end{equation}
which blows up at $t=0$ with
$\|\nabla S(t)\|_{L^2}\sim |t|^{-1}$. The minimal-mass finite-time blow-up
solutions in $H^1$ were classified in \cite{Merle92,Merle93}, and the
corresponding $L^2$ classification in dimensions $d\leq15$ was obtained in
\cite{D2,D1}.

\subsection{Log--log blow-up in the energy space}
The dynamics above the threshold mass are much richer. For $\alpha^*>0$, set
\begin{equation}\label{loglog-initial}
\mathcal B_{\alpha^*}
:=\left\{u_0\in H^1(\mathbb R^d):
\|Q\|_{L^2}^2<\|u_0\|_{L^2}^2
<\|Q\|_{L^2}^2+\alpha^*\right\}.
\end{equation}
For $\alpha^*>0$ sufficiently small, two distinct blow-up mechanisms are known
for initial data in $\mathcal B_{\alpha^*}$:
\begin{itemize}
\item In dimensions $d\in\{1,2\}$, a family of blow-up solutions of
pseudoconformal type was constructed in \cite{Bourgain-Wang}, with blow-up
speed
\[
\|\nabla u(t)\|_{L^2}\sim\frac1{T-t},
\]
where $T$ is the blow-up time. This blow-up regime is unstable, as shown in
\cite{MRS13AJM}. 
\item A second family, suggested by the numerical simulations in
\cite{numerical}, exhibits the log--log blow-up speed. Its existence and
$H^1$ stability were first obtained in dimension one in \cite{Perelman}.
\end{itemize}
Later, in a series of works, Merle and Rapha\"el
\cite{MR05annmath,MR03GAFA,MR04Invemath,MR06JAMS} identified a stable
blow-up regime whose leading profile is $Q$ and whose scale obeys the
log--log law.

The log--log analysis is based on a coercivity property of the linearized
flow around $Q$. We use the following standard formulation from
\cite{MR05annmath}. Set $\Lambda=\frac d2+y\cdot\nabla$.

\begin{assumption}[Spectral property]\label{spectral-property}
Let $d\geq1$. Consider
\[
L_1=-\Delta+\frac2d\left(\frac4d+1\right)
Q^{\frac4d-1}y\cdot\nabla Q,
\qquad
L_2=-\Delta+\frac2dQ^{\frac4d-1}y\cdot\nabla Q,
\]
and
\[
H(\varepsilon,\varepsilon)
=(L_1\varepsilon_1,\varepsilon_1)
+(L_2\varepsilon_2,\varepsilon_2),
\qquad \varepsilon=\varepsilon_1+i\varepsilon_2\in H^1(\mathbb R^d).
\]
There exists $\delta_1>0$ such that, whenever
\[
(\varepsilon_1,Q)=(\varepsilon_1,\Lambda Q)
=(\varepsilon_1,yQ)=(\varepsilon_2,\Lambda Q)
=(\varepsilon_2,\Lambda^2Q)=(\varepsilon_2,\nabla Q)=0,
\]
one has
\[
H(\varepsilon,\varepsilon)
\geq\delta_1\int_{\mathbb R^d}
\left(|\nabla\varepsilon|^2+|\varepsilon|^2e^{-|y|}\right)dy.
\]
\end{assumption}

This property was proved in dimension one in \cite{MR05annmath}, verified
with numerical assistance in dimensions $2\leq d\leq4$ in \cite{FMR},
and established in dimensions $5\leq d\leq10$, as well as for radial
functions in dimensions $d=11,12$, in \cite{YRZ}. For the further discussion of such spectral property, we refer the reader
to the very recent work \cite{LiYang2026}. 

Based on the works \cite{MR05annmath,MR03GAFA,MR04Invemath,MR06JAMS,MR05CMP,FMR,Ra05Mathann}, we now have:

\begin{theorem}[Existence of a stable log--log regime]\label{thm:dynamssnls}
Let $d\geq1$ and assume the spectral property in dimension $d$. There exists a universal constant $\alpha^*>0$ such
that the following holds true.

\begin{enumerate}[(i)]
\item \underline{Existence and description of the log--log regime}: for any
initial data $u_0\in H^1$ with
\begin{equation}\label{equ:H1CRmass}
\|Q\|_{L^2}^2\leq\|u_0\|_{L^2}^2<\|Q\|_{L^2}^2+\alpha^*
\end{equation}
and
\[
E^G(u_0):=E(u_0)
-\frac12\frac{|P(u_0)|^2}{\|u_0\|_{L^2}^2}<0,
\]
then, the corresponding solution to
\eqref{equ:nls} blows up in finite time $0<T<+\infty$ according to the
following blowup dynamics: there exist geometrical parameters
$(\lambda(t),x(t),\gamma(t))\in(0,\infty)\times\mathbb R^d\times\mathbb R$
and an asymptotic residual profile $u^*\in L^2$ such that:
\[
u(t)-\frac1{\lambda^{\frac d2}(t)}
Q\left(\frac{x-x(t)}{\lambda(t)}\right)e^{-i\gamma(t)}
\to u^*\quad\text{in }L^2.
\]
The blowup point converges at blowup time:
\[
x(t)\to x(T)\in\mathbb R^d\quad\text{as }t\to T
\]
and the blowup speed is given by the log--log law:
\begin{equation}\label{loglog-speed}
\lambda(t)\sqrt{\frac{\log|\log(T-t)|}{T-t}}
\to\sqrt{2\pi}\quad\text{as }t\to T.
\end{equation}

\item \underline{$H^1$ stability of the log--log regime}: more generally,
the set of initial data satisfying \eqref{equ:H1CRmass} and such that the
corresponding solution to \eqref{equ:nls} blows up in finite time with the
log--log law \eqref{loglog-speed} is open in $H^1$.

\item \underline{Nonsmoothness of the asymptotic profile}: the residual
profile $u^*\in L^2$ satisfies:
\begin{equation}\label{equ:H1CRnonsmooth}
u^*\notin L^p,\quad\forall\,p>2.
\end{equation}
\end{enumerate}
\end{theorem}

\subsection{The low-regularity problem}

In Theorem \ref{thm:dynamssnls}, the log--log blowup set is
\emph{open} in $H^1$. Since \eqref{equ:nls} is $L^2$-critical, it is
natural to ask whether this stability persists under $H^s$
perturbations for $0\leq s<1$. Indeed, part of the
Merle--Rapha\"el analysis relies only on local $L^2$ interactions,
and the blowup mechanism itself is highly localized, suggesting
that such stability may hold at lower regularity; see
\cite{Planchon-Raphael} for further discussion of this localization.

This question was first resolved in dimension two by Colliander and
Rapha\"el \cite{CR09}. By combining the Merle--Rapha\"el blowup analysis
with the $I$-method \cite{CKSTT02}, they replaced the energy by the
modified energy $E(I_Nu)$ and proved stability under $H^s$ perturbations
for every $s>0$. The key observation in \cite{CR09} is that the
energy enters the estimates for the log--log dynamics with a factor of
$\lambda(t)^2$. The relation between
$\lambda(t)$ and the radiation terms allows for growth in the modified
energy, provided that the error after rescaling remains \emph{negligible}.
They therefore choose the frequency scale in terms of $\lambda(t)$ and
apply the $I$-method within a bootstrap for the log--log regime. 

More precisely, Colliander and Rapha\"el  introduced a nonnegative correction $\Xi(t)$.
Their almost conservation estimate takes the form
\[
\left|E(I_{N(t)}u(t))+\frac{\Xi(t)}{\lambda(t)^2}\right|
\leq\lambda(t)^{-2(1-\alpha_1)},\qquad \alpha_1=\alpha_1(s)>0;
\]
see \cite[Proposition~3.1 and Remark~3.2]{CR09}. The positive power of
$\lambda(t)$ gained after rescaling makes the error in this estimate
negligible in the log--log dynamics. Together with the virial estimates
and the $L^2$ flux calculation, this gives them the control of the excess
of mass and the geometrical parameters needed to close the bootstrap. The same modified-energy strategy was used in \cite{FanMendelson2024} to construct log--log solutions  from structured random $L^2$ perturbations in dimension two.

The argument in \cite{CR09}, however, depends strongly on the frequency
decomposition underlying the $I$-method. For the cubic nonlinearity, the
increment of the modified energy can be expanded into a finite sum of
multilinear expressions, and each frequency interaction can be estimated
separately. In dimensions $d\geq3$, the mass-critical nonlinearity
$$
F(u)=|u|^{\frac4d}u
$$
is non-algebraic.  The multilinear expansion used for
the cubic nonlinearity in \cite{CR09} does not apply directly in this
setting. To treat such nonlinearities within the $I$-method, Vi\c{s}an and Zhang \cite{VZ07} used Bony's paraproduct calculus to estimate the commutator between $I_N$ and $F$. These estimates were subsequently combined with the log--log blow-up analysis in \cite{SunZheng} to prove stability under $H^s$ perturbations for $s>1/(1+\nu)$.

The purpose of this paper is to prove that the log--log regime is stable in
$H^s$ for every $0<s<1$ and every $d\geq3$ for which the spectral property
holds. This extends the range $s>1/(1+\nu)$ in \cite{SunZheng} to the full
range $0<s<1$.

\subsection{Statement of the main result}
Now, we state our main result.
\begin{theorem}[$H^s$ stability of the log--log regime]\label{thm:main}
Let $d\geq3$ and assume the spectral property in dimension $d$. Fix
$0<s<1$ and let $u_0\in H^1(\mathbb R^d)$ generate a log--log blow-up
solution as in Theorem \ref{thm:dynamssnls}. There exists
$\varepsilon=\varepsilon(d,s,u_0)>0$ such that, whenever
$v_0\in H^s(\mathbb R^d)$ satisfies
\[
\|v_0-u_0\|_{H^s(\mathbb R^d)}<\varepsilon,
\]
the solution with initial data $v_0$ blows up at a finite time $T$. There
exist parameters
$(\lambda(t),x(t),\gamma(t))\in(0,+\infty)\times\mathbb R^d\times\mathbb R$
and a profile $v^\ast\in L^2(\mathbb R^d)$ with
$v^\ast\notin L^p$ for every $p>2$ such that
\begin{align}
v(t)-\frac1{\lambda(t)^{\frac d2}}
Q\left(\frac{\cdot-x(t)}{\lambda(t)}\right)e^{-i\gamma(t)}
&\longrightarrow v^\ast
\quad\hbox{in }L^2,\label{equ:convt}\\
x(t)&\longrightarrow x(T)\in\mathbb R^d,\label{equ:xtconx}\\
\lambda(t)\sqrt{\frac{\log|\log(T-t)|}{T-t}}
&\longrightarrow\sqrt{2\pi}.\label{equ:lamdts}
\end{align}
as $t\to T.$
\end{theorem}

\begin{remark}
We expect that the tools and techniques developed in this paper will also
improve low-regularity global well-posedness and scattering results for the
defocusing NLS with a non-algebraic nonlinearity.
\end{remark}

\subsection{Strategy of the proof}

Our main new ingredient is a smoothing operator $I_N$ constructed
as a weighted sum of heat operators at different scales. This operator retains the frequency profile of the usual $I$-multiplier,
but yields \emph{stronger} estimates for the nonlinear commutator
$I_NF(u)-F(I_Nu)$. These estimates then provide the control of the
modified-energy increment needed to close the log--log bootstrap for every $0<s<1$.

To motivate the construction, we first consider the commutator
of $F$ with a convolution operator. Write $F(u)=|u|^{4/d}u$
and $\nu=\min\{1,4/d\}$. Let $S_Mf=K_M*f$ be a
convolution operator with a real-valued kernel such that
$K_M,\nabla K_M\in L^1(\mathbb R^d)$ and
\[
\int_{\mathbb R^d}K_M(y)\,dy=1,
\qquad
\int_{\mathbb R^d}\nabla K_M(y)\,dy=0.
\]
We set
\[
H_M=S_MF(u)-F(S_Mu).
\]
Using these properties, we obtain
\begin{align}
H_M(t,x)
={}&\int_{\mathbb R^d}K_M(y)
\bigl[F(u(t,x-y))-F(S_Mu(t,x))\bigr]\,dy\notag\\
&-\left[\int_{\mathbb R^d}K_M(y)
\bigl[u(t,x-y)-S_Mu(t,x)\bigr]\,dy\right]
\cdot F^\prime(S_Mu(t,x))\notag\\
={}&\int_{\mathbb R^d}K_M(y)
R\bigl(S_Mu(t,x),u(t,x-y)-S_Mu(t,x)\bigr)\,dy,
\label{equ:introheatdefect}
\end{align}
and similarly,
\begin{equation}\label{equ:introheatdefectgrad}
\nabla H_M(t,x)=\int_{\mathbb R^d}\nabla K_M(y)
R\bigl(S_Mu(t,x),u(t,x-y)-S_Mu(t,x)\bigr)\,dy,
\end{equation}
where\footnote{See \eqref{nota1} for the definition of $b\cdot F^{\prime}(a)$}
\[
R(a,b)=F(a+b)-F(a)-b\cdot F^\prime(a).
\]
By Lemma~\ref{lem:nonltermest} below, this remainder satisfies
\[
|R(a,b)|\lesssim_d(|a|+|b|)^{\frac4d-\nu}|b|^{1+\nu}.
\]

These identities play a central role in the refined commutator estimates
of Section~\ref{sec:ACL}. Heuristically, they reduce the commutator bounds to estimates
for the Taylor remainder, which in turn depend on suitable weighted bounds
for the kernel. The kernel of the usual $I$-operator does not have the
required properties; in particular, its gradient is not integrable. We
therefore use the Gaussian kernel as the basic building block of our
construction, taking advantage of its localization in both physical and
frequency space.

Noticing that the rapid decay of Gaussian kernel prevents it
from satisfying the comparison estimates needed in the $I$-method. To recover the required frequency profile, we take $S_M=e^{\Delta/M^2}$ and form a weighted sum of heat operators at different scales. More precisely, set
\[
\alpha=1-s,
\qquad
\vartheta=2^{-\alpha},
\]
and define
\begin{equation}\label{equ:introheatI}
I_N=(1-\vartheta)\sum_{j\geq0}\vartheta^jS_{2^jN}.
\end{equation}
We call \eqref{equ:introheatI} the \emph{heat-kernel $I$-operator}. This
definition is motivated in physical space, but it also has precisely the
expected frequency profile. Indeed, let $m_N$ denote the symbol of $I_N$ and
suppose that $|\xi|\sim2^kN$ with $k\geq0$. The exponential decay of the heat symbol
suppresses the terms with $j\ll k$, so the principal contribution comes from
the indices $j\gtrsim k$. The geometric weights then give
\[
(1-\vartheta)\sum_{j\gtrsim k}\vartheta^j
\simeq\vartheta^k
\simeq\left(\frac{N}{|\xi|}\right)^{1-s}.
\]
Thus, the multiplier in \eqref{equ:introheatI} has order $s-1$ at high
frequency, as does the classical $I$-multiplier introduced by \cite{CKSTT02}. Taking this weighted sum of heat operators introduces an additional difficulty, since the nonlinearity couples the different scales. We observe, however, that the geometric average in
\eqref{equ:introheatI} satisfies the exact recursion
\begin{equation}\label{equ:introrecursion}
I_N=(1-\vartheta)S_N+\vartheta I_{2N}.
\end{equation}
This recursion provides the cancellation needed to control the nonlinear
error introduced by the weighted sum in \eqref{equ:introheatI}. We define
\begin{equation}\label{equ:commutator-def}
\mathcal{C}_N(u):=I_NF(u)-F(I_Nu).
\end{equation}
Setting $N_j=2^jN$ and iterating \eqref{equ:introrecursion}, we obtain the
decomposition
\[
\mathcal{C}_N(u)=\mathcal A_N(u)+\mathcal B_N(u),
\]
where
\begin{align*}
\mathcal A_N(u)
&=(1-\vartheta)\sum_{j\geq0}\vartheta^j
  \bigl[S_{N_j}F(u)-F(S_{N_j}u)\bigr],\\
\mathcal B_N(u)
&=\sum_{j\geq0}\vartheta^j
  \bigl[(1-\vartheta)F(S_{N_j}u)
  +\vartheta F(I_{2N_j}u)-F(I_{N_j}u)\bigr].
\end{align*}
The term $\mathcal A_N$ is the sum of the fixed-scale heat commutator in
\eqref{equ:introheatdefect}. Thus,  every summand already enjoys the heat-flow
cancellation. The term $\mathcal B_N$ is more subtle. It is precisely the nonlinear error introduced by this weighted sum. The recursion
expresses this contribution in terms of adjacent scales and gives
\[
I_{N_j}u=(1-\vartheta)S_{N_j}u+\vartheta I_{2N_j}u,
\]
so that $I_{N_j}u$ is the exact barycenter of the two adjacent scales. If both
nonlinear terms are expanded about this common center, their weighted
first-order contributions cancel identically. Each summand of
$\mathcal B_N$ is again a $C^{1,\nu}$ remainder, and it satisfies
\begin{align*}
&\left|(1-\vartheta)F(S_{N_j}u)
+\vartheta F(I_{2N_j}u)-F(I_{N_j}u)\right|\\
&\qquad\lesssim_{d,s}
(|S_{N_j}u|+|I_{2N_j}u|)^{\frac4d-\nu}
|(S_{N_j}-I_{2N_j})u|^{1+\nu}.
\end{align*}
There are thus two cancellations in the argument: one within each heat
scale, and one between adjacent scales. Both remove the linear Taylor term
before the individual frequency contributions are estimated.

The same $C^{1,\nu}$ remainder controls both the fixed-scale terms in
$\mathcal A_N$ and the cross-scale terms in $\mathcal B_N$. The resulting
estimates, for $d\geq3$, $0<s<1$, $N\geq1$, and a time interval $J$, are
\begin{align*}
\|\nabla\mathcal{C}_N(u)\|_{L_t^2L_x^{\frac{2d}{d+2}}(J\times\mathbb R^d)}
&\lesssim_{d,s}N^{-\nu}
\|u\|_{L_t^\infty L_x^2}^{\frac4d-\nu}
\|\langle\nabla\rangle I_Nu\|_{L_t^\infty L_x^2}^{\nu}
\|\langle\nabla\rangle I_Nu\|_{L_t^2L_x^{\frac{2d}{d-2}}},\\
\|\mathcal{C}_N(u)\|_{L_t^2L_x^{\frac{2d}{d+2}}(J\times\mathbb R^d)}
&\lesssim_{d,s}N^{-(1+\nu)}
\|u\|_{L_t^\infty L_x^2}^{\frac4d-\nu}
\|\langle\nabla\rangle I_Nu\|_{L_t^\infty L_x^2}^{\nu}
\|\langle\nabla\rangle I_Nu\|_{L_t^2L_x^{\frac{2d}{d-2}}}.
\end{align*}
Here all norms on the right hand side are taken over $J\times\mathbb R^d$.

These commutator estimates are the key to treating the full range
$0<s<1$. More precisely, they allow us to obtain control on local constancy
intervals analogous to that available in the usual $I$-method for
algebraic nonlinearities. Such control does not follow from the standard
commutator estimates in the non-algebraic case. They also yield
sufficiently strong bounds for the modified-energy increment to ensure
that its contribution to the log-log dynamics remains perturbative.
Moreover, the same estimates play an important role in the Lyapunov
analysis, since they allow us to work directly with the regularized
solution.\footnote{Similar ideas have already been used in
\cite{CR09,FanMendelson2024}.} In contrast, the weaker commutator estimates
available in \cite{SunZheng} lead the authors to derive the corresponding
bounds separately from the modulation equations for the unregularized
solution. The localized momentum flux arising in this argument requires
the restriction $s>1/2$, see the proof of Lemma~3.9 in \cite{SunZheng}.

Finally, our approach also applies to the mass-critical NLS in dimensions $d=1,2$, where the nonlinearity is algebraic. Indeed, as observed in \cite{CR09}, the log-log blowup mechanism
does not require particularly strong estimates for the modified-energy
increment. The estimates obtained here, without decomposing the
nonlinearity into individual frequency interactions, are sufficient for
this purpose. 

\subsection{Setting of the bootstrap}

We next incorporate the preceding estimates into the geometrical
decomposition in \cite{MR03GAFA,MR04Invemath}. We begin with the localized
self-similar profiles.

\begin{proposition}[Localized self-similar profiles, \cite{MR03GAFA,MR04Invemath}]\label{pro:self-similar}
There exist universal constants $C>0$ and $\eta^\ast>0$ with the following
property. For each $0<\eta<\eta^\ast$, there exist parameters
$\varepsilon^\ast(\eta)>0$ and $b^\ast(\eta)>0$ tending to zero with
$\eta$, such that for all $|b|<b^\ast(\eta)$ there exists a unique radial solution
$\tilde Q_b$ to the following elliptic equation
\begin{equation}
\begin{cases}
\Delta\tilde{Q}_b-\tilde{Q}_b+ib\Lambda\tilde{Q}_b+\tilde{Q}_b|\tilde{Q}_b|^\frac4d=0,\\
P_b=\tilde{Q}_b e^{i\frac{b|y|^2}4}>0~\text{in}~B_{R_b},\\
|\tilde{Q}_b(0)-Q(0)|<\varepsilon^\ast(\eta),~\tilde{Q}_b(R_b)=0,
\end{cases}
\end{equation}
where $R_b=\frac2{|b|}\sqrt{1-\eta}$ and
$B_{R_b}=\{y\in\R^d:|y|\leq R_b\}$. Let $\phi_b$ be a smooth spherically symmetric
function satisfying
\begin{equation*}
  \phi_b(y)=\begin{cases}
  1 \quad \text{if}\quad |y|\leq R_b^-=\sqrt{1-\eta}R_b,\\
  0\quad \text{if}\quad |y|\geq R_b,
  \end{cases} ~0\leq\phi_b\leq1,
\end{equation*}
and $\|\nabla\phi_b\|_{L^\infty}+\|\Delta\phi_b\|_{L^\infty}\to0$ when
$|b|\to0$. Denote by
\begin{equation}\label{def:qb}
Q_b(r)=\tilde{Q}_b(r)\phi_b(r).
\end{equation}
Then
\begin{align}\label{equ:qbclsq}
\big\|e^{Cr}(Q_b-Q)\big\|_{H^{10}\cap C^2}\to0\quad\text{as}\quad |b|\to0,\\\label{equ:paqbest}
\big\|e^{Cr}\big(\tfrac{\pa Q_b}{\pa b}+i\tfrac{|y|^2}4Q\big)\big\|_{C^2}\to0\quad \text{as}\quad |b|\to0,\\\label{equ:eqb}
|E(Q_b)|\leq e^{-\frac{C}{|b|}}.
\end{align}
Moreover, the profile $Q_b$ has  supercritical mass in the sense that
\begin{equation}\label{equ:qbcrimass}
0<\frac{d}{d(b^2)}\Big(\int|Q_b|^2\Big)\Big|_{b^2=0}=d_0<+\infty.
\end{equation}

\end{proposition}

\begin{remark}\label{Rem:qbequ}
The localized profile $Q_b$ is an approximate self-similar profile. We
define its error $\Psi_b$ by
\begin{equation}\label{def:psib}
\Delta Q_b-Q_b+ib\Lambda Q_b+Q_b|Q_b|^\frac4d=-\Psi_b.
\end{equation}
Here
\[
-\Psi_b=2\nabla\tilde{Q}_b\cdot\nabla \phi_b
+\tilde{Q}_b\Delta\phi_b+ib\tilde{Q}_b\,y\cdot\nabla\phi_b
+(\phi_b^{1+\frac4d}-\phi_b)|\tilde{Q}_b|^\frac4d\tilde{Q}_b.
\]
\end{remark}

The error $\Psi_b$ generates the outgoing radiation used in the refined
virial estimate.

\begin{lemma}[Linear outgoing radiation, \cite{MR04Invemath}]\label{lem:linoutrad}
There exist universal constants $C>0$ and $\eta^\ast>0$ such that the
following holds. For any $0<\eta<\eta^\ast$, there is $b^\ast>0$ such
that for $|b|<b^\ast$ there exists a unique solution $\zeta_b$ of
\begin{equation}
\begin{cases}
\Delta \zeta_b-\zeta_b+ib\Lambda\zeta_b=\Psi_b,\\
\int|\nabla \zeta_b(y)|^2\, dy<+\infty.
\end{cases}
\end{equation}
If
\begin{equation}\label{equ:gammab}
\Gamma_b:=\lim_{|y|\to+\infty}|y|^d|\zeta_b(y)|^2,
\end{equation}
then
\begin{equation}\label{equ:geb}
e^{-(1+C\eta)\frac{\pi}{|b|}}\leq\Gamma_b\leq e^{-(1-C\eta)\frac{\pi}{|b|}}.
\end{equation}
Moreover,
\[
\left\||y|^{\frac d2}(|\zeta_b|+|y||\nabla\zeta_b|)\right\|_{L^\infty(|y|\geq R_b)}
\leq\Gamma_b^{\frac12-C\eta},
\qquad
\int|\nabla\zeta_b|^2\leq\Gamma_b^{1-C\eta}.
\]
Define
\[
\theta(r)=
\begin{cases}
\displaystyle\int_0^r\sqrt{1-\frac{z^2}{4}}\,dz,
&0\leq r\leq2,\\[2mm]
\displaystyle\frac{\pi}{2},
&r\geq2.
\end{cases}
\]
For every $\sigma\in(0,5)$, there exists $\eta_\sigma>0$ such that, for any
$0<\eta<\eta_\sigma$, there exists $b^{\ast\ast}(\eta)>0$ such that for any $0<|b|<b^{\ast\ast}(\eta)$, it holds
\[
\left\|\zeta_be^{-\sigma\theta(|b||y|)/|b|}\right\|_{C^2(|y|\leq R_b)}
\leq\Gamma_b^{\frac12+\sigma/10}.
\]
Moreover, $\zeta_b$ is differentiable with respect to $b$ with the bound
\[
\|\partial_b\zeta_b\|_{C^1}\leq\Gamma_b^{\frac12-C\eta}.
\]
\end{lemma}

The heat-kernel $I$-operator connects the $H^s$ solution to the energy-space
modulation theory. By Proposition \ref{highcont}, the operator $I_N$ defined in
\eqref{equ:introheatI} maps $H^s$ to $H^1$ and satisfies
\begin{equation}\label{equ:equcont-intro}
\|u\|_{H^s}\lesssim_s\|\langle\nabla\rangle I_Nu\|_{L^2}\lesssim_sN^{1-s}\|u\|_{H^s}.
\end{equation}
If $\delta_\lambda f(x)=f(x/\lambda)$, then
\begin{equation}\label{equ:scaio-intro}
I_N\delta_\lambda=\delta_\lambda I_{N\lambda}.
\end{equation}

The case $s=1$ is contained in the $H^1$ theory. For fixed $0<s<1$, we
reduce Theorem~\ref{thm:main} to the following proposition, as in
\cite[Section~2.2]{CR09}.

\begin{proposition}[Description of the blow-up set]\label{prop:main}
Let $0<s<1$ and write
\begin{equation}\label{equ:iniassum}
u_0=G(0)+H(0),~G(0)\in H^1,~H(0)\in H^s.
\end{equation}
Suppose that $G(0)$ has the geometrical decomposition
\begin{equation}\label{equ:g0x}
G(0,x)=\frac1{\lambda(0)^\frac d2}\big(Q_{b(0)}+g(0)\big)\Big(\frac{x-x(0)}{\lambda(0)}\Big)e^{-i\gamma(0)},
\end{equation}
and the following bounds hold.

$(i)$ The scaling parameters satisfy
\begin{equation}\label{equ:b0l0}
0<b(0)\ll1,~~0<\lambda(0)\leq e^{-e^\frac{2\pi}{3b(0)}},
\end{equation}

$(ii)$ The excess mass satisfies
\begin{equation}\label{equ:excmass}
\|g(0)\|_{L^2}+\|H(0)\|_{L^2}\ll1.
\end{equation}

$(iii)$ The rough component satisfies
\begin{equation}\label{equ:h0hs}
\|H(0)\|_{H^s}\leq \lambda(0)^{10}.
\end{equation}

$(iv)$ The energy-space remainder satisfies
\begin{equation}\label{equ:g0small}
\int|\nabla g(0)|^2 +\int |g(0)|^2e^{-|y|}\leq \Gamma_{b(0)}^\frac34.
\end{equation}

$(v)$ The energy and momentum of $G(0)$ satisfy
\begin{align}\label{equ:eg0}
|E(G(0))|\leq&\frac1{\sqrt{\lambda(0)}},\\\label{equ:gg0}
|P(G(0))|\leq&\frac1{\sqrt{\lambda(0)}}.
\end{align}
Then the $H^s$ solution to \eqref{equ:nls} with initial data $u_0$ blows up in finite time in
the log--log regime and satisfies all conclusions of Theorem \ref{thm:main}.

\end{proposition}

\begin{proof}[Proof of Theorem~\ref{thm:main} assuming
Proposition~\ref{prop:main}]
Let $G(t)$ be the $H^1$ solution with initial data $u_0$ in
Theorem~\ref{thm:main}. Following \cite[Section~2.2]{CR09}, we choose
$t_0$ sufficiently close to its blowup time so that $G(t_0)$ admits the
decomposition
\[
G(t_0,x)=\frac1{\lambda(t_0)^{\frac d2}}
\bigl(Q_{b(t_0)}+g(t_0)\bigr)
\left(\frac{x-x(t_0)}{\lambda(t_0)}\right)e^{-i\gamma(t_0)},
\]
with the controls \eqref{equ:b0l0} and \eqref{equ:g0small} at $t_0$,
and with $\|g(t_0)\|_{L^2}\ll1$. These properties follow from the
$H^1$ log--log analysis in \cite{MR06JAMS}. By conservation of energy
and momentum, and by taking $t_0$ closer to blowup if necessary, we also
have
\[
|E(G(t_0))|+|P(G(t_0))|
=|E(u_0)|+|P(u_0)|\leq\frac1{\sqrt{\lambda(t_0)}}.
\]
By continuous dependence of the flow in $H^s$ on the fixed interval
$[0,t_0]$, we may choose $\varepsilon(d,s,u_0)>0$ so that, for
$\|v_0-u_0\|_{H^s}<\varepsilon(d,s,u_0)$, the corresponding solution
$v(t)$ exists on $[0,t_0]$ and satisfies
\[
\|v(t_0)-G(t_0)\|_{H^s}\leq\lambda(t_0)^{10}.
\]
We set $H(t_0)=v(t_0)-G(t_0)$ and translate $t_0$ to zero. The initial
data then satisfy the hypotheses of Proposition~\ref{prop:main}, which
gives the conclusions of Theorem~\ref{thm:main}.
\end{proof}

The rest of the paper is devoted to the proof of
Proposition~\ref{prop:main}. Let $u_0$ satisfy its hypotheses. We rewrite
\eqref{equ:iniassum} as
\begin{equation}\label{equ:reu0}
u(0,x)=\frac1{\lambda(0)^\frac d2}\big(Q_{b(0)}+\varepsilon(0)\big)\Big(\frac{x-x(0)}{\lambda(0)}\Big)e^{-i\gamma(0)}
\end{equation}
with $\varepsilon(0)=g(0)+h(0)$ and
\begin{equation}\label{equ:defh0}
H(0,x)=\frac1{\lambda(0)^\frac d2}h\Big(0,\frac{x-x(0)}{\lambda(0)}\Big)e^{-i\gamma(0)}.
\end{equation}
By scaling and \eqref{equ:h0hs},
\begin{equation}\label{equ:h0small}
\|h(0)\|_{\dot{H}^s}=\lambda(0)^s\|H(0)\|_{\dot{H}^s}\leq \lambda(0)^{10+s}.
\end{equation}
Combining this with \eqref{equ:excmass} and \eqref{equ:g0small}, we obtain
\begin{equation}\label{equ:vesma}
\|\varepsilon(0)\|_{H^s}\leq\|g(0)\|_{H^s}+\|h(0)\|_{H^s}\ll1.
\end{equation}
We next derive a frequency-localized version of \eqref{equ:g0small} for
$\varepsilon(0)$. Set
\begin{equation}\label{equ:n0def}
N(0)=\Big(\frac1{\lambda(0)}\Big)^\frac1{\beta}, 
\end{equation}
where $\beta=1-\frac{2s}{3}$. Then
\begin{equation}\label{equ:nlamd}
1\ll\Big(\frac1{\lambda(0)}\Big)^\frac{1-\beta}{\beta}=N(0)\lambda(0).
\end{equation}
Using \eqref{equ:geb}, \eqref{equcont}, \eqref{equ:b0l0}, and
\eqref{equ:h0small}, we obtain
\begin{equation*}
\int|I_{N(0)\lambda(0)}\nabla h(0)|^2\lesssim (N(0)\lambda(0))^{2(1-s)}\|h(0)\|_{\dot{H}^s}^2\lesssim \lambda(0)^{10}\leq\Gamma_{b(0)}^{10}.
\end{equation*}
This together with \eqref{equ:h0hs} and \eqref{equ:g0small} implies
\begin{equation}\label{equ:ncsma}
\int|I_{N(0)\lambda(0)}\nabla \varepsilon(0)|^2+\int|\varepsilon(0)|^2e^{-|y|}\leq\Gamma_{b(0)}^\frac34.
\end{equation}

Let $[0,T)$ be the maximal forward existence interval of the $H^s$
solution $u(t)$. We now extend the decomposition \eqref{equ:reu0} to
$u(t)$. By \eqref{equ:vesma} and continuity of the flow in $H^s$, we may consider
the evolution on a small time interval on which the remainder
$\varepsilon(t)$ remains small in $H^s$. We then apply the implicit
function theorem, as in \cite[Section~2.2]{CR09}, to introduce a
geometrical decomposition of the solution. We fix the geometrical
parameters $\lambda(t),b(t),\gamma(t),x(t)$ by imposing the orthogonality
conditions below. More precisely, we have the following modulation lemma.
\begin{lemma}[Nonlinear modulation theory, \cite{MR05annmath,MR03GAFA,MR06JAMS}]\label{lem:moduthe}
There exist $0<T_0<T$ and continuous functions
\[
(\lambda,\gamma,x,b):[0,T_0]\longrightarrow
(0,+\infty)\times\R\times\R^d\times\R
\]
such that
\begin{equation}\label{equ:gemdec}
\varepsilon(t,y)=e^{i\gamma(t)}\lambda(t)^\frac d2u\big(t,\lambda(t)y+x(t)\big)-Q_{b(t)}(y)
\end{equation}
satisfies the following orthogonality conditions. Writing
$Q_{b(t)}=\Sigma+i\Theta$ and $\varepsilon=\varepsilon_1+i\varepsilon_2$, one has
\begin{align}
(\varepsilon_1(t),|y|^2\Sigma)+(\varepsilon_2(t),|y|^2\Theta)&=0,
\label{equ:cy2}\\
(\varepsilon_1(t),y\Sigma)+(\varepsilon_2(t),y\Theta)&=0,
\label{equ:cy1}\\
-(\varepsilon_1(t),\Lambda\Theta)+(\varepsilon_2(t),\Lambda\Sigma)&=0,
\label{equ:cth}\\
-(\varepsilon_1(t),\Lambda^2\Theta)+(\varepsilon_2(t),\Lambda^2\Sigma)&=0.
\label{equ:cth2}
\end{align}

\end{lemma}

\begin{remark}
The decomposition follows from local $L^2$ smallness and the smoothness and
spatial decay of $Q_b$. By \eqref{equ:ncsma}, the change in the parameters of
\eqref{equ:reu0} needed to impose the orthogonality conditions is exponentially
small in $b(0)$. We use the same notation for the adjusted parameters at
$t=0$.
\end{remark}

We now turn to the bootstrap argument for Proposition~\ref{prop:main}. We have established the initial
estimates above and fixed the geometrical parameters by the orthogonality
conditions \eqref{equ:cy2}--\eqref{equ:cth2}. Then as observed in \cite[Section~2.2]{CR09}, we claim that
these initial controls determine a \emph{trapped dynamical region}. More
precisely, we shall assume uniform bounds on the geometrical parameters
and the remainder on a time interval $[0,T^+]$, and then prove in
Lemma~\ref{lem:bootstrap} that all these bounds can be improved. We
will then use continuity to extend the bounds for as long as the
solution exists.

Consider $0<T^+<T$ such that the solution $u$ admits the decomposition
on $[0,T^+]$
\begin{equation}\label{equ:decass}
u(t,x)=\frac1{\lambda(t)^\frac d2}\big(Q_{b(t)}(\cdot)+\var(t,\cdot)\big)\Big(\frac{x-x(t)}{\lambda(t)}\Big)e^{-i\gamma(t)},~t\in[0,T^+]
\end{equation}
with the orthogonality conditions \eqref{equ:cy2}--\eqref{equ:cth2}. Assume
the following bounds hold on $[0,T^+]$.

$(i)$ {\bf Control of $b(t)$ and the $L^2$ mass of the remainder:}
\begin{equation}\label{equ:btvart}
b(t)>0\quad\text{and}\quad b(t)+\|\var(t)\|_{L^2}\leq10\big(b(0)+\|\var(0)\|_{L^2}\big);
\end{equation}

$(ii)$ {\bf Control and monotonicity of the scaling parameter:}
\begin{equation}\label{equ:ltass}
\lambda(t)\leq e^{-e^\frac{\pi}{100b(t)}}
\end{equation}
and almost monotonicity:
\begin{equation}\label{equ:almons}
\forall~0<t_1\leq t_2\leq T^+,~\lambda(t_2)\leq \frac32\lambda(t_1).
\end{equation}
Choose integers $k_0\leq k^+$ such that
\begin{equation}\label{equ:l0k0}
\frac1{2^{k_0}}\leq\lambda(0)\leq\frac1{2^{k_0-1}},~\frac1{2^{k^+}}\leq\lambda(T^+)\leq\frac1{2^{k^+-1}},
\end{equation}
For $k_0\leq k\leq k^+$, let $t_k$ be the first time such that
\begin{equation}\label{equ:ltk}
\lambda(t_k)=\frac1{2^k},
\end{equation}
and assume that the scale-doubling intervals satisfy
\begin{equation}\label{equ:dotime}
t_{k+1}-t_k\leq k\lambda(t_k)^2.
\end{equation}

$(iii)$ {\bf Frequency-localized control of the remainder:} Set
\begin{equation}\label{equ:ntass}
N(t)=\Big(\frac1{\lambda(t)}\Big)^\frac1{\beta},
\end{equation}
and assume
\begin{equation}\label{equ:vartass}
\int|I_{N(t)\lambda(t)}\nabla \var(t)|^2+\int|\var(t)|^2e^{-|y|}\leq \Gamma_{b(t)}^\frac14.
\end{equation}

The central step is the simultaneous improvement of these bootstrap bounds.
\begin{lemma}[Bootstrap lemma]\label{lem:bootstrap}
Under the preceding assumptions, the following estimates hold on $[0,T^+]$:
\begin{align}\label{equ:btimp}
b(t)>0\quad\text{and}\quad b(t)+\|\var(t)\|_{L^2}\leq&5\big(b(0)+\|\var(0)\|_{L^2}\big),\\\label{equ:lamtimp}
\lambda(t)\leq& e^{-e^\frac{\pi}{10b(t)}},\\\label{equ:lmt2imp}
\forall~0<t_1\leq t_2\leq T^+,~\lambda(t_2)\leq& \frac54\lambda(t_1),\\\label{equ:tktk-1}
t_{k+1}-t_k\leq& \sqrt{k}\lambda(t_k)^2,\\\label{equ:lamtimpse}
\int|I_{N(t)\lambda(t)}\nabla \var(t)|^2+\int|\var(t)|^2e^{-|y|}\leq& \Gamma_{b(t)}^\frac23.
\end{align}
\end{lemma}

\begin{remark}\label{rem:uths}
By \eqref{equ:equcont-intro} and the bootstrap assumptions
\eqref{equ:btvart}, \eqref{equ:ntass}, and \eqref{equ:vartass}, we have
\begin{equation}\label{equ:varsmall}
\|\var(t)\|_{H^s}
\lesssim
\|I_{N(t)\lambda(t)}\var(t)\|_{H^1}
\ll 1.
\end{equation}
Together with the geometrical decomposition \eqref{equ:decass}
and the approximation \eqref{equ:qbclsq}, this yields
\begin{equation}\label{equ:uthslam}
\|u(t)\|_{\dot H^s}
=
\frac{\|Q_{b(t)}+\var(t)\|_{\dot H^s}}{\lambda(t)^s}
\sim
\frac{1}{\lambda(t)^s}.
\end{equation}
By the conservation of mass and \eqref{equ:almons}, if $\lambda(0)\ll1$, then we also have
\[
\|u(t)\|_{H^s}\sim\frac{1}{\lambda(t)^s}.
\]
\end{remark}

Lemma \ref{lem:bootstrap} closes the trapped bootstrap and proves Proposition
\ref{prop:main}, which yields Theorem \ref{thm:main}.

\subsection*{Organization of the paper}

In Section~\ref{sec:ACL}, we develop the analytic estimates associated with the heat-kernel
$I$-operator. We first establish its multiplier and commutator
properties, and then derive the local well-posedness estimates and the
almost conservation laws needed in the bootstrap argument. In Section~\ref{sec:Loglog},
we prove Lemma \ref{lem:bootstrap}. We begin with the control of the
geometrical parameters, then establish the virial dispersion estimates,
and finally close the bootstrap argument and complete the proof of
Proposition \ref{prop:main}.

\medskip

\subsection*{Acknowledgements}

C. Miao was supported by the National Key R\&D Program of China under
Grant 2022YFA1005700 and by NSFC Grants 12371095 and 12531005.
J. Zheng was supported by the National Key R\&D Program of China under
Grant 2021YFA1002500 and by NSFC Grant 12671284.

\section{Notation and almost conservation law}\label{sec:ACL}

Throughout the rest of the paper we fix a dimension $d\geq3$ and a regularity exponent $0<s<1$. We set
\begin{equation}\label{equ:alphaa}\alpha=1-s,\,\, \vartheta=2^{-\alpha},\,\, \beta=1-\frac{2s}{3}\,\, \mbox{and } \nu=\min\left\{1,\frac4d\right\}.\end{equation}
Thus $0<\alpha<1$, $\frac13<\beta<1$, and $0<\nu\leq1$.  The endpoint $s=1$ is already covered by the $H^1$ theory.

\subsection{Some notation}
For nonnegative quantities $X$ and $Y$, we write $X\lesssim Y$ to denote the estimate $X\leq CY$ for some $C>0$.  If $X\lesssim Y\lesssim X$, we write $X\sim Y$.  Dependence of implicit constants on the regularity exponent will be displayed by a subscript when needed.

For a spacetime slab $I\times\R^d$, we write $L_t^qL_x^r(I\times\R^d)$ for the usual mixed Lebesgue space, and abbreviate $L_t^qL_x^q$ by $L_{t,x}^q$.  The Fourier transform is given by
\[
\widehat f(\xi)=(2\pi)^{-\frac d2}\int_{\R^d}e^{-ix\cdot\xi}f(x)\,dx.
\]
We define $|\nabla|^\rho$ and $\langle\nabla\rangle^\rho$ by the symbols $|\xi|^\rho$ and $\langle\xi\rangle^\rho=(1+|\xi|^2)^{\rho/2}$, respectively, and use the standard homogeneous and inhomogeneous Sobolev norms.

In this paper, we denote by
\begin{equation}\label{equ:Fdef}
 F(z)=|z|^{\frac4d}z,\qquad z\in\C\simeq\R^2.
\end{equation}
Then,
$$
F_z(z)=\frac{\partial F}{\partial z}=\frac{d+2}{d}|z|^\frac4d,\,\,F_{\bar z}(z)=\frac{\partial F}{\partial \bar z}(z)=\frac2d|z|^\frac4d\frac{z}{\bar z}.
$$
We write $F^\prime$ for the vector $(F_z,F_{\bar z})$ and adopt the notation
\begin{align}\label{nota1}
    w\cdot F^\prime(z)=wF_z(z)+\bar w F_{\bar z}(z).
\end{align}

The following elementary regularity statement will be used both in the
 commutator argument and in the modulation estimates. 
\begin{lemma}\label{lem:nonltermest}
For all $z,w\in\C$,
\begin{align}
\bigl|F'(z)-F'(w)\bigr|
&\lesssim_{d} (|z|+|w|)^{\frac{4}{d}-\nu}\,|z-w|^{\nu}
 \label{equ:DFholderall}\\
 \bigl|F(z+w)-F(z)-w\cdot F'(z)\bigr|
&\lesssim_{d} (|z|+|w|)^{\frac{4}{d}-\nu}\,|w|^{1+\nu}.
 \label{equ:Rpoint}
\end{align}
In particular, $F\in C^{2,\frac13}$ in $d=3$ and $F\in C^{1,\frac4d}$ in $d\geq4$.
\end{lemma}

\subsection{The heat-kernel $I$-operator}
For $M>0$, let $S_M$ be the heat semigroup
$S_M=e^{\frac{\Delta}{M^2}}$. First, we denote the radial symbol by
\[
m(r)=(1-\vartheta)\sum_{j\ge 0}\vartheta^j e^{-4^{-j}r^2}.
\]
We use the following positive average of heat operators:
\begin{equation}\label{equ:Iheatdef}
I_N=(1-\vartheta)\sum_{j=0}^{\infty}\vartheta^jS_{2^jN}.
\end{equation}
Thus $I_N$ is the Fourier multiplier with symbol
\begin{equation}\label{equ:mheat}
m_N(\xi)=m\Big(\frac{|\xi|}{N}\Big).
\end{equation}
The advantage of \eqref{equ:Iheatdef} is the exact recursion
\begin{equation}\label{equ:Irecursion}
I_N=(1-\vartheta)S_N+\vartheta I_{2N},
\end{equation}
which will encode the second-order cancellation in the nonlinearity.

To show several boundednesses of operators associated with $I_N$, we recall the following classic multiplier lemma:
\begin{lemma}[Mikhlin multiplier lemma, \cite{Stein70}]\label{lem:Mikhlin}
    Let $a$ be a smooth Fourier symbol away from the origin and satisfies the bound
    \begin{align*}
        \big|\partial_\xi^\gamma a(\xi)\big|\lesssim A|\xi|^{-|\gamma|}, \,\,\xi\neq0,\,\,|\gamma|\leq L_d
    \end{align*}
    where $L_d$ is any fixed integer strictly larger than $\frac d2$ and $\gamma\in\mathbb N^d$ is a multi-index. Then the Fourier multiplier is bounded on $L^p$, $1<p<\infty$ with operator norm $O_{d,p}(A)$.
\end{lemma}
Now, we are in a position to prove the following boundedness of $I_N$.
\begin{proposition}[Properties of $I_N$]\label{highcont}
Let $1<p<\infty$ and $N\geq1$.  Then:

\noindent $(i)$ 
\begin{equation}\label{equ:Icontract} 
\|I_Nf\|_{L^p}\leq\|f\|_{L^p}.
\end{equation}

\noindent $(ii)$ For every multi-index $\gamma$,
\begin{equation}\label{equ:msymbol}
|\partial_\xi^\gamma m_N(\xi)|\lesssim_{s,\gamma}(N+|\xi|)^{-|\gamma|}m_N(\xi),
\end{equation}
and
\begin{equation}\label{equ:masymp}
m_N(\xi)\sim_s
\begin{cases}
1,&|\xi|\leq N,\\
\left(\dfrac{N}{|\xi|}\right)^{1-s},&|\xi|\geq N.
\end{cases}
\end{equation}
Consequently,
\begin{equation}\label{equcont}
\|f\|_{H^s}\lesssim_s\|\langle\nabla\rangle I_Nf\|_{L^2}\lesssim_sN^{1-s}\|f\|_{H^s}.
\end{equation}

\noindent $(iii)$ If $M\geq N$, then
\begin{equation}\label{equ:multcompare}
\|\langle\nabla\rangle I_Mf\|_{L^p}\lesssim_s\left(\frac{M}{N}\right)^{1-s}\|\langle\nabla\rangle I_Nf\|_{L^p}.
\end{equation}
For every real $k\geq0$ and $0\leq\rho\leq2$,
\begin{equation}\label{equ:profileapprox}
\begin{aligned}
\|(1-I_N)f\|_{\dot H^k}
&\lesssim_{s,\rho}N^{-\rho}\|f\|_{\dot H^{k+\rho}},\\
\|(1-I_N)f\|_{H^k}
&\lesssim_{s,\rho}N^{-\rho}\|f\|_{H^{k+\rho}}.
\end{aligned}
\end{equation}
Moreover, for every $M\geq1$,
\begin{align}
\|(1-I_M)f\|_{L^p}+\|(1-S_M)f\|_{L^p}
&\lesssim_sM^{-1}\|\langle\nabla\rangle I_Mf\|_{L^p},\label{equ:IMSMdiff}\\
\|\nabla S_Mf\|_{L^p}+\|\nabla I_Mf\|_{L^p}
&\lesssim_s\|\langle\nabla\rangle I_Mf\|_{L^p}.\label{equ:heatDM}
\end{align}
In particular,
\begin{equation}\label{equ:L2tailI}
\|(1-I_N)f\|_{L^2}\lesssim_sN^{-1}\|\langle\nabla\rangle I_Nf\|_{L^2}.
\end{equation}

\noindent $(iv)$ Let $\dot I_N$ denote the multiplier with symbol $N\partial_Nm_N$.  Then
\begin{equation}\label{equ:Idot}
\|\dot I_Nf\|_{L^2}\lesssim_sN^{-1}\|\langle\nabla\rangle I_Nf\|_{L^2}.
\end{equation}
Finally, if $\delta_\lambda f(x)=f(x/\lambda)$, then
\begin{equation}\label{equ:scaio}
I_N\delta_\lambda=\delta_\lambda I_{N\lambda}.
\end{equation}
\end{proposition}

\begin{proof}Fix $s\in(0,1)$ and set
$
\alpha=1-s,\,\,\vartheta=2^{-\alpha}.
$
By  definition, we have the $L^p$ boundedness of $I_N$,
\begin{align*}
    \|I_N f\|_{L^p}\leq (1-\vartheta)\sum_{j=0}^\infty\vartheta^j\|S_{2^jN}f\|_{L^p}\leq\|f\|_{L^p},
\end{align*}
which implies \eqref{equ:Icontract}.

Next, we  show \eqref{equ:msymbol}. First, we prove the following derivative bound
\begin{equation}\label{derivative-m(r)}
\big|(r\partial_r)^q m(r)\big|\lesssim_{s,q}m(r),\qquad r>0,\quad q\in\Bbb N.
\end{equation}
Denote the radial derivative by  $\mathcal D=r\partial_r$, for any $q\in\Bbb N$,
\[
 \mathcal D^q e^{-4^{-j}r^2}=P_q(4^{-j}r^2)e^{-4^{-j}r^2},
\]
where $P_q$ is a fixed polynomial. Hence
$$
|\mathcal D^q e^{-4^{-j}r^2}|\lesssim_q e^{-c4^{-j}r^2}.
$$
Using exactly the same dyadic partition as above, the sum
\[
\sum_{j\ge 0}\vartheta^j e^{-c4^{-j}r^2}\lesssim_{s,q}m(r)
\]
is bounded by  $m(r)$ up to a constant depending on $s,q$. Hence, we have proved the claim \eqref{derivative-m(r)}.

Now, we use this derivative bound to deduce the estimate of $m(|\xi|/N)$. Let $k=|\gamma|$. First assume $\rho=|\xi|\le 2N$. By the direct calculation, we have
\[
\partial_\xi^\gamma e^{-\rho^2/(2^{2j}N^2)}
=(2^jN)^{-k}P_\gamma\!\left(\frac{\xi}{2^jN}\right)
e^{-\rho^2/(2^{2j}N^2)},
\]
where $P_\gamma$ is a fixed polynomial satisfying
$$
\left|P_\gamma\!\left(\frac{\xi}{A}\right)\right|
\lesssim_\gamma
\left(1+\frac{|\xi|}{A}\right)^{|\gamma|}.
$$
Therefore
\begin{align*}
|\partial_\xi^\gamma m_N(\xi)|
&\lesssim_\gamma
(1-\vartheta)\sum_{j\ge 0}\vartheta^j(2^jN)^{-k}\Big(1+\frac{\rho}{2^jN}\Big)^ke^{-\rho^2/(2^{2j}N^2)}
\lesssim_{s,\gamma} N^{-k}m_N(\xi).
\end{align*}

Now assume $\rho\ge N$. Since $m_N$ is radial, we use the radial differentiation formula
\[
\partial_\xi^\gamma m_N(\xi)
=
\sum_{q=1}^{k}\rho^{-k}P_{\gamma,q}\!\left(\frac{\xi}{\rho}\right)
\big[(r\partial_r)^q m(r)\big]_{r=\rho/N},
\]
where $P_{\gamma,q}$ are smooth  functions on $\mathbb S^{d-1}$. Together with \eqref{derivative-m(r)}, we get
\[
|\partial_\xi^\gamma m_N(\xi)|
\lesssim_{s,\gamma}\rho^{-k}m_N(\xi)
\lesssim_{s,\gamma}(N+\rho)^{-k}m_N(\xi),
\]
which implies \eqref{equ:msymbol}.

\medskip
\underline{\emph{Proof of  \eqref{equ:masymp}  and \eqref{equcont}}.}
By scaling, it is sufficient to estimate $m(r)$.
For $r\leq1$, the upper bound of $m(r)\lesssim1$ is trivial.  Now, it remains to show the lower bound. For arbitrary $j\geq0$, it holds $4^{-j}r^2\leq1$. Thus, we have
\begin{align*}
    m(r)\geq(1-\vartheta)e^{-1}\sum_{j=0}^\infty\vartheta^j=e^{-1}\sim_s1,
\end{align*}
which implies the lower bound.

For $r\ge 1$, choose $k\in\mathbb Z$ such that
$
2^k\le r<2^{k+1}.
$
For $j\ge k+2$, 
notice that $e^{-4^{-j}r^2}\ge e^{-1/4}$, we obtain
\[
m(r)\ge (1-\vartheta)e^{-1/4}\sum_{j\ge k+2}\vartheta^j
=(1-\vartheta)e^{-1/4}\frac{\vartheta^{k+2}}{1-\vartheta}\gtrsim_s r^{-\alpha}.
\]
For the upper bound, splitting the sum into $j\ge k-1$ and $j\le k-2$. For the first part, we have
\[\sum_{j\geq k-1}\vartheta^je^{-4^{-j}r^2}
\leq\sum_{j\ge k-1}\vartheta^j\lesssim \vartheta^k.
\]
For $j=k-\ell\le k-2$ with $\ell\ge 2$, we have
$
4^{-j}r^2\ge 4^{\ell}.
$
Therefore, we get
\[
\vartheta^j e^{-4^{-j}r^2}
\le \vartheta^k\vartheta^{-\ell}e^{-4^\ell}
\lesssim_s \vartheta^k.
\]
 Summing over all $\ell\ge 2$ gives
$
m(r)\lesssim_s r^{-\alpha}.
$
This proves \eqref{equ:masymp} for $r\geq1$. Putting two parts together, we finish the proof of  \eqref{equ:masymp}.

 Moreover, \eqref{equ:masymp} implies the pointwise
bound
\[
\langle\xi\rangle^s\lesssim_s
\langle\xi\rangle m_N(\xi)
\lesssim_s N^{1-s}\langle\xi\rangle^s.
\]
Therefore, by using  Plancherel's theorem, we have  \eqref{equcont}.

\medskip
\underline{\emph{Proof of \eqref{equ:multcompare} and \eqref{equ:profileapprox}}}

Let $M\ge N$ and set
\[
\lambda=\frac MN\ge 1,\qquad 
R(r)=\frac{m(r/\lambda)}{m(r)},\qquad r>0.
\]
Then
\[
\frac{m_M(\xi)}{m_N(\xi)}=R\!\left(\frac{|\xi|}{N}\right).
\]
We claim that for every $q\ge 0$,
\begin{align}
|(r\partial_r)^q R(r)|\lesssim_{s,q}\lambda^{\alpha},\qquad r>0.
\label{claim}
\end{align}

To prove this estimate,  we first claim the following estimate
\begin{align}\label{new-2.28}
|(r\partial_r)^q m(r)^{-1}|\lesssim_{s,q}m(r)^{-1}.
\end{align}
Indeed, the Leibniz rule gives the following identity
$$\mathcal D^q(fg) = \sum_{a=0}^{q}\binom{q}{a}(\mathcal D^a f)(\mathcal D^{q-a}g),$$
where  $\mathcal D=r\partial_r$.
Next, we prove \eqref{new-2.28} by induction. For $q=0$, the desired estimate holds automatically. For $q=1$, using \eqref{derivative-m(r)}, we have
$
    |\mathcal D[m^{-1}]| = m^{-2}|\mathcal Dm| \lesssim_s m^{-2}\cdot m = m^{-1}. 
$
Now, we assume that for $0\le b\le q-1$, it holds
$$|\mathcal D^b[m^{-1}]| \lesssim_{s,b}\, m^{-1}.$$
We need to show that $|\mathcal D^q[m^{-1}]|\lesssim_{s,q} m^{-1}$. By definition, one has
$$0 = \mathcal D^q[m\cdot m^{-1}] = \sum_{a=0}^{q}\binom{q}{a}(\mathcal D^a m)(\mathcal D^{q-a}[m^{-1}]).$$
Notice that 
$$m\cdot \mathcal D^q[m^{-1}] = -\sum_{a=1}^{q}\binom{q}{a}(\mathcal D^a m)\bigl(\mathcal D^{q-a}[m^{-1}]\bigr)$$
and taking $b = q-a$, 
 we have the pointwise estimate
$$|\mathcal D^q[m^{-1}]| \le m^{-1}\sum_{b=0}^{q-1}\binom{q}{q-b}\bigl|\mathcal D^{q-b}m\bigr|\cdot\bigl| \mathcal D^b[m^{-1}]\bigr|.$$
Next, we consider the  terms in the summation. If $b=0$,  $|\mathcal D^b[m^{-1}]|=m^{-1}$. Since $q-b\ge1$, by \eqref{derivative-m(r)}, we get 
$$|\mathcal D^q m|\cdot|m^{-1}| \lesssim m\cdot m^{-1} = 1.$$
If $1\le b\le q-1$, we have $q-b\ge 1$, by \eqref{derivative-m(r)} we have $|\mathcal D^{q-b}m|\lesssim_{s,q-b} m$. For the term containing the $b$-th order derivative, it can be estimated by the induction assumption $|\mathcal D^b[m^{-1}]|\lesssim_{s,b} m^{-1}$. Hence we have
$$|\mathcal D^{q-b}m|\cdot|\mathcal D^b[m^{-1}]| \lesssim m\cdot m^{-1} = 1.$$
Incorporating the summation, we obtain the desired estimate \eqref{new-2.28}.

We now go back to show  \eqref{claim}. The Leibniz rule therefore provides that 
\[
|(r\partial_r)^q R(r)|
\lesssim_{s,q}
\sum_{a+b=q}
|(r\partial_r)^a m(r/\lambda)|\,|(r\partial_r)^b m(r)^{-1}|
\lesssim_{s,q}
\frac{m(r/\lambda)}{m(r)}
= R(r).
\]
The direct computation yields the pointwise bound
\begin{align*}
R(r)\lesssim_s
\begin{cases}
1, & 0<r\leq1,\\
r^\alpha, & 1\leq r\leq\lambda,\\
\lambda^\alpha, & r\geq\lambda.
\end{cases}
\end{align*}
Hence, we have proved the claim \eqref{claim}. 

To obtain \eqref{equ:multcompare}, by Lemma \ref{lem:Mikhlin}, it remains to  provide the derivative bound of $a(\xi):=R(\frac{|\xi|}{N})$. 
Applying the radial differentiation formula to $a(\xi)$, for $|\gamma|=k\ge 1$, it holds
\begin{align}\label{radial}
\partial_\xi^\gamma a(\xi)
=
\sum_{q=1}^{k}|\xi|^{-k}P_{\gamma,q}\!\left(\frac{\xi}{|\xi|}\right)
(r\partial_r)^qR(r)\big|_{r=\frac{|\xi|}{N}}.
\end{align}
Using \eqref{claim}, we obtain
\[
|\partial_\xi^\gamma a(\xi)|
\lesssim_{s,\gamma}\lambda^{\alpha}|\xi|^{-|\gamma|}.
\]
The symbol $a$ is smooth at the origin. Lemma \ref{lem:Mikhlin}  gives \eqref{equ:multcompare}.

Now, we turn to  prove \eqref{equ:profileapprox}. Fix $0\leq\sigma\leq2$.  Since
$
1-e^{-x}\lesssim_\sigma x^{\sigma/2}
$ with $x\geq0$,
we obtain
\begin{align}
0\leq1-m_N(\xi)
&=(1-\vartheta)\sum_{j\geq0}\vartheta^j
\left(1-e^{-4^{-j}|\xi|^2/N^2}\right)\nonumber\\
&\lesssim_{s,\sigma}
\left(\frac{|\xi|}{N}\right)^\sigma
(1-\vartheta)\sum_{j\geq0}\vartheta^j2^{-j\sigma}
\lesssim_{s,\sigma}
\left(\frac{|\xi|}{N}\right)^\sigma.
\label{equ:1minusm}
\end{align}

For every real $k\geq0$,
\begin{align*}
\|(1-I_N)f\|_{\dot H^k}^2
&=\int_{\R^d}|\xi|^{2k}|1-m_N(\xi)|^2|\widehat f(\xi)|^2\,d\xi\\
&\lesssim_{s,\sigma}N^{-2\sigma}
\int_{\R^d}|\xi|^{2(k+\sigma)}|\widehat f(\xi)|^2\,d\xi,
\end{align*}
which is corresponding to the estimate with the homogeneous norm.  Since
$|\xi|^\sigma\leq\langle\xi\rangle^\sigma$, the same pointwise bound also gives
\begin{align*}
\|(1-I_N)f\|_{H^k}^2
&=\int_{\R^d}\langle\xi\rangle^{2k}
|1-m_N(\xi)|^2|\widehat f(\xi)|^2\,d\xi
\lesssim_{s,\sigma}N^{-2\sigma}
\int_{\R^d}\langle\xi\rangle^{2(k+\sigma)}
|\widehat f(\xi)|^2\,d\xi.
\end{align*}
Therefore, we complete the proof of \eqref{equ:profileapprox}.

\medskip
\underline{\emph{Proof of \eqref{equ:IMSMdiff}--\eqref{equ:L2tailI}}}
Let $M\ge 1$, $\rho=|\xi|$, $r=\rho/M$.  By the direct computation, we have 
\begin{gather*}
[(I-I_M)f]^\wedge(\xi)=\frac{1-m_M(\rho)}{\langle\rho\rangle m_M(\rho)}\cdot \langle\rho\rangle m_M(\rho)\hat{f}(\xi),\\
[(I-S_M)f]^\wedge(\xi)=\frac{1-e^{-|\xi|^2/M^2}}{\langle\rho\rangle m_M(\rho)}\cdot \langle\rho\rangle m_M(\rho)\hat{f}(\xi).
\end{gather*}
Then the estimate \eqref{equ:IMSMdiff}
 is reduced to 
\begin{gather*}
\|(1-I_M)f\|_{L^p}
=
\|T_{1}(\langle\nabla\rangle I_M f)\|_{L^p},\\
\|(1-S_M)f\|_{L^p}
=
\|T_{2}(\langle\nabla\rangle I_M f)\|_{L^p},
\end{gather*}
where $T_1$ and $T_2$ are both Fourier multipliers with symbol 
$$b_M^1:=\frac{1-m_M(\rho)}{\langle\rho\rangle m_M(\rho)},\,\,b_M^2:=\frac{1-e^{-\rho^2/M^2}}{\langle\rho\rangle m_M(\rho)}$$ respectively. Now, it remains to  verify that $b_M^1$ and $b_M^2$ satisfy the condition in Lemma \ref{lem:Mikhlin}. For convenience, we also denote by 
\begin{align*}
    h_M^1(\rho)=1-m_M(\rho),\,\,h_M^2(\rho)=1-e^{-\rho^2/M^2}.
\end{align*}
We first establish that for both choices of $h_M^j$ with $j=1,2,$
\begin{align}\label{claim-2}
|(\rho\partial_\rho)^q h_M^j(\rho)|=|(r\partial_r)^qh^j(r)|
\lesssim_q \min(r^2,1),\qquad r>0,
\end{align}
where $h^1(r)=1-m(r)$ and $h^2(r)=1-e^{-r^2}$.
Under the change of variable $r=\rho/M$, it holds
$$h_M^1(\rho)=h^1(r),\qquad h_M^2(\rho)=h^2(r).$$
For $h^2(r)$, this follows from Taylor expansion when $r\le 1$ while the Gaussian is dominating the polynomial  when $r\ge 1$.

For $h^1(r)$, by definition, we deduce that 
\[
1-m(r)=(1-\vartheta)\sum_{j\ge 0}\vartheta^j(1-e^{-4^{-j}r^2}).
\]
For each $j$, we have
\[
|(r\partial_r)^q(1-e^{-4^{-j}r^2})|
\lesssim_q \min(4^{-j}r^2,1).
\]
If $r\le 1$, then $4^{-j}r^2\le 1$ for all $j\ge 0$, and hence
\[
|(r\partial_r)^q(1-m(r))|
\lesssim_q
(1-\vartheta)\sum_{j\ge 0}\vartheta^j4^{-j}r^2
\lesssim_s r^2.
\]
If $r\ge 1$, the same dyadic partition as in the proof of \eqref{equ:masymp} shows that the sum is bounded by a constant depending only on $s,q$. Thus \eqref{claim-2} holds. As a consequence, it holds
\[
\rho^k\bigl(|\partial_\xi^\gamma h_M^1(\xi)|
+|\partial_\xi^\gamma h_M^2(\xi)|\bigr)
\lesssim_{s,\gamma}
\min\left\{\frac{\rho^2}{M^2},1\right\}.
\]
Now, we are in the position to prove the $L^p$-boundedness of $T_1$ and $T_2$ respectively.
Recall that from \eqref{new-2.28}, we have
\[
|(\rho\partial_\rho)^c m_M(\rho)^{-1}|
\lesssim_{s,c} m_M(\rho)^{-1}.
\]
Also,
\[
|(\rho\partial_\rho)^b \langle\rho\rangle^{-1}|
\lesssim_b \langle\rho\rangle^{-1}.
\]
Using the Leibniz rule,
\[
|(\rho\partial_\rho)^q b_M(\rho)|
\lesssim_{q}
\sum_{a+b+c=q}
|(\rho\partial_\rho)^a h_M(\rho)|
|(\rho\partial_\rho)^b\langle\rho\rangle^{-1}|
|(\rho\partial_\rho)^c m_M(\rho)^{-1}|.
\]
Therefore
\[
|(\rho\partial_\rho)^q b_M(\rho)|
\lesssim_{s,q}
\frac{\min(r^2,1)}{\langle\rho\rangle m_M(\rho)}.
\]
We now show that the right-hand side is bounded by $C_{s,q}M^{-1}$.

If $r\ge 1$,  it then follows from \eqref{equ:masymp} that 
$m_M(\rho)\gtrsim_s r^{-\alpha}.$
Hence
\[
\frac{1}{\langle\rho\rangle m_M(\rho)}
\lesssim_s
\frac{r^{\alpha}}{\rho}
\le M^{-1}.
\]

If $r\le 1$, then $\min(r^2,1)=r^2$ and $m_M(\rho)\gtrsim_s 1$. If $\rho=Mr\le 1$, then
\[
\frac{r^2}{\langle\rho\rangle m_M(\rho)}
\lesssim_s r^2
=\frac{\rho^2}{M^2}
\le M^{-2}
\le M^{-1}.
\]
If $1\le \rho\le M$,  we have
\[
\frac{r^2}{\langle\rho\rangle m_M(\rho)}
\lesssim_s
\frac{r^2}{\rho}
=
\frac{\rho}{M^2}
\le M^{-1}.
\]
Thus in both cases, we obtain
\[
|(\rho\partial_\rho)^q b_M(\rho)|\lesssim_{s,q}M^{-1}.
\]

Since $b_M$ is radial, the radial differentiation formula \eqref{radial} gives
\[
|\partial_\xi^\gamma b_M(\xi)|
\lesssim_{s,\gamma}M^{-1}|\xi|^{-|\gamma|},
\qquad \xi\ne 0.
\]
The symbol is smooth at the origin. The Mikhlin multiplier theorem (Lemma \ref{lem:Mikhlin}) therefore proves
\[
\|(1-I_M)f\|_{L^p}
\lesssim_s M^{-1}\|\langle\nabla\rangle I_M f\|_{L^p},
\]
and using the same argument with $h_M^2$ gives
\[
\|(1-S_M)f\|_{L^p}
\lesssim_s M^{-1}\|\langle\nabla\rangle I_M f\|_{L^p}.
\] 
Therefore, we complete the proof of  \eqref{equ:IMSMdiff}.

It remains to verify the gradient multipliers.
For $\nabla I_M$, the multiplier relative to $\langle\nabla\rangle I_M$ is
\[
a_0(\xi)=\frac{i\xi}{\langle\xi\rangle}.
\]

For $\nabla S_M$, the multiplier relative to $\langle\nabla\rangle I_M$ is
\[
a_S(\xi)=\frac{i\xi\,e^{-|\xi|^2/M^2}}{\langle\xi\rangle m_M(\xi)}.
\]
Using the similar strategy as in the proof of \eqref{equ:IMSMdiff}, one can  prove \eqref{equ:heatDM}.  Taking $p=2$ in \eqref{equ:IMSMdiff} gives
\eqref{equ:L2tailI}.

\medskip
\underline{\emph{Proof of \eqref{equ:Idot} and \eqref{equ:scaio}}}

We observe that the direct differentiation of \eqref{equ:mheat} gives
\[
|N\partial_Nm_N(\xi)|\lesssim_s
\begin{cases}
(\frac{|\xi|}{N})^2,&|\xi|\leq N,\\
m_N(\xi),&|\xi|\geq N.
\end{cases}
\]
Consequently,
\[
\frac{|N\partial_Nm_N(\xi)|}
{\langle\xi\rangle m_N(\xi)}
\lesssim_s N^{-1}.
\]
Indeed, when $|\xi|\leq N$, this follows from $m_N\sim_s1$ and
$|\xi|^2/(N^2\langle\xi\rangle)\leq N^{-1}$, and when
$|\xi|\geq N$ from $\langle\xi\rangle^{-1}\leq N^{-1}$.
Plancherel's theorem now proves \eqref{equ:Idot}.  The scaling identity \eqref{equ:scaio} follows directly from the Fourier transform. 
\end{proof}

The following estimate is the main replacement for the paraproduct commutator estimate used in \cite{SunZheng}.  It uses the $C^{1,\min\{1,\frac4d\}}$ structure of the mass-critical nonlinearity.

\begin{lemma}[Nonlinear commutator]\label{commutor}
Let $J$ be a time interval, let $N\geq1$, and let $u$ be a smooth function on
$\bar J\times\R^d$.  Let $\mathcal{C}_N(u)$ be defined as in
\eqref{equ:commutator-def}. Then
\begin{align}
&\|\nabla\mathcal{C}_N(u)\|_{L_t^2L_x^{\frac{2d}{d+2}}(J\times\R^d)}\notag\\
\lesssim&_{d,s}N^{-\nu}
\|u\|_{L_t^\infty L_x^2(J\times\R^d)}^{\frac4d-\nu}
\|\langle\nabla\rangle I_Nu\|_{L_t^\infty L_x^2(J\times\R^d)}^{\nu}
\|\langle\nabla\rangle I_Nu\|_{L_t^2L_x^{\frac{2d}{d-2}}(J\times\R^d)},\label{equ:commgrad}\\
&\|\mathcal{C}_N(u)\|_{L_t^2L_x^{\frac{2d}{d+2}}(J\times\R^d)}\notag\\
\lesssim&_{d,s}N^{-(1+\nu)}
\|u\|_{L_t^\infty L_x^2(J\times\R^d)}^{\frac4d-\nu}
\|\langle\nabla\rangle I_Nu\|_{L_t^\infty L_x^2(J\times\R^d)}^{\nu}
\|\langle\nabla\rangle I_Nu\|_{L_t^2L_x^{\frac{2d}{d-2}}(J\times\R^d)}.\label{equ:commzero}
\end{align}
In dimension $d=3$, one also has
\begin{align}
&\|\mathcal{C}_N(u)\|_{L_t^1L_x^2(J\times\R^3)}
\lesssim_sN^{-2}
\|u\|_{L_t^\infty L_x^2(J\times\R^3)}^{\frac13}
\|\langle\nabla\rangle I_Nu\|_{L_t^2L_x^6(J\times\R^3)}^2.\label{equ:commL1L2}
\end{align}
For $d\geq4$,
\begin{align}
&\big\||\nabla|^{1-\frac4d}\mathcal{C}_N(u)\big\|_{L_t^2L_x^{\frac{2d}{d+2}}(J\times\R^d)}\notag\\
\lesssim&_{d,s}N^{-\frac8d}
\|\langle\nabla\rangle I_Nu\|_{L_t^\infty L_x^2(J\times\R^d)}^{\frac4d}
\|\langle\nabla\rangle I_Nu\|_{L_t^2L_x^{\frac{2d}{d-2}}(J\times\R^d)}.\label{equ:commfractional}
\end{align}
\end{lemma}

\begin{proof}
Without loss of generality, we prove the estimates for
$u\in C_c^\infty(J\times\R^d)$.  The general case follows by a standard
approximation and Fatou's lemma.  Throughout the proof $F$ is regarded as a
real map from $\R^2$ to $\R^2$.

Let $R(z,w)$ be the Taylor remainder 
\begin{align}\label{taylor-remainder}
 R(z,w)=F(z+w)-F(z)-w\cdot F^\prime(z).
\end{align}
Then by  \eqref{equ:Rpoint} and Lemma \ref{lem:nonltermest}, we obtain
\begin{equation}\label{equ:Rall}
 |R(z,w)|\lesssim_d(|z|+|w|)^{\frac4d-\nu}|w|^{1+\nu}.
\end{equation}

To treat the commutator, we need a further decomposition.
Set $N_j=2^jN$.
From \eqref{equ:Iheatdef}, we get
\begin{align}
\mathcal{C}_N(u)=&\sum_{j\geq0}c_j\{S_{N_j}F(u)-F(S_{N_j}u)\}\notag\\\notag
&+\left\{\sum_{j\geq0}c_jF(S_{N_j}u)-F(I_Nu)\right\}\\\label{equ:commAB}
=&: \mathcal A_N+\mathcal B_N,
\end{align}
where $c_j=(1-\vartheta)\vartheta^j$.

\textbf{Step 1: The estimate of $\mathcal A_N$.}For $M>0$, write the heat semigroup as
\begin{align}\label{convolution}
S_Mf(x)=\int_{\mathbb R^d}K_M(y)f(x-y)\,dy,
\end{align}
where $K_M$ is the positive Gaussian heat kernel satisfying
\begin{align}\label{heat-property}
\int_{\mathbb R^d}K_M(y)\,dy=1,\qquad \int_{\mathbb R^d}\nabla K_M(y)\,dy=0.
\end{align}
We  denote the commutator by
\[
 H_M=S_MF(u)-F(S_Mu).
\]
For fixed $(t,x)$, notice that $\int_{\R^d}K_M(y)dy=1$, then we have
\begin{equation}\label{equ:centeredheat}
 \int_{\R^d}K_M(y)\big[u(t,x-y)-S_Mu(x)\big]\,dy=0.
\end{equation}
Consequently,
\begin{align*}
\int_{\R^d}K_M(y)R(S_Mu,h_y)dy
={}&\int_{\R^d}K_M(y)F(u(t,x-y))dy-F(S_Mu)\\
&-\left[\int_{\R^d}K_M(y)h_y\,dy\right]\cdot F^\prime(S_Mu)\\
={}&S_MF(u)(t,x)-F(S_Mu(t,x)),
\end{align*}
where \begin{equation}\label{hy}
    h_y=u(t,x-y)-S_Mu(x).
\end{equation}
Hence, by using \eqref{taylor-remainder}, one has
\begin{equation}\label{equ:heatdefectall}
 H_M(t,x)=\int_{\R^d}K_M(y)R(S_Mu,h_y)dy.
\end{equation}
We next derive the differentiated identity directly from the commutator.
By \eqref{convolution},
we have
\[
 \nabla S_Mf(x)=\int_{\R^d}\nabla K_M(y)f(x-y)dy.
\]
Thus
\begin{align}\label{HM}
 \nabla H_M(t,x)
 ={}&\int_{\R^d}\nabla K_M(y)F(u(t,x-y))dy\\
 &-\left[\int_{\R^d}\nabla K_M(y)u(t,x-y)dy\right]\cdot F^\prime(S_Mu).
\end{align}
By \eqref{heat-property}, one has 
\begin{gather*}
\int_{\mathbb R^d}\nabla K_M(y)F(u(t,x-y))\,dy
=
\int_{\mathbb R^d}\nabla K_M(y)\bigl[F(u(t,x-y))-F(S_Mu)\bigr]\,dy,\\
\int_{\mathbb R^d}\nabla K_M(y)u(t,x-y)\,dy
=
\int_{\mathbb R^d}\nabla K_M(y)\bigl[u(t,x-y)-S_Mu\bigr]\,dy.
\end{gather*}Inserting the above identities into \eqref{HM}, we have
\begin{equation}\label{equ:heatdefectgradall}
 \nabla H_M(t,x)=\int_{\R^d}\nabla K_M(y)R(S_Mu,h_y)dy.
\end{equation}

We next claim a difference estimate.  For every $1<p<\infty$, $1\le q\le\infty$, and $M\ge N$, it holds
\begin{equation}\label{eq:hy-bound}
\|u(\cdot-y)-S_Mu\|_{L_t^qL_x^p}
\lesssim_{d,s}
N^{-(1-s)}M^{-s}(1+M|y|)
\|\langle\nabla\rangle I_Nu\|_{L_t^qL_x^p}.
\end{equation}
Indeed, writing
\[
u(x-y)-S_Mu(x)
=(1-S_M)u(x-y)+\bigl[S_Mu(x-y)-S_Mu(x)\bigr]
\]
and using Proposition~\ref{highcont}, we get
\[
\|(1-S_M)u\|_{L_t^qL_x^p}
\lesssim M^{-1}\|\langle\nabla\rangle I_Mu\|_{L_t^qL_x^p}.
\]
By the mean value theorem,
\[
\|S_Mu(\cdot-y)-S_Mu\|_{L_t^qL_x^p}
\lesssim |y|\,\|\nabla S_Mu\|_{L_t^qL_x^p}
\lesssim |y|\,\|\langle\nabla\rangle I_Mu\|_{L_t^qL_x^p}.
\]
Thus
$$
\|u(\cdot-y)-S_Mu\|_{L_t^qL_x^p}
\lesssim (M^{-1}+|y|)\|\langle\nabla\rangle I_Mu\|_{L_t^qL_x^p}.
$$
Since $M\ge N$, \eqref{equ:multcompare} gives
\[
\|\langle\nabla\rangle I_Mu\|_{L_t^qL_x^p}
\lesssim_s \left(\frac MN\right)^{1-s}
\|\langle\nabla\rangle I_Nu\|_{L_t^qL_x^p}.
\]
Therefore
$$
(M^{-1}+|y|)\left(\frac MN\right)^{1-s}
=M^{-1}\left(\frac MN\right)^{1-s}(1+M|y|)
=N^{-(1-s)}M^{-s}(1+M|y|),
$$
which proves \eqref{eq:hy-bound}.

By using H\"older's inequality, we deduce
\begin{align*}
\|R(S_Mu,h_y)\|_{L_t^2L_x^{\frac{2d}{d+2}}}
\lesssim_d{}&
 \|u\|_{L_t^\infty L_x^2}^{\frac4d-\nu}
 \|h_y\|_{L_t^\infty L_x^2}^{\nu}
 \|h_y\|_{L_t^2L_x^{\frac{2d}{d-2}}}.
\end{align*}
Here all spacetime norms are taken over $J\times\R^d$.
Using \eqref{eq:hy-bound}, we obtain
\begin{align}\label{equ:Rendpointall}
\|R(S_Mu,h_y)\|_{L_t^2L_x^{\frac{2d}{d+2}}}
\lesssim_{d,s}{}&
 \|u\|_{L_t^\infty L_x^2}^{\frac4d-\nu}
 \big[N^{-(1-s)}M^{-s}(1+M|y|)\big]^{1+\nu}\notag\\
&\times
 \|\langle\nabla\rangle I_Nu\|_{L_t^\infty L_x^2}^{\nu}
 \|\langle\nabla\rangle I_Nu\|_{L_t^2L_x^{\frac{2d}{d-2}}}.
\end{align}
On the other hand, the truncated heat kernel satisfies
\begin{equation}\label{equ:heatkernelmomentsall}
 \int K_M(y)(1+M|y|)^{1+\nu}dy\lesssim_d1,
 \qquad
 \int|\nabla K_M(y)|(1+M|y|)^{1+\nu}dy\lesssim_dM.
\end{equation}
Taking $M=N_j=2^jN$ in
\eqref{equ:heatdefectall}--\eqref{equ:heatdefectgradall} gives
\begin{align}
\|H_{N_j}\|_{L_t^2L_x^{\frac{2d}{d+2}}}
&\lesssim_{d,s}N^{-(1+\nu)}2^{-s(1+\nu)j}
 \|u\|_{L_t^\infty L_x^2}^{\frac4d-\nu}
 \|\langle\nabla\rangle I_Nu\|_{L_t^\infty L_x^2}^{\nu}
 \|\langle\nabla\rangle I_Nu\|_{L_t^2L_x^{\frac{2d}{d-2}}},
\label{equ:Hjendpointzero}\\
\|\nabla H_{N_j}\|_{L_t^2L_x^{\frac{2d}{d+2}}}
&\lesssim_{d,s}N^{-\nu}2^{[1-s(1+\nu)]j}
 \|u\|_{L_t^\infty L_x^2}^{\frac4d-\nu}
 \|\langle\nabla\rangle I_Nu\|_{L_t^\infty L_x^2}^{\nu}
 \|\langle\nabla\rangle I_Nu\|_{L_t^2L_x^{\frac{2d}{d-2}}}.
\label{equ:Hjendpointgrad}
\end{align}
Since $c_j=(1-\vartheta)\vartheta^j\lesssim_s2^{-(1-s)j}$, we have
\[
c_j2^{-s(1+\nu)j}
\lesssim_s 2^{-(1+s\nu)j},
\,\,
c_j2^{[1-s(1+\nu)]j}
\lesssim_s 2^{-s\nu j}.
\]
Therefore, both series are summable since $s>0$ and $\nu>0$. 
This proves the contribution of $\mathcal A_N$ to
\eqref{equ:commgrad}--\eqref{equ:commzero}.

\textbf{Step 2: The estimate of $\mathcal B_N$.}
The recursion \eqref{equ:Irecursion} gives
\begin{equation}\label{equ:Wrecall}
 I_{N_j}u=(1-\vartheta)S_{N_j}u+\vartheta I_{2N_j}u.
\end{equation}
Define
\begin{equation}\label{equ:Bjdefall}
 B_j=(1-\vartheta)F(S_{N_j}u)+\vartheta F(I_{2N_j}u)-F(I_{N_j}u).
\end{equation}
For every integer $L\geq0$,
\begin{align}\label{equ:Bfiniteall}
 \sum_{j=0}^L\vartheta^jB_j
 ={}&\sum_{j=0}^L(1-\vartheta)\vartheta^jF(S_{N_j}u)-F(I_Nu)
 +\vartheta^{L+1}F(I_{2N_L}u).
\end{align}
  Since $u$ is compactly
supported, $I_{2^{L+1}N}u$ is uniformly bounded in every mixed
Sobolev norm used below.  Hence the terminal term, and its first spatial
derivative when required, tend to zero after multiplication by
$\vartheta^{L+1}$.  Letting $L\to\infty$ gives
$
 \mathcal B_N=\sum_{j\geq0}\vartheta^jB_j.$
 Now, it remains to estimate $B_j$. Since
\begin{equation}\label{equ:XYZall}
S_{N_j}u=I_{N_j}u+\vartheta (S_{N_j}u-I_{2N_j}u),\qquad
 I_{2N_j}u=I_{N_j}u-(1-\vartheta)(S_{N_j}u-I_{2N_j}u),
\end{equation}
Taylor's formula and the cancellation of the first-order terms give
\begin{align*}
B_j
=&
\vartheta(1-\vartheta)
\int_0^1(S_{N_j}u-I_{2N_j}u)\cdot
\Bigl[
F^\prime\bigl(I_{N_j}u+\tau\vartheta (S_{N_j}u-I_{2N_j}u)\bigr)
\\
&-
F^\prime\bigl(I_{N_j}u-\tau(1-\vartheta)(S_{N_j}u-I_{2N_j}u)\bigr)
\Bigr]\,d\tau.
\end{align*}  Lemma \ref{lem:nonltermest} and
\eqref{equ:XYZall} therefore yield
\begin{equation}\label{equ:Bjpointall}
|B_j|\lesssim_d
(|S_{N_j}u|+|I_{2N_j}u|)^{\frac4d-\nu}
|(S_{N_j}-I_{2N_j})u|^{1+\nu}.
\end{equation}
Differentiating \eqref{equ:Bjdefall} directly and using
$\nabla I_{N_j}u=(1-\vartheta)\nabla S_{N_j}u+\vartheta\nabla I_{2N_j}u$, we obtain
\begin{align*}
 \nabla B_j={}&(1-\vartheta)\nabla S_{N_j}u\cdot[F^\prime(S_{N_j}u)-F^\prime(I_{N_j}u)]\\
 &+\vartheta\nabla I_{2N_j}u\cdot[F^\prime(I_{2N_j}u)-F^\prime(I_{N_j}u)].
\end{align*}
Thus \eqref{equ:DFholderall} and \eqref{equ:XYZall} give
\begin{equation}\label{equ:Bjgradpointall}
 |\nabla B_j|\lesssim_d
 (|S_{N_j}u|+|I_{2N_j}u|)^{\frac4d-\nu}|(S_{N_j}-I_{2N_j})u|^\nu
 (|\nabla S_{N_j}u|+|\nabla I_{2N_j}u|).
\end{equation}

Next, we aim to give the estimates of $(S_{N_j}-I_{2N_j})u,S_{N_j}u$ and $I_{2N_j}u$. First,
\eqref{equ:Irecursion} gives
\[
 (S_{N_j}-I_{2N_j})u
 =\vartheta^{-1}(S_{N_j}-I_{N_j})u.
\]
Consequently, by Proposition \ref{highcont}, for every $1<p<\infty$, it holds
\begin{align}
 \| (S_{N_j}-I_{2N_j})u\|_{L^p}
 &\lesssim_{d,s}N^{-1}2^{-sj}
 \|\langle\nabla\rangle I_Nu\|_{L^p},\label{equ:Zjzeroall}\\
 \|\nabla S_{N_j}u\|_{L^p}+\|\nabla I_{2N_j}u\|_{L^p}
 &\lesssim_{d,s}2^{(1-s)j}
 \|\langle\nabla\rangle I_Nu\|_{L^p}.\label{equ:XYgradall}
\end{align}
Moreover,
\begin{equation}\label{equ:XYcontractall}
 \|S_{N_j}u\|_{L_t^qL_x^p}+\|I_{2N_j}u\|_{L_t^qL_x^p}
 \lesssim\|u\|_{L_t^qL_x^p}.
\end{equation}
Putting \eqref{equ:Bjpointall}--\eqref{equ:XYcontractall} together and  using H\"older's inequality, we have
\begin{align}
\|B_j\|_{L_t^2L_x^{\frac{2d}{d+2}}}
&\lesssim_{d,s}N^{-(1+\nu)}2^{-s(1+\nu)j}
 \|u\|_{L_t^\infty L_x^2}^{\frac4d-\nu}
 \|\langle\nabla\rangle I_Nu\|_{L_t^\infty L_x^2}^{\nu}
 \|\langle\nabla\rangle I_Nu\|_{L_t^2L_x^{\frac{2d}{d-2}}},
\label{equ:Bjendpointzero}\\
\|\nabla B_j\|_{L_t^2L_x^{\frac{2d}{d+2}}}
&\lesssim_{d,s}N^{-\nu}2^{[(1-s)-s\nu]j}
 \|u\|_{L_t^\infty L_x^2}^{\frac4d-\nu}
 \|\langle\nabla\rangle I_Nu\|_{L_t^\infty L_x^2}^{\nu}
 \|\langle\nabla\rangle I_Nu\|_{L_t^2L_x^{\frac{2d}{d-2}}}.
\label{equ:Bjendpointgrad}
\end{align}
Since   both series are summable, we have proved
 the contribution of $\mathcal B_N$ to
\eqref{equ:commgrad}--\eqref{equ:commzero}.

It remains to prove \eqref{equ:commL1L2} and \eqref{equ:commfractional}.  First, we consider the case $d=3$. In this case, we have $F(z)=|z|^\frac43z\in C^{2,\frac13}$.  Then
\begin{equation}\label{equ:Rpointd3again}
 |R(x,y)|\lesssim(|x|^{\frac13}+|y|^{\frac13})|y|^2.
\end{equation}
For the fixed variable $y$, the H\"older inequality implies 
\[
\|R(S_Mu,h_y)\|_{L_x^2}
\lesssim
\||S_Mu|+|h_y|\|_{L_x^2}^{\frac13}
\|h_y\|_{L_x^6}^2,
\]
where $h_y$ is defined in \eqref{hy}. By invoking the $L^2$ contraction of the heat semigroup, one has
$$\|R(S_Mu,h_y)\|_{L^1_tL^2_x}\lesssim \|u\|^{\frac13}_{L^\infty_tL^2_x}\|h_y\|^2_{L^2_tL^6_x}.$$
Furthermore, by using \eqref{eq:hy-bound}
 and \eqref{equ:Rpointd3again}, we get 
\begin{align*}
\|H_{N_j}\|_{L_t^1L_x^2}
\lesssim_s{}&N^{-2}2^{-2sj}
 \|u\|_{L_t^\infty L_x^2}^{\frac13}
 \|\langle\nabla\rangle I_Nu\|_{L_t^2L_x^6}^2.
\end{align*}
Applying \eqref{equ:Bjpointall} with $d=3$ provides that
\begin{align*}
\|B_j\|_{L_t^1L_x^2}
\lesssim_s{}&N^{-2}2^{-2sj}
 \|u\|_{L_t^\infty L_x^2}^{\frac13}
 \|\langle\nabla\rangle I_Nu\|_{L_t^2L_x^6}^2.
\end{align*}
Summing over  $c_j$ and $\vartheta^j$, we have
\eqref{equ:commL1L2} in $d=3$.

Next, we assume that $d\geq4$ and taking   $\nu=\frac4d$.  Indeed, when $d=4$,
\eqref{equ:commfractional} degenerates to \eqref{equ:commzero}.  Suppose that $d\geq5$,
\eqref{equ:commzero} and
\eqref{equ:commgrad} give
$$
\|\mathcal{C}_N(u)\|_{L_t^2L_x^{\frac{2d}{d+2}}}
\lesssim_{d,s}
N^{-(1+\frac4d)}\|\langle\nabla\rangle I_Nu\|_{L_t^\infty L_x^2}^{\frac4d}
\|\langle\nabla\rangle I_Nu\|_{L_t^2L_x^{\frac{2d}{d-2}}},
$$
and
\[
\|\nabla\mathcal{C}_N(u)\|_{L_t^2L_x^{\frac{2d}{d+2}}}
\lesssim_{d,s}
N^{-\frac4d}\|\langle\nabla\rangle I_Nu\|_{L_t^\infty L_x^2}^{\frac4d}
\|\langle\nabla\rangle I_Nu\|_{L_t^2L_x^{\frac{2d}{d-2}}}.
\]
For $0\leq\sigma\leq1$, the interpolation implies that
\begin{align*}
&\big\||\nabla|^\sigma\mathcal{C}_N(u)\big\|_{L_t^2L_x^{\frac{2d}{d+2}}}\notag\\
&\qquad\lesssim_{d,s}N^{-(1+\frac4d-\sigma)}
 \|\langle\nabla\rangle I_Nu\|_{L_t^\infty L_x^2}^{\frac4d}
 \|\langle\nabla\rangle I_Nu\|_{L_t^2L_x^{\frac{2d}{d-2}}}.
\end{align*}
Taking $\sigma=1-\frac4d$ proves \eqref{equ:commfractional} and concludes the
proof.
\end{proof}

\begin{remark}\label{remark1}
Lemma \ref{commutor} is proved for a multiplier with time-independent symbol.  However, since all the estimates are pointwise in the frequency parameter, the same estimates hold uniformly on any local interval such that $N\sim N(t)$.
\end{remark}

\begin{remark}\label{rem:commutatornorms}
We record how the three commutator bounds will be used below.  The
 estimate \eqref{equ:commzero} will be used in the control of the momentum increment (Lemma \ref{lem:conincre}), the localized $L^2$ flux, and the  strong $L^2$ convergence.
The estimate \eqref{equ:commL1L2} is used only for controlling the potential-energy 
in Lemma \ref{lem:conincre}.  
\end{remark}

\subsection{Strichartz estimates and local well-posedness}
A pair $(q,r)$ is called admissible in dimension $d$ if $q,r\geq2$ and
\begin{equation}\label{def1}
\frac2q=d\left(\frac12-\frac1r\right),
\end{equation}
with the usual exclusion of $(q,r,d)=(2,\infty,2)$.  For a time interval $I$, we define
\[
\|u\|_{S^0(I)}=\sup_{(q,r)\text{ admissible}}\|u\|_{L_t^qL_x^r(I\times\R^d)},
\]
and denote the dual Strichartz space by $N^0(I)$.

\begin{lemma}[Strichartz estimate, \cite{GiV85b,KeT98}]\label{lem22}
Let $u$ solve $iu_t+\Delta u+h=0$ on $I\times\R^d$.  Then, for $\sigma\geq0$ and $t_0\in I$,
\begin{equation}\label{dispers}
\||\nabla|^\sigma u\|_{S^0(I)}\lesssim_d\||\nabla|^\sigma u(t_0)\|_{L^2}+\||\nabla|^\sigma h\|_{N^0(I)}.
\end{equation}
\end{lemma}

Throughout this subsection, mass conservation fixes $M(u(t))=M(u_0)$.  Whenever this fixed quantity is absorbed into an implicit constant, we write the dependence explicitly as $\lesssim_{d,M(u)}$ or $\lesssim_{d,s,M(u)}$.

\begin{lemma}[$H^s$-LWP]\label{lem:hslwp}
Let $d\geq3$, $0<s<1$, let $u_0\in H^s(\R^d)$ satisfy $\|u_0\|_{L^2}\leq M_*$, and set
\begin{equation}\label{equ:tlwp}
T_{\rm LWP}=c_{d,s}\|u_0\|_{H^s}^{-2/s}.
\end{equation}
For $c_{d,s}=c(d,s,M_*)$ sufficiently small, the solution to \eqref{equ:nls} exists on $[0,T_{\rm LWP}]$ and satisfies
\begin{align}
\|u\|_{S^0([0,T_{\rm LWP}])}&\lesssim_d\|u_0\|_{L^2},\label{equ:uest}\\
\|\langle\nabla\rangle^su\|_{S^0([0,T_{\rm LWP}])}&\lesssim_{d,s}\|u_0\|_{H^s},\label{equ:uest123}\\
\|u\|_{L_{t,x}^{\frac{2(d+2)}d}([0,T_{\rm LWP}]\times\R^d)}&\leq\delta_{d,s},\label{equ:smallcritical}
\end{align}
where $\delta_{d,s}>0$ can be made arbitrarily small by decreasing $c_{d,s}$.
\end{lemma}

\begin{proof}
The first two estimates are the standard subcritical local theory.  The same argument also gives, for every $0\leq\sigma\leq s$,
\begin{equation}\label{equ:intermediatelwp}
\|\langle\nabla\rangle^\sigma u\|_{S^0([0,T_{\rm LWP}])}
\lesssim_{d,s}\|u_0\|_{H^\sigma}.
\end{equation}
Since
\[
 (q,r)=\left(\frac{4(d+2)}{2d-s(d+2)},
 \frac{2d(d+2)}{d^2+s(d+2)}\right),
\]
is the admissible pair,
the Sobolev embedding
\[
 W^{s/2,\,\frac{2d(d+2)}{d^2+s(d+2)}}(\R^d)
 \hookrightarrow L^{\frac{2(d+2)}d}(\R^d)
\]
together with the H\"older inequality provides that
\begin{align*}
\|u\|_{L_{t,x}^{\frac{2(d+2)}d}([0,T_{\rm LWP}]\times\R^d)}
&\lesssim_d T_{\rm LWP}^{s/4}
\|\langle\nabla\rangle^{s/2}u\|_{L_t^{\frac{4(d+2)}{2d-s(d+2)}}
L_x^{\frac{2d(d+2)}{d^2+s(d+2)}}}\lesssim_{d,s}T_{\rm LWP}^{s/4}\|u_0\|_{H^{s/2}}.
\end{align*}
By interpolation and the definition of $T_{\rm LWP}$,
\[
\|u\|_{L_{t,x}^{\frac{2(d+2)}d}([0,T_{\rm LWP}]\times\R^d)}
\lesssim_{d,s}c_{d,s}^{s/4}M_*^{\frac{1}{2}},
\]
which proves \eqref{equ:smallcritical} after choosing $c_{d,s}$ sufficiently small.
\end{proof}

The following corollary provides the control of the Strichartz norm for the modified solution $Iu$ on the same local interval as obtained in Lemma \ref{lem:hslwp}. Our commutator estimates improve on those in \cite{VZ07}, allowing us to obtain a stronger local theory than that in \cite{SunZheng}.

\begin{corollary}\label{cor:modlwp}
Let $J=[t_0,t_1]$ be an interval given by Lemma \ref{lem:hslwp}.  Suppose
that $N\geq1$, and assume
\begin{equation}\label{equ:smallZN}
 N^{-1}\|\langle\nabla\rangle I_Nu(t_0)\|_{L^2}\leq c_{d,s}.
\end{equation}
Then
\begin{equation}\label{equ:mlwpes}
 \|\langle\nabla\rangle I_Nu\|_{S^0(J)}
 \lesssim_{d,s,M(u)}\|\langle\nabla\rangle I_Nu(t_0)\|_{L^2}.
\end{equation}
\end{corollary}

\begin{proof}
For convenience, we let $w=I_Nu$.  Then, we can write
\begin{equation}\label{equ:weq}
 iw_t+\Delta w+F(w)=-\mathcal{C}_N(u).
\end{equation}
By \eqref{equ:Icontract} and \eqref{equ:smallcritical},
\[
 \|w\|_{L_{t,x}^{\frac{2(d+2)}d}(J\times\R^d)}\leq\delta_{d,s}.
\]
The  chain rule therefore gives
\begin{align*}
\|\langle\nabla\rangle F(w)\|_{L_{t,x}^{\frac{2(d+2)}{d+4}}(J\times\R^d)}\lesssim_d
 \|w\|_{L_{t,x}^{\frac{2(d+2)}d}(J\times\R^d)}^{\frac4d}
 \|\langle\nabla\rangle w\|_{L_{t,x}^{\frac{2(d+2)}d}(J\times\R^d)}.
\end{align*}
On the other hand, Lemma \ref{commutor} and mass conservation give
\begin{align*}
&\|\mathcal{C}_N(u)\|_{L_t^2L_x^{\frac{2d}{d+2}}(J\times\R^d)}
 +\|\nabla\mathcal{C}_N(u)\|_{L_t^2L_x^{\frac{2d}{d+2}}(J\times\R^d)}\\
\lesssim&_{d,s,M(u)}
 (N^{-\nu}+N^{-(1+\nu)})
 \|\langle\nabla\rangle I_Nu\|_{S^0(J)}^{1+\nu}.
\end{align*}
Both displayed spacetime norms belong to $N^0(J)$.  Applying Lemma
\ref{lem22} to \eqref{equ:weq}, we obtain
\begin{align*}
\|\langle\nabla\rangle I_Nu\|_{S^0(J)}
\lesssim{}&\|\langle\nabla\rangle I_Nu(t_0)\|_{L^2}
 +\delta_{d,s}^{\frac4d}\|\langle\nabla\rangle I_Nu\|_{S^0(J)}\\
&+(N^{-\nu}+N^{-(1+\nu)})
 \|\langle\nabla\rangle I_Nu\|_{S^0(J)}^{1+\nu}.
\end{align*}
Choose first $\delta_{d,s}$ sufficiently small.  On a continuity bootstrap
where the left-hand side is a fixed multiple of its initial value, the last
line is bounded by that left-hand side times
\[
 C_{d,s,M(u)}
 \left(N^{-1}\|\langle\nabla\rangle I_Nu(t_0)\|_{L^2}\right)^\nu.
\]
It is therefore absorbed after decreasing $c_{d,s}$ in
\eqref{equ:smallZN}.  This proves \eqref{equ:mlwpes}.
\end{proof}

Let $\{t_k\}_{k_0\leq k\leq k^+}$ be the doubling times defined in \eqref{equ:ltk}.  We cover $[t_k,t_{k+1}]$ by standard local intervals $[\tau_k^j,\tau_k^{j+1}]$ of Lemma \ref{lem:hslwp}. Let $J_k$ denote the number of intervals in this covering. Then we have:

\begin{lemma}\label{lem:numberintervals}
Let $T^\ast\geq t_{k+1}$ and set $N_\ast=N(T^\ast)$.  Then
\begin{equation}\label{equ:jk}
J_k\lesssim_{d,s} k,
\end{equation}
and on every $[\tau_k^j,\tau_k^{j+1}]$,
\begin{equation}\label{equ:dius}
\|\langle\nabla\rangle I_{N_\ast}u\|_{S^0([\tau_k^j,\tau_k^{j+1}])}
\lesssim_{d,s,M(u)}\left(\frac{N_\ast}{N(t_k)}\right)^{1-s}\frac1{\lambda(t_k)}.
\end{equation}
\end{lemma}

\begin{proof}
From \eqref{equcont}, the almost monotonicity \eqref{equ:almons} and the geometrical decomposition give
\[
\|u(t)\|_{H^s}\sim\lambda(t_k)^{-s}
\]
for $t\in[t_k,t_{k+1}]$. By Lemma \ref{lem:hslwp}, the standard local interval   has length
\[
T_{LWP}(\tau_k^j)=c_{d,s}\|u(\tau_k^j)\|_{H^s}^{-2/s}
\sim_{d,s}\lambda(\tau_k^j)^2\sim_{d,s}\lambda(t_k)^2.
\]
 The bootstrap assumption
\eqref{equ:dotime} gives $t_{k+1}-t_k\le k\lambda(t_k)^2$, so the number $J_k$ of local
intervals needed to cover $[t_k,t_{k+1}]$ satisfies
\[
J_k\lesssim_{d,s}\frac{t_{k+1}-t_k}{\lambda(t_k)^2}+1\lesssim k,
\]
which is \eqref{equ:jk}.
We claim that
\begin{equation}\label{eq:ref}
\|\langle\nabla\rangle I_{N(\tau_k^j)}u(\tau_k^j)\|_{L^2}
\lesssim_{M(u)}\lambda(\tau_k^j)^{-1}. 
\end{equation} Indeed, it follows from the geometrical decomposition
\[
I_{N(\tau)}u(\tau,x)
=\lambda(\tau)^{-\frac d2}\big[I_{N(\tau)\lambda(\tau)}(Q_b(\tau)+\varepsilon(\tau))\big]
\Big(\frac{x-x(\tau)}{\lambda(\tau)}\Big)e^{-i\gamma(\tau)},
\]
and bootstrap assumption \eqref{equ:vartass} and Proposition \ref{pro:self-similar}. Notice that by the definition of $N(t)$ and almost monotonicity, we have $N(t)\sim N(\tau_k^j)$ and $\lambda(t_k)\sim\lambda(\tau_k^j)$.
By \eqref{equ:multcompare} and the bootstrap bound,
\[
\|\langle\nabla\rangle I_{N_\ast}u(\tau_k^j)\|_{L^2}
\lesssim_{d,s}\left(\frac{N_\ast}{N(t_k)}\right)^{1-s}\lambda(t_k)^{-1}.
\]
Moreover, using $N(t)=\lambda(t)^{-\frac{1}{\beta}}$, we obtain
\begin{align*}
N_\ast^{-1}\left(\frac{N_\ast}{N(t_k)}\right)^{1-s}\lambda(t_k)^{-1}
&=N_\ast^{-s}N(t_k)^{\beta-(1-s)}\\
&=N_\ast^{-s}N(t_k)^{\frac{s}{3}}\leq N_\ast^{-\frac{2s}{3}}.
\end{align*}
This together with Corollary \ref{cor:modlwp} implies \eqref{equ:dius}.
\end{proof}

We shall also need  the commutator estimate with the dynamical frequency scale. On each local interval, $N(t)$ is comparable to $N(t_k)$.

\begin{lemma}\label{lem:xianchafa}
Let $[\tau_k^j,\tau_k^{j+1}]$ be one of the above local intervals, and let
$\mathcal{C}_N$ be defined as in \eqref{equ:commutator-def}. Then
\begin{align}
\|\mathcal{C}_{N(t_k)}(u)\|_{L_t^2L_x^{\frac{2d}{d+2}}([\tau_k^j,\tau_k^{j+1}]\times\R^d)}&\lesssim_{d,s,M(u)}[N(t_k)\lambda(t_k)]^{-(1+\nu)},
\label{equ:xianchafa}\\
\|\nabla\mathcal{C}_{N(t_k)}(u)\|_{L_t^2L_x^{\frac{2d}{d+2}}([\tau_k^j,\tau_k^{j+1}]\times\R^d)}
&\lesssim_{d,s,M(u)}\lambda(t_k)^{-1}
 [N(t_k)\lambda(t_k)]^{-\nu}.
\label{equ:xianchafagrad}
\end{align}
\end{lemma}

\begin{proof}
Corollary \ref{cor:modlwp} at the current frequency gives
\[
 \|\langle\nabla\rangle I_{N(t_k)}u\|_{S^0([\tau_k^j,\tau_k^{j+1}])}
 \lesssim_{d,s,M(u)}\lambda(t_k)^{-1}.
\]
The result follows directly from \eqref{equ:commgrad} and
\eqref{equ:commzero}, with the conserved mass absorbed into the implicit
constant.
\end{proof}
\begin{remark}\label{remark2}
The estimates of Lemma \ref{lem:xianchafa} remain valid with $N(t_k)$
replaced by $N(t)$ on $[\tau_k^j,\tau_k^{j+1}]$, since
$N(t)\sim N(t_k)$ there.
\end{remark}

\subsection{Almost conservation law}
In this subsection, we prove the almost conservation estimates with the
correction $\Xi(t)$. We first estimate the increments of the modified
energy and momentum on a local interval.
\begin{lemma}[Energy and momentum increment]\label{lem:conincre}
Let $[\tau_k^j,\tau_k^{j+1}]$ be a local interval and let
$T^\ast\geq\tau_k^{j+1}$. Set $N=N(T^\ast)$. Then
\begin{align}
|E(I_Nu(\tau_k^{j+1}))-E(I_Nu(\tau_k^j))|
&\lesssim_{d,s,M(u)}N^{-\nu}
 \|\langle\nabla\rangle I_Nu\|_{S^0([\tau_k^j,\tau_k^{j+1}])}^{2+\nu},
\label{equ:modenerlwp}\\
|P(I_Nu(\tau_k^{j+1}))-P(I_Nu(\tau_k^j))|
&\lesssim_{d,s,M(u)}N^{-(1+\nu)}
 \|\langle\nabla\rangle I_Nu\|_{S^0([\tau_k^j,\tau_k^{j+1}])}^{2+\nu}.
\label{equ:modenerlwp123}
\end{align}
\end{lemma}

\begin{proof}
Let $I=[\tau_k^j,\tau_k^{j+1}]$ and $w=I_Nu$. Rewrite
\[
 i\partial_tw+\Delta w+F(w)=-\mathcal{C}_N(u),
\]
where 
$\mathcal{C}_N(u)$ is defined as in \eqref{equ:commutator-def}.
Differentiating the modified energy
\[
 E(w)=\frac12\|\nabla w\|_{L^2}^2
 -\frac{d}{2(d+2)}\|w\|_{L^{2+\frac4d}}^{2+\frac4d},
\]
we have
\begin{align*}
E(w(\tau_k^{j+1}))-E(w(\tau_k^j))
&=\int_I\frac{d}{dt}E(w(t))\,dt\\
&=\operatorname{Re}\int_I\int_{\R^d}
 \overline{\partial_tw}\,[-\Delta w-F(w)]\,dx\,dt\\
&=\operatorname{Re}\int_I\int_{\R^d}
 \overline{\partial_tw}\,\mathcal{C}_N(u)\,dx\,dt\\
&=\operatorname{Im}\int_I\int_{\R^d}
 \big[-\overline{\nabla w}\cdot\nabla\mathcal{C}_N(u)
       +\overline{F(w)}\,\mathcal{C}_N(u)\big]\,dx\,dt.
\end{align*}
Thus
\begin{align}\label{equ:energydiffall}
\left|\frac d{dt}E(w)\right|
\leq{}&\int_{\R^d}|\nabla w|\,|\nabla\mathcal{C}_N(u)|dx
 +\int_{\R^d}|F(w)|\,|\mathcal{C}_N(u)|dx.
\end{align}

We estimate the two terms in the energy increment separately for $d=3$
and $d\geq4$. For $d=3$, H\"older's inequality gives
\begin{align}
\int_I\int_{\R^3}|\nabla w|\,|\nabla\mathcal{C}_N(u)|dxdt
&\leq \|\nabla w\|_{L_t^2L_x^6(I)}
 \|\nabla\mathcal C_N(u)\|_{L_t^2L_x^{6/5}(I)}
\notag\\
&\lesssim_{s,M(u)}N^{-1}
 \|\langle\nabla\rangle I_Nu\|_{S^0(I)}^3,
\label{equ:kineticincrementd3}\\
\int_I\int_{\R^3}|F(w)|\,|\mathcal{C}_N(u)|dxdt
&\leq \|F(w)\|_{L_t^\infty L_x^2(I)}
 \|\mathcal C_N(u)\|_{L_t^1L_x^2(I)}
\notag\\
&\lesssim_{s,M(u)}N^{-2}
 \|\langle\nabla\rangle I_Nu\|_{S^0(I)}^4.
\label{equ:potentialincrementd3}
\end{align}
Here we used \eqref{equ:commgrad}, \eqref{equ:commL1L2}, mass
conservation, and the Gagliardo--Nirenberg inequality.  Moreover, Lemma
\ref{lem:numberintervals}, \eqref{equ:multcompare}, and \eqref{equ:almons}
give
\begin{equation}\label{equ:ZNsmallincrementd3}
 \frac{\|\langle\nabla\rangle I_Nu\|_{S^0(I)}}N
 \lesssim_{s,M(u)}N^{-s}N(t_k)^{\frac{s}{3}}
 \lesssim_{s,M(u)}N^{-\frac{2s}{3}}\ll1.
\end{equation}
Thus \eqref{equ:potentialincrementd3} is bounded by
\eqref{equ:kineticincrementd3}, which proves \eqref{equ:modenerlwp} for
$d=3$.

The increment of the modified momentum may be estimated in a similar way.  Integration by parts and translation invariance give
\[
 \frac d{dt}P(w)
 =2\Re\int_{\R^3}\nabla\mathcal C_N(u)\,\overline w\,dx
 =-2\Re\int_{\R^3}\mathcal C_N(u)\,\nabla\overline w\,dx.
\]
It follows from H\"older's inequality and \eqref{equ:commzero} that
\begin{align*}
|P(w(\tau_k^{j+1}))-P(w(\tau_k^j))|
&\lesssim
 \|\mathcal C_N(u)\|_{L_t^2L_x^{6/5}(I)}
 \|\nabla w\|_{L_t^2L_x^6(I)}\\
&\lesssim_{s,M(u)}N^{-2}
 \|\langle\nabla\rangle I_Nu\|_{S^0(I)}^3.
\end{align*}
This proves \eqref{equ:modenerlwp123} for $d=3$.

We now turn to the case $d\geq4$.  By using the H\"older inequality, we have
\begin{align}
\int_I\int_{\R^d}|\nabla w|\,|\nabla\mathcal{C}_N(u)|dxdt
&\leq \|\nabla w\|_{L_t^2L_x^{\frac{2d}{d-2}}(I)}
 \|\nabla\mathcal C_N(u)\|_{L_t^2L_x^{\frac{2d}{d+2}}(I)}
\notag\\
&\lesssim_{d,s,M(u)}N^{-\frac4d}
 \|\langle\nabla\rangle I_Nu\|_{S^0(I)}^{2+\frac4d},
\label{equ:kineticincrementhigh}\\
\left|\int_I\int_{\R^d}F(w)\overline{\mathcal{C}_N(u)}dxdt\right|
&\leq
 \big\||\nabla|^{-(1-\frac4d)}F(w)\big\|_{L_t^2L_x^{\frac{2d}{d-2}}}
\notag\\
&\quad\times
 \big\||\nabla|^{1-\frac4d}\mathcal C_N(u)\big\|_{L_t^2L_x^{\frac{2d}{d+2}}}
\notag\\
&\lesssim_{d,s}N^{-\frac8d}
 \|\langle\nabla\rangle I_Nu\|_{S^0(I)}^{2+\frac8d}.
\label{equ:potentialincrementhigh}
\end{align}
In the above inequality, we use \eqref{equ:commgrad}, \eqref{equ:commfractional}, the
Hardy--Littlewood--Sobolev inequality, with the identity operator for
$d=4$, and Sobolev embedding.  Moreover, Lemma
\ref{lem:numberintervals}, \eqref{equ:multcompare}, and \eqref{equ:almons}
give
\begin{equation}\label{equ:ZNsmallincrementhigh}
 \frac{\|\langle\nabla\rangle I_Nu\|_{S^0(I)}}N
 \lesssim_{d,s,M(u)}N^{-s}N(t_k)^{\frac{s}{3}}
 \lesssim_{d,s,M(u)}N^{-\frac{2s}{3}}\ll1.
\end{equation}
It follows that \eqref{equ:potentialincrementhigh} is bounded by
\eqref{equ:kineticincrementhigh}, which proves \eqref{equ:modenerlwp} for
$d\geq4$.

The increment of the modified momentum may be calculated using the
equation.  Integration by parts and translation invariance give
\[
 \frac d{dt}P(w)
 =2\Re\int_{\R^d}\nabla\mathcal C_N(u)\,\overline w\,dx
 =-2\Re\int_{\R^d}\mathcal C_N(u)\,\nabla\overline w\,dx.
\]
It follows from H\"older's inequality and \eqref{equ:commzero} that
\begin{align*}
|P(w(\tau_k^{j+1}))-P(w(\tau_k^j))|
&\lesssim
 \|\mathcal C_N(u)\|_{L_t^2L_x^{\frac{2d}{d+2}}(I)}
 \|\nabla w\|_{L_t^2L_x^{\frac{2d}{d-2}}(I)}\\
&\lesssim_{d,s,M(u)}N^{-(1+\frac4d)}
 \|\langle\nabla\rangle I_Nu\|_{S^0(I)}^{2+\frac4d},
\end{align*}
which proves \eqref{equ:modenerlwp123} for $d\geq4$.
\end{proof}

We now estimate the modified energy and momentum of the initial data.
Recall that $u_0=G(0)+H(0)$, where $G(0)\in H^1$ and $H(0)\in H^s$.

\begin{lemma}\label{lem:einit}
Let
\begin{equation}\label{equ:thetadef}
\Xi(t)=\frac{\lambda(t)^2}{2}
\int\left(|\nabla G(0)|^2-|\nabla I_{N(t)}G(0)|^2\right)dx.
\end{equation}
Then, for $t\in[0,T^+]$,
\begin{align}
\left|E(I_{N(t)}u(0))+\frac{\Xi(t)}{\lambda(t)^2}\right|
&\lesssim_{d,s}N(t)^{2(1-s)}
+\lambda(t)^{-2}(N(t)\lambda(t))^{-\frac{2}{d+2}},
\label{equ:eint0}\\
|P(I_{N(t)}u(0))|
&\lesssim_{d,s}N(t)^{1-s}
+\lambda(t)^{-1}(N(t)\lambda(t))^{-\frac12}
+|P(G(0))|.
\label{equ:pu0}
\end{align}
\end{lemma}

\begin{proof}
Writing $u(0)=G(0)+H(0)$, we have
\begin{align}
\left|E(I_{N(t)}u(0))+\frac{\Xi(t)}{\lambda(t)^2}\right|
&\leq
\left|E(I_{N(t)}G(0))+\frac{\Xi(t)}{\lambda(t)^2}\right|
\label{target-1}\\
&\quad+
\left|E(I_{N(t)}(G(0)+H(0)))-E(I_{N(t)}G(0))\right|.
\label{target-2}
\end{align}
By the definition of $\Xi$,
\[
E(I_{N(t)}G(0))+\frac{\Xi(t)}{\lambda(t)^2}
=E(G(0))+\frac{d}{2(d+2)}
\int\left(|G(0)|^{2+\frac4d}
-|I_{N(t)}G(0)|^{2+\frac4d}\right)dx.
\]
By \eqref{equ:profileapprox} and Sobolev embedding,
\[
\|(1-I_{N(t)})G(0)\|_{\dot H^{\frac{d}{d+2}}}
\lesssim_{d,s}N(t)^{-\frac{2}{d+2}}
\|G(0)\|_{\dot H^1}.
\]
The geometrical decomposition and interpolation give
\[
\|G(0)\|_{\dot H^1}\lesssim\lambda(0)^{-1},
\qquad
\|G(0)\|_{\dot H^{\frac d{d+2}}}
\lesssim\lambda(0)^{-\frac d{d+2}}.
\]
Since $I_{N(t)}$ is an $L^{2+4/d}$ contraction, it follows that
\begin{align*}
&\left|\|G(0)\|_{L^{2+\frac4d}}^{2+\frac4d}
-\|I_{N(t)}G(0)\|_{L^{2+\frac4d}}^{2+\frac4d}\right|\\
&\quad\lesssim_d
\|(1-I_{N(t)})G(0)\|_{L^{2+\frac4d}}
\left(
\|G(0)\|_{L^{2+\frac4d}}^{1+\frac4d}
+\|I_{N(t)}G(0)\|_{L^{2+\frac4d}}^{1+\frac4d}
\right)\\
&\quad\lesssim_{d,s}
N(t)^{-\frac2{d+2}}
\|G(0)\|_{\dot H^1}
\|G(0)\|_{\dot H^{\frac d{d+2}}}^{1+\frac4d}\\
&\quad\lesssim_{d,s}
N(t)^{-\frac2{d+2}}
\lambda(0)^{-\frac{2(d+3)}{d+2}}
\lesssim_{d,s}
\lambda(t)^{-2}(N(t)\lambda(t))^{-\frac2{d+2}}.
\end{align*}
This together with \eqref{equ:eg0} implies
\[
\left|E(I_{N(t)}G(0))+\frac{\Xi(t)}{\lambda(t)^2}\right|
\lesssim_d
\lambda(t)^{-2}(N(t)\lambda(t))^{-\frac{2}{d+2}},
\]
which proves the required bound for \eqref{target-1}.

We next estimate \eqref{target-2}.
By \eqref{equ:h0hs} and Proposition \ref{highcont},
\[
\|I_{N(t)}H(0)\|_{L^2}\leq\lambda(0)^{10},
\qquad
\|\nabla I_{N(t)}H(0)\|_{L^2}
\lesssim_s N(t)^{1-s}\lambda(0)^{10}.
\]
The difference estimate used in \cite[(2.42)]{SunZheng} and the
Gagliardo--Nirenberg inequality give
\begin{align*}
&|E(I_{N(t)}(G(0)+H(0)))-E(I_{N(t)}G(0))|\\
\lesssim&_d
\|\nabla I_{N(t)}G(0)\|_2
\|\nabla I_{N(t)}H(0)\|_2
+\|\nabla I_{N(t)}H(0)\|_2^2\\
&+
\|I_{N(t)}G(0)\|_{L^{2+\frac4d}}^{1+\frac4d}
\|I_{N(t)}H(0)\|_{L^{2+\frac4d}}
+\|I_{N(t)}H(0)\|_{L^{2+\frac4d}}^{2+\frac4d}
\lesssim_{d,s}N(t)^{2(1-s)}.
\end{align*}
This proves \eqref{equ:eint0}.

For the momentum,
\begin{align}\label{target-3}
|P(I_{N(t)}u(0))|
\leq{}&
|P(I_{N(t)}G(0)+I_{N(t)}H(0))-P(I_{N(t)}G(0))|
\notag\\
&+|P(I_{N(t)}G(0))-P(G(0))|+|P(G(0))|.
\end{align}
As in \cite[proof of (2.35)]{SunZheng}, for
$f,g\in\dot H^{1/2}(\R^d)$,
\[
|P(f+g)-P(f)|
\lesssim
\|g\|_{\dot H^{1/2}}
\left(\|f\|_{\dot H^{1/2}}+\|g\|_{\dot H^{1/2}}\right).
\]
By interpolation and \eqref{equ:profileapprox},
\begin{align*}
\|I_{N(t)}H(0)\|_{\dot H^{1/2}}
&\lesssim_s
\lambda(0)^{10}N(t)^{\frac{1-s}{2}},\\
\|I_{N(t)}G(0)\|_{\dot H^{1/2}}
&\lesssim\lambda(0)^{-\frac12},\\
\|(1-I_{N(t)})G(0)\|_{\dot H^{1/2}}
&\lesssim_s N(t)^{-1/2}\lambda(0)^{-1},\\
|P(I_{N(t)}G(0)+I_{N(t)}H(0))-P(I_{N(t)}G(0))|
&\lesssim_s N(t)^{1-s},\\
|P(I_{N(t)}G(0))-P(G(0))|
&\lesssim_s
\lambda(t)^{-1}(N(t)\lambda(t))^{-1/2}.
\end{align*}
Inserting these estimates in \eqref{target-3} proves \eqref{equ:pu0}.
\end{proof}
\begin{remark}
The control of $E(G(0))$ relies on a cancellation between the kinetic and potential energies, which need not survive the application of $I_N$ as $G(0)$ has only $H^1$ regularity. We therefore introduce the nonnegative
correction $\Xi(t)$ to compensate for this loss and recover
the cancellation in $E(G(0))$. Rather than estimating $\Xi$
as a small error, we retain it in the subsequent virial
estimates, where it has a favorable sign. See section \ref{sec:Loglog} for details.
\end{remark}
\begin{proposition}[Almost conservation laws]\label{prop:almostcons}
Let $d\geq3$, $0<s<1$, and let $\beta$ be given by \eqref{equ:alphaa}.  There exists
\begin{equation}\label{equ:alpha1def}
\alpha_1=\frac{s}{(d+2)(3-2s)}>0
\end{equation}
such that, on $[0,T^+]$,
\begin{align}
\left|E(I_{N(t)}u(t))+\frac{\Xi(t)}{\lambda(t)^2}\right|&\leq\lambda(t)^{-2(1-\alpha_1)},\label{equ:modiener}\\
|P(I_{N(t)}u(t))|&\leq\lambda(t)^{-1+\alpha_1}.
\label{equ:modimom}
\end{align}
Consequently, after taking $b(0)$ sufficiently small, we have
\begin{align}
|\lambda(t)^2E(I_{N(t)}u(t))+\Xi(t)|&\leq\lambda(t)^{2\alpha_1}\leq\Gamma_{b(t)}^{10},\label{equ:lamenermod}\\
\lambda(t)|P(I_{N(t)}u(t))|&\leq\lambda(t)^{\alpha_1}\leq\Gamma_{b(t)}^{10}.
\label{equ:lammonmod}
\end{align}
\end{proposition}

\begin{proof}
We may assume $t=T^+$ and set
$N_\ast=N(T^+)$ and $\lambda_\ast=\lambda(T^+)$.

\medskip
\noindent\emph{\bf Case $d=3$.}
By Lemmas \ref{lem:conincre}, \ref{lem:numberintervals}, and
\ref{lem:einit},
\begin{align*}
&\left|E(I_{N_\ast}u(T^+))+\frac{\Xi(T^+)}{\lambda_\ast^2}\right|\\
&\quad\lesssim_{s,M(u)}N_\ast^{2(1-s)}
+\lambda_\ast^{-2}(N_\ast\lambda_\ast)^{-\frac25}
+\sum_{k=k_0}^{k^+}kN_\ast^{-1}
\left[\left(\frac{N_\ast}{N(t_k)}\right)^{1-s}
\lambda(t_k)^{-1}\right]^3.
\end{align*}
Since $\lambda(t_k)^{-1}=N(t_k)^\beta$ and
$\beta-(1-s)=s/3$, we now sum up the geometric series, as in
\cite[Step~4]{CR09}, to get
\begin{align*}
&\sum_{k=k_0}^{k^+}k
\left[\left(\frac{N_\ast}{N(t_k)}\right)^{1-s}
\lambda(t_k)^{-1}\right]^3
=N_\ast^{3(1-s)}
\sum_{k=k_0}^{k^+}kN(t_k)^s\\
&\hspace{35mm}\lesssim_s|\log\lambda_\ast|N_\ast^{3\beta}.
\end{align*}
Multiplying by $\lambda_\ast^2=N_\ast^{-2\beta}$ and using
$\alpha_1=s/[5(3-2s)]$, we obtain
\begin{align*}
&\left|\lambda_\ast^2E(I_{N_\ast}u(T^+))+\Xi(T^+)\right|\\
&\quad\lesssim_{s,M(u)}
\lambda_\ast^{\frac{2s}{3-2s}}
+\lambda_\ast^{4\alpha_1}
+|\log\lambda_\ast|\lambda_\ast^{\frac{2s}{3-2s}}\\
&\quad=\lambda_\ast^{10\alpha_1}
+\lambda_\ast^{4\alpha_1}
+|\log\lambda_\ast|\lambda_\ast^{10\alpha_1}.
\end{align*}
Therefore, after taking $b(0)$ sufficiently small,
\[
\left|\lambda_\ast^2E(I_{N_\ast}u(T^+))+\Xi(T^+)\right|
\leq\lambda_\ast^{2\alpha_1}.
\]

For the momentum, Lemmas \ref{lem:conincre},
\ref{lem:numberintervals}, and \ref{lem:einit} give
\begin{align*}
|P(I_{N_\ast}u(T^+))|\lesssim_{s,M(u)}{}&N_\ast^{1-s}
+\lambda_\ast^{-1}(N_\ast\lambda_\ast)^{-\frac12}
+|P(G(0))|\\
&+\sum_{k=k_0}^{k^+}kN_\ast^{-2}
\left[\left(\frac{N_\ast}{N(t_k)}\right)^{1-s}
\lambda(t_k)^{-1}\right]^3.
\end{align*}
By \eqref{equ:gg0} and \eqref{equ:almons},
\(\lambda_\ast|P(G(0))|\lesssim\lambda_\ast^{1/2}\).  The geometric
series above therefore yields
\begin{align*}
\lambda_\ast|P(I_{N_\ast}u(T^+))|
\lesssim_{s,M(u)}{}&
\lambda_\ast^{\frac{s}{3-2s}}
+\lambda_\ast^{\frac12}
+|\log\lambda_\ast|
\lambda_\ast^{\frac{4s}{3-2s}}.
\end{align*}
Since
\[
\frac{s}{3-2s}>\alpha_1,\qquad
\frac12>\alpha_1,\qquad
\frac{4s}{3-2s}>\alpha_1,
\]
we obtain, after taking $b(0)$ sufficiently small,
\[
\lambda_\ast|P(I_{N_\ast}u(T^+))|
\leq\lambda_\ast^{\alpha_1}.
\]
This proves \eqref{equ:modiener}--\eqref{equ:modimom} for $d=3$.

\medskip
\noindent\emph{\bf Case $d\geq4$.}
By Lemmas \ref{lem:conincre}, \ref{lem:numberintervals}, and
\ref{lem:einit},
\begin{align*}
&\left|E(I_{N_\ast}u(T^+))+\frac{\Xi(T^+)}{\lambda_\ast^2}\right|\\
&\quad\lesssim_{d,s,M(u)}N_\ast^{2(1-s)}
+\lambda_\ast^{-2}(N_\ast\lambda_\ast)^{-\frac2{d+2}}
+\sum_{k=k_0}^{k^+}kN_\ast^{-\frac4d}
\left[\left(\frac{N_\ast}{N(t_k)}\right)^{1-s}
\lambda(t_k)^{-1}\right]^{2+\frac4d}.
\end{align*}
Since $\lambda(t_k)^{-1}=N(t_k)^\beta$ and
$\beta-(1-s)=s/3$, we now sum up the geometric series, as in
\cite[proof of Proposition~2.16]{SunZheng}, to get
\begin{align*}
\sum_{k=k_0}^{k^+}k
\left[\left(\frac{N_\ast}{N(t_k)}\right)^{1-s}
\lambda(t_k)^{-1}\right]^{2+\frac4d}=&N_\ast^{(1-s)(2+\frac4d)}
\sum_{k=k_0}^{k^+}kN(t_k)^{\frac{s}{3}(2+\frac4d)}\\
\lesssim&_{d,s}|\log\lambda_\ast|
N_\ast^{\beta(2+\frac4d)}.
\end{align*}
Multiplying by $\lambda_\ast^2=N_\ast^{-2\beta}$ gives
\begin{align*}
&\left|\lambda_\ast^2E(I_{N_\ast}u(T^+))+\Xi(T^+)\right|\\
&\quad\lesssim_{d,s,M(u)}
\lambda_\ast^{\frac{2s}{3-2s}}
+\lambda_\ast^{4\alpha_1}
+|\log\lambda_\ast|
\lambda_\ast^{\frac{8s}{d(3-2s)}}\\
&\quad=
\lambda_\ast^{2(d+2)\alpha_1}
+\lambda_\ast^{4\alpha_1}
+|\log\lambda_\ast|
\lambda_\ast^{\frac{8s}{d(3-2s)}}.
\end{align*}
Since
\[
2(d+2)\alpha_1>2\alpha_1,\qquad
4\alpha_1>2\alpha_1,\qquad
\frac{8s}{d(3-2s)}>2\alpha_1,
\]
we obtain, after taking $b(0)$ sufficiently small,
\[
\left|\lambda_\ast^2E(I_{N_\ast}u(T^+))+\Xi(T^+)\right|
\leq\lambda_\ast^{2\alpha_1}.
\]

For the momentum, Lemmas \ref{lem:conincre},
\ref{lem:numberintervals}, and \ref{lem:einit} give
\begin{align*}
|P(I_{N_\ast}u(T^+))|\lesssim_{d,s,M(u)}{}&N_\ast^{1-s}
+\lambda_\ast^{-1}(N_\ast\lambda_\ast)^{-\frac12}
+|P(G(0))|\\
&+\sum_{k=k_0}^{k^+}kN_\ast^{-(1+\frac4d)}
\left[\left(\frac{N_\ast}{N(t_k)}\right)^{1-s}
\lambda(t_k)^{-1}\right]^{2+\frac4d}.
\end{align*}
By \eqref{equ:gg0} and \eqref{equ:almons},
\(\lambda_\ast|P(G(0))|\lesssim\lambda_\ast^{1/2}\).  Multiplying by
$\lambda_\ast=N_\ast^{-\beta}$ and using the geometric series above,
we obtain
\begin{align*}
\lambda_\ast|P(I_{N_\ast}u(T^+))|
\lesssim_{d,s,M(u)}{}&
\lambda_\ast^{\frac{s}{3-2s}}
+\lambda_\ast^{\frac12}+|\log\lambda_\ast|
\lambda_\ast^{\frac{2s(1+\frac4d)}{3-2s}}.
\end{align*}
Therefore, after taking $b(0)$ sufficiently small,
we obtain that 
\[
\lambda_\ast|P(I_{N_\ast}u(T^+))|
\leq\lambda_\ast^{\alpha_1}.
\]
Finally, \eqref{equ:ltass} gives
\[
\lambda_\ast^{2\alpha_1}\leq\Gamma_{b(T^+)}^{10},
\qquad
\lambda_\ast^{\alpha_1}\leq\Gamma_{b(T^+)}^{10},
\]
which proves \eqref{equ:modiener}--\eqref{equ:modimom} for $d\geq4$ and
\eqref{equ:lamenermod}--\eqref{equ:lammonmod} in both cases.
\end{proof}

\section{The proof of Theorem \ref{thm:main}}\label{sec:Loglog}
In this section, we will prove Lemma~\ref{lem:bootstrap} and then Proposition~\ref{prop:main},
and hence we conclude the proof of Theorem~\ref{thm:main}. The argument follows the
analysis of the log--log regime in
\cite{MR05annmath,MR03GAFA,MR06JAMS,Ra05Mathann,CR09}.

\subsection{Control of the geometrical parameters}
Consider the geometrical decomposition \eqref{equ:decass}. The rescaled
time is defined by
\begin{equation}\label{eq:1}
 \frac{ds}{dt}=\frac1{\lambda(s)^2}\quad\text{with}\quad
 s_0=s(0)=e^{\frac{5\pi}{9b(0)}},\quad s_+=s(T_+).
\end{equation}
Then, $\varepsilon$ satisfies the equation  on $[s_0,s_+]$:
\begin{equation}\label{eq:2}
\begin{aligned}
 \frac{\partial\Sigma}{\partial s}+\partial_s\varepsilon_1-M_-(\varepsilon)+b\Lambda\varepsilon_1
 &=\left(\frac{\lambda_s}\lambda+b\right)\Lambda\Sigma
    +\widetilde\gamma_s\Theta+\frac{x_s}\lambda\cdot\nabla\Sigma\\
 &\quad+\left(\frac{\lambda_s}\lambda+b\right)\Lambda\varepsilon_1
    +\widetilde\gamma_s\varepsilon_2+\frac{x_s}\lambda\cdot\nabla\varepsilon_1\\
 &\quad+\operatorname{Im}(\Psi_b)-R_2(\varepsilon)
\end{aligned}
\end{equation}
and
\begin{equation}\label{eq:3}
\begin{aligned}
 \frac{\partial\Theta}{\partial s}+\partial_s\varepsilon_2+M_+(\varepsilon)+b\Lambda\varepsilon_2
&=\left(\frac{\lambda_s}\lambda+b\right)\Lambda\Theta
    -\widetilde\gamma_s\Sigma+\frac{x_s}\lambda\cdot\nabla\Theta\\
 &\quad+\left(\frac{\lambda_s}\lambda+b\right)\Lambda\varepsilon_2
    -\widetilde\gamma_s\varepsilon_1+\frac{x_s}\lambda\cdot\nabla\varepsilon_2\\
 &\quad-\operatorname{Re}(\Psi_b)+R_1(\varepsilon),
\end{aligned}
\end{equation}
with $\widetilde\gamma(s)=-s-\gamma(s)$, $\varepsilon=\varepsilon_1+i\varepsilon_2$ and
$Q_{b(t)}=\Sigma+i\Theta$ in terms of real and imaginary parts. The linear operator close to $Q_b$ is now a deformation of the linear operator $L$ close to $Q$ and is $M=(M_+,M_-)$ with
\begin{align*}
M_+(\var)=&-\Delta\var_1+\var_1-\Big(\frac{4\Sigma^2}{d|Q_b|^2}+1\Big)|Q_b|^\frac4d\var_1-\Big(\frac{4\Sigma\Theta}{d|Q_b|^2}|Q_b|^\frac4d\Big)\var_2\\
M_-(\var)=&-\Delta\var_2+\var_2-\Big(\frac{4\Theta^2}{d|Q_b|^2}+1\Big)|Q_b|^\frac4d\var_2-\Big(\frac{4\Sigma\Theta}{d|Q_b|^2}|Q_b|^\frac4d\Big)\var_1.
\end{align*}
The formally quadratic in $\varepsilon$ interaction terms are:
\begin{align*}
R_1(\varepsilon)
 &= (\varepsilon_1+\Sigma)|\varepsilon+Q_b|^{\frac4d}
       -\Sigma|Q_b|^{\frac4d}\\
 &\quad-\left(\frac{4\Sigma^2}{d|Q_b|^2}+1\right)|Q_b|^{\frac4d}\varepsilon_1
       -\left(\frac{4\Sigma\Theta}{d|Q_b|^2}|Q_b|^{\frac4d}\right)\varepsilon_2,\\
R_2(\varepsilon)
 &= (\varepsilon_2+\Theta)|\varepsilon+Q_b|^{\frac4d}
       -\Theta|Q_b|^{\frac4d}\\
 &\quad-\left(\frac{4\Theta^2}{d|Q_b|^2}+1\right)|Q_b|^{\frac4d}\varepsilon_2
       -\left(\frac{4\Sigma\Theta}{d|Q_b|^2}|Q_b|^{\frac4d}\right)\varepsilon_1.
\end{align*}
Recall the linearized operator
\begin{equation}\label{eq:4}
 L_+=-\Delta+1-\left(1+\frac4d\right)Q^{4/d},\qquad
 L_-=-\Delta+1-Q^{4/d},
\end{equation}
We claim the following estimates on $\varepsilon$ and the geometrical parameters $\lambda,b,x,\gamma$.

\begin{lemma}[Control of the geometrical parameters]\label{lem:4.1}
For all $s\in[s_0,s^+]$, there holds:\nopagebreak[4]

\noindent(i) Estimates induced by the almost conservation of energy and momentum:\nopagebreak[4]
\begin{align}\label{eq:5}
 &\Big|\big[2(\varepsilon_1,\Sigma+b\Lambda\Theta-\operatorname{Re}(\Psi_b))
       +2(\varepsilon_2,\Theta-b\Lambda\Sigma-\operatorname{Im}(\Psi_b))\big]
       -\Big[2\Xi(s)+\int|I_{N\lambda}\nabla\varepsilon|^2\notag\\
 &\quad-\left(1+\frac4d\right)\int Q^{4/d}(I_{N(s)\lambda(s)}\varepsilon_1(s))^2
       -\int Q^{4/d}(I_{N(s)\lambda(s)}\varepsilon_2(s))^2\Big]\Big|\notag\\
 &\quad\leq\delta_0\left(\int|\nabla I_{N(s)\lambda(s)}\varepsilon(s)|^2
                      +\int|\varepsilon(s)|^2e^{-|y|}\right)+\Gamma_{b(s)}^{1-C\eta},
\end{align}
\begin{equation}\label{eq:6}
 |(\varepsilon_2,\nabla Q)|\leq\delta_0\left(\int|I_{N(s)\lambda(s)}\nabla\varepsilon(s)|^2\right)^{1/2}
   +\Gamma_{b(s)}^{10}.
\end{equation}
\noindent(ii) Estimates on the modulation parameters:
\begin{equation}\label{eq:7}
\begin{aligned}
 \left|\frac{\lambda_s}\lambda+b\right|+|b_s|
 &\leq C\left(\Xi(s)+\int|\nabla I_{N(s)\lambda(s)}\varepsilon(s)|^2
                 +\int|\varepsilon(s)|^2e^{-|y|}\right)\\
 &\quad+\Gamma_{b(s)}^{1-C\eta}+F(s),
\end{aligned}
\end{equation}
\begin{equation}\label{eq:8}
\begin{aligned}
 &\left|\widetilde\gamma_s-\frac{(\varepsilon_1,L_+\Lambda^2Q)}{|\Lambda Q|_2^2}\right|
  +\left|\frac{x_s}\lambda\right|\\
 &\quad\leq\delta_0\left[\left(\int|\varepsilon(s)|^2e^{-|y|}\right)^{1/2}
                +\left\lVert \nabla I_{N(s)\lambda(s)}\varepsilon(s)\right\rVert_{L^2}\right]
       +\Gamma_{b(s)}^{1-C\eta}\\
 &\qquad+C\left(\Xi(s)+\int|\nabla I_{N(s)\lambda(s)}\varepsilon(s)|^2\right)+F(s),
\end{aligned}
\end{equation}
Here $\delta_0$ is a small enough universal constant $0<\delta_0\ll1$ and
$F(s)\geq0$ satisfies the uniform estimate: $\forall\,s_1\in[s_0,s^+]$,
\begin{equation}\label{eq:9}
 \int_{s_1}^{s^+}F(s)\,ds\lesssim\lambda(s_1)^{\alpha_2}
\end{equation}
for some $\alpha_2>0$.
\end{lemma}

\begin{remark}\label{rem:4.2}In the low regularity setting, since $u_0\in H^s$ with $s\in(0,1)$,  various terms arises from the use of the $I$-operator, which will be absorbed into $F$ and controlled in time
averaging sense using Strichartz estimate, i.e.  Lemma~\ref{commutor} and
Remark~\ref{remark2}. In the energy-space,  these terms will
vanish.
\end{remark}

\begin{proof}[Proof of Lemma~\ref{lem:4.1}]
The proof follows \cite[Lemma~4.1]{CR09} and uses the orthogonality
conditions and the almost conservation laws.

Let us begin with the proof of \eqref{eq:5}. Observe that
\[
 I_Nu(t,x)=\frac1{\lambda(t)^{d/2}}[I_{N\lambda}(Q_b+\varepsilon)]
 \left(\frac{x-x(t)}\lambda\right)e^{-i\gamma(t)}
\]
and thus \eqref{equ:lamenermod} and \eqref{equ:ltass} imply:
\begin{equation}\label{eq:10}
 \left|E(I_{N\lambda}(Q_b+\varepsilon))+\Xi(s)\right|\leq\Gamma_b^{10}.
\end{equation}
We expand this relation. Since $Q$ is a Schwartz function,  from
\eqref{equ:profileapprox}, we have
\begin{equation}\label{eq:11}
 \left\lVert I_{N(t)\lambda(t)}Q-Q\right\rVert_{H^p}
 \lesssim (N(t)\lambda(t))^{-2}\left\lVert Q\right\rVert_{H^{p+2}}
 \lesssim\Gamma_b^{10}.
\end{equation}
Consequently,  this kind of error is negligible. The expansion of the almost conservation
 yields:
\begin{equation}\label{eq:12}
\begin{aligned}
 &2(I_{N\lambda}\varepsilon_1,\Sigma+b\Lambda\Theta-\operatorname{Re}(\Psi_b))
   +2(I_{N\lambda}\varepsilon_2,\Theta-b\Lambda\Sigma-\operatorname{Im}(\Psi_b))\\
 =&-2\lambda^2E(I_Nu)+2E(Q_b)+\int|\nabla I_{N\lambda}\varepsilon|^2\\
 &-\int\left(1+\frac{4\Sigma^2}{d|Q_b|^2}\right)
       |Q_b|^{4/d}(I_{N\lambda}\varepsilon_1)^2
       -\int\left(1+\frac{4\Theta^2}{d|Q_b|^2}\right)
       |Q_b|^{4/d}(I_{N\lambda}\varepsilon_2)^2\\
 &-\frac8d\int|Q_b|^{\frac4d-2}\Sigma\Theta
       (I_{N\lambda}\varepsilon_1)(I_{N\lambda}\varepsilon_2)\\
 &-\frac d{d+2}\int\Big[
       |Q_b+I_{N\lambda}\varepsilon|^{2+\frac4d}-|Q_b|^{2+\frac4d}\\
 &\hspace{20mm}-\left(2+\frac4d\right)|Q_b|^{4/d}
       (\Sigma I_{N\lambda}\varepsilon_1+\Theta I_{N\lambda}\varepsilon_2)\\
 &\hspace{20mm}-\left(1+\frac2d\right)|Q_b|^{4/d}|I_{N\lambda}\varepsilon|^2\\
 &\hspace{20mm}-\frac4d\left(1+\frac2d\right)|Q_b|^{\frac4d-2}
       (\Sigma I_{N\lambda}\varepsilon_1+\Theta I_{N\lambda}\varepsilon_2)^2\Big]\\
 &+2E(I_{N\lambda}(Q_b+\varepsilon))-2E(Q_b+I_{N\lambda}\varepsilon).
\end{aligned}
\end{equation}
The nonlinear terms are estimated using Sobolev embedding and the
$L^2$ smallness in \eqref{equ:btvart}, as in
\cite[Lemma~3.2]{SunZheng}:
\begin{align*}
 &\left|\int\Big[
 |Q_b+I_{N\lambda}\varepsilon|^{2+\frac4d}-|Q_b|^{2+\frac4d}
 -\left(2+\frac4d\right)|Q_b|^{4/d}
       (\Sigma I_{N\lambda}\varepsilon_1+\Theta I_{N\lambda}\varepsilon_2)\right.\\
 &\qquad\left.-\left(1+\frac2d\right)|Q_b|^{4/d}|I_{N\lambda}\varepsilon|^2
 -\frac4d\left(1+\frac2d\right)|Q_b|^{\frac4d-2}
       (\Sigma I_{N\lambda}\varepsilon_1+\Theta I_{N\lambda}\varepsilon_2)^2
 \Big]\right|\\
 &\quad\leq\delta_0\left(\int|\nabla I_{N\lambda}\varepsilon|^2
                         +\int|\varepsilon|^2e^{-|y|}\right)+\Gamma_b^{10}.
\end{align*}
By \eqref{equ:profileapprox} and \eqref{equ:L2tailI},
\[
 \left|E(I_{N\lambda}(Q_b+\varepsilon))-E(Q_b+I_{N\lambda}\varepsilon)\right|
 \lesssim\Gamma_b^{10}.
\]
The error in the linear term from \eqref{eq:12} to \eqref{eq:5} is estimated using
the regularity of $Q_b$ and \eqref{eq:11}, and the almost conservation law
\eqref{equ:lamenermod} now implies \eqref{eq:5}.

The proof of \eqref{eq:6} follows similarly by expanding the almost conservation
of the momentum \eqref{equ:lammonmod}, using \eqref{equ:profileapprox}, and is left
to the reader.

We now turn to the proof of \eqref{eq:7} which amounts to taking the inner product
of the equation with respect to the directions of the orthogonality conditions and
estimating the various nonlinear terms. Note that this computation makes sense
thanks to the smoothness of $Q$. Let us for example compute $b_s$ following the proof
of Lemma~5 in~\cite{MR03GAFA}. From \eqref{equ:cth}, we have:
\begin{align*}
 &b_s\big[(\partial_b\Theta,\Lambda\Sigma)-(\partial_b\Sigma,\Lambda\Theta)
           +(\varepsilon_1,\partial_b\Lambda\Theta)
           -(\varepsilon_2,\partial_b\Lambda\Sigma)\big]\\
 =&-(M_+(\varepsilon),\Lambda\Sigma)-(M_-(\varepsilon),\Lambda\Theta)
       -\frac{\lambda_s}{\lambda}
          \{(\varepsilon_2,\Lambda^2\Sigma)-(\varepsilon_1,\Lambda^2\Theta)\}\\
 &-\widetilde\gamma_s
          \{(\varepsilon_1,\Lambda\Sigma)+(\varepsilon_2,\Lambda\Theta)\}
       -\frac{x_s}{\lambda}\cdot
          \{(\varepsilon_2,\nabla\Lambda\Sigma)-(\varepsilon_1,\nabla\Lambda\Theta)\}\\
 &+(R_1(\varepsilon),\Lambda\Sigma)+(R_2(\varepsilon),\Lambda\Theta)
       -(\operatorname{Re}(\Psi_b),\Lambda\Sigma)
       -(\operatorname{Im}(\Psi_b),\Lambda\Theta).
\end{align*}
We integrate by parts and compute:
\begin{equation}\label{eq:13}
\begin{aligned}
 &-(M_+(\varepsilon),\Lambda\Sigma)-(M_-(\varepsilon),\Lambda\Theta)\\
 =&2(\varepsilon_1,\Sigma+b\Lambda\Theta-\operatorname{Re}(\Psi_b))
       +2(\varepsilon_2,\Theta-b\Lambda\Sigma-\operatorname{Im}(\Psi_b))\\
 &-b\bigl((\varepsilon_2,\Lambda^2\Sigma)-(\varepsilon_1,\Lambda^2\Theta)\bigr)
       -(\varepsilon_1,\operatorname{Re}(\Lambda\Psi_b))
       -(\varepsilon_2,\operatorname{Im}(\Lambda\Psi_b)).
\end{aligned}
\end{equation}
We now inject \eqref{eq:5}, rearrange the terms and arrive to the following:
\begin{equation}\label{eq:14}
\begin{aligned}
 \frac{\left\lVert yQ\right\rVert_{L^2}^2}{4}b_s
 =&2\Xi+H(I_{N\lambda}\varepsilon,I_{N\lambda}\varepsilon)
   -(\varepsilon_1,\operatorname{Re}(\Lambda\Psi_b))-(\varepsilon_2,\operatorname{Im}(\Lambda\Psi_b))\\
 &-\left(\frac{\lambda_s}\lambda+b\right)
       \{(\varepsilon_2,\Lambda^2\Sigma)-(\varepsilon_1,\Lambda^2\Theta)\}\\
 &-\widetilde\gamma_s\{(\varepsilon_1,\Lambda\Sigma)+(\varepsilon_2,\Lambda\Theta)\}
 -\frac{x_s}\lambda\cdot\{(\varepsilon_2,\nabla\Lambda\Sigma)-(\varepsilon_1,\nabla\Lambda\Theta)\}\\
 &+O\!\left(\delta_0\left(\Xi+\int|\nabla I_{N(s)\lambda(s)}\varepsilon(s)|^2
                     +\int|\varepsilon(s)|^2e^{-|y|}\right)+\Gamma_{b(s)}^{1-C\eta}\right)\\
 &+O(F(s))
\end{aligned}
\end{equation}
where $H$ is the quadratic form defined in Assumption~\ref{spectral-property},
and $F$ encodes the nonlinear error terms containing $I_{N\lambda}-\mathit{Id}$.
The nonlinear terms in $I_{N\lambda}\varepsilon$ are estimated as in
\cite[proof of Lemma~3.4]{SunZheng}. The remaining nonlinear terms are
\[
 (R_\ell(\varepsilon)-R_\ell(I_{N\lambda}\varepsilon),\phi),
 \qquad \ell=1,2,
\]
for the Schwartz functions $\phi$ appearing in the modulation equations.
The control of these terms requires the use of Strichartz estimates and
we claim that they satisfy \eqref{eq:9}.

Let us assume the estimate \eqref{eq:9}. The remaining orthogonality
conditions and the properties of $L_+,L_-$ give the equations for the
other geometrical parameters, as in \cite{MR03GAFA,MR06JAMS}.
Differentiating \eqref{equ:cy1}, substituting \eqref{eq:2}--\eqref{eq:3},
and integrating by parts gives
\begin{equation}\label{source:translationidentity}
\begin{aligned}
 &\sum_{k=1}^d\mathcal A_{jk}\frac{(x_s)_k}{\lambda}
   =(M_+(\varepsilon),y_j\Theta)-(M_-(\varepsilon),y_j\Sigma)\\
 &\quad+\widetilde\gamma_s
     \{(\varepsilon_1,y_j\Theta)-(\varepsilon_2,y_j\Sigma)\}
   -b_s\{(\varepsilon_1,y_j\partial_b\Sigma)
                  +(\varepsilon_2,y_j\partial_b\Theta)\}\\
 &\quad+\frac{\lambda_s}{\lambda}
       \{(\varepsilon_1,\Lambda(y_j\Sigma))+(\varepsilon_2,\Lambda(y_j\Theta))\}\\
 &\quad-(R_1(\varepsilon),y_j\Theta)+(R_2(\varepsilon),y_j\Sigma)
       -(\operatorname{Im}\Psi_b,y_j\Sigma)+(\operatorname{Re}\Psi_b,y_j\Theta),
\end{aligned}
\end{equation}
where $1\leq j\leq d$ and
\[
 \mathcal A_{jk}=-\frac12\left\lVert Q_b\right\rVert_{L^2}^2\delta_{jk}
       -(\varepsilon_1,\partial_k(y_j\Sigma))
       -(\varepsilon_2,\partial_k(y_j\Theta)).
\]
The matrix $\mathcal A$ is invertible by the smallness of $\varepsilon$ in $L^2$.
At $b=0$, we have
\[
 -(L_-\varepsilon_2,y_jQ)=2(\varepsilon_2,\partial_jQ),\qquad
 L_-(y_jQ)=-2\partial_jQ.
\]
Solving the modulation system and estimating the nonlinear terms as in
\cite[proof of Lemma~3.6]{SunZheng} gives
\begin{equation}\label{source:translationmomentuminput}
\begin{aligned}
 \left|\frac{x_s}{\lambda}\right|
 &\leq C|(\varepsilon_2,\nabla Q)|
       +\delta_0\left(\int|\varepsilon|^2e^{-|y|}\right)^{1/2}\\
 &\quad+C\left(\Xi(s)+\int|\nabla I_{N\lambda}\varepsilon|^2
                       +\int|\varepsilon|^2e^{-|y|}\right)
       +\Gamma_b^{1-C\eta}+F(s).
\end{aligned}
\end{equation}
The $I$-operator errors are estimated in
\eqref{eq:17}--\eqref{eq:18}. Applying \eqref{eq:6} to
\eqref{source:translationmomentuminput}, and the phase equation to the
remaining component, yields \eqref{eq:7} and \eqref{eq:8}, after reducing
the small constants.

\noindent\emph{Proof of \eqref{eq:9}.}
Let us introduce the decomposition:
\begin{equation}\label{eq:15}
 u(t,x)=(Q_b)_{\mathrm{sing}}(t,x)+\widetilde u(t,x)
\end{equation}
with
\[
 \begin{aligned}
 (Q_b)_{\mathrm{sing}}(t,x)&=\frac1{\lambda^{d/2}}
 Q_{b(t)}\!\left(\frac{x-x(t)}\lambda\right)e^{-i\gamma(t)},\\
 \widetilde u(t,x)&=\frac1{\lambda(t)^{d/2}}
 \varepsilon\!\left(t,\frac{x-x(t)}{\lambda(t)}\right)e^{-i\gamma(t)}.
 \end{aligned}
\]
Let $s_1\in[s_0,s^+]$ corresponding to some time $t_1\in[0,T^+]$.
Let $t_k$ be the doubling time intervals \eqref{equ:ltk} and $k_1$ such that
$t_1\in[t_{k_1},t_{k_1+1}]$. We first split the time interval $[t_{k_1},T^+]$
into doubling time intervals. Let
\[
 \phi_\lambda(t,x)=\frac1{\lambda(t)^{d/2}}
 \phi\!\left(\frac{x-x(t)}{\lambda(t)}\right)e^{-i\gamma(t)}.
\]
We now split the doubling time intervals $[t_k,t_{k+1}]$ into LWP time intervals
$[\tau_k^j,\tau_k^{j+1}]$, $1\leq j\leq J_k\lesssim k$ from \eqref{equ:jk}.
Let $J_{k,j}$ be the corresponding interval in rescaled time. Lemma~\ref{lem:numberintervals},
\eqref{equ:multcompare} and Remark~\ref{remark2} give for $|J_{k,j}|\lesssim1$
\begin{equation}\label{eq:16}
\begin{aligned}
 &\left\lVert Q_b+\varepsilon\right\rVert_{L_s^\infty L_y^2(J_{k,j})}
 +\left\lVert\langle\nabla\rangle I_{N\lambda}(Q_b+\varepsilon)\right\rVert_
 {L_s^\infty L_y^2(J_{k,j})}\\
 &\qquad+\left\lVert\langle\nabla\rangle I_{N\lambda}(Q_b+\varepsilon)\right\rVert_
 {L_s^2L_y^{\frac{2d}{d-2}}(J_{k,j})}\lesssim1.
\end{aligned}
\end{equation}
We begin with the terms in which $I_{N\lambda}-\mathit{Id}$ acts on a
Schwartz function $\phi$. From \eqref{equ:profileapprox} and Sobolev,
\begin{align*}
 \left\lVert(\mathit{Id}-I_N)\phi_\lambda\right\rVert_
 {L_t^2L_x^{\frac{2d}{d-2}}([\tau_k^j,\tau_k^{j+1}])}^2
 =&\int_{J_{k,j}}
 \left\lVert(\mathit{Id}-I_{N\lambda})\phi\right\rVert_
 {L_y^{\frac{2d}{d-2}}}^2\,ds\\
 \lesssim& |J_{k,j}|[N(t_k)\lambda(t_k)]^{-4}\lesssim [N(t_k)\lambda(t_k)]^{-4}.
\end{align*}
Next from Strichartz control on $[\tau_k^j,\tau_k^{j+1}]$,
\[
 \left\lVert |u|^{4/d}u\right\rVert_
 {L_t^2L_x^{\frac{2d}{d+2}}([\tau_k^j,\tau_k^{j+1}])}
 =\left\lVert u\right\rVert_
 {L_t^{\frac{2(d+4)}d}L_x^{\frac{2(d+4)}{d+2}}([\tau_k^j,\tau_k^{j+1}])}^{1+4/d}
 \lesssim1.
\]
Hence
\begin{align*}
 &\int_{J_{k,j}}\left|\int |Q_b+\varepsilon|^{4/d}(Q_b+\varepsilon)
       \overline{(\mathit{Id}-I_{N\lambda})\phi}\right|\,ds\\
 =&\int_{\tau_k^j}^{\tau_k^{j+1}}
       \left|\int |u|^{4/d}u\,
       \overline{(\mathit{Id}-I_N)\phi_\lambda}\right|\,dt
       \lesssim [N(t_k)\lambda(t_k)]^{-2}.
\end{align*}
The errors involving $Q_b$ are estimated by \eqref{equ:profileapprox} and
\eqref{equ:L2tailI} in the same way.

We now turn to the commutator terms. By self-adjointness, the definition of
$R_1,R_2$, and the preceding estimates, for each Schwartz function
$\phi$ in these scalar products,
\begin{align*}
 &\int_{J_{k,j}}\sum_{\ell=1}^2
 \bigl|(R_\ell(\varepsilon)-R_\ell(I_{N\lambda}\varepsilon),\phi)\bigr|\,ds\\
 \lesssim& [N(t_k)\lambda(t_k)]^{-2}
\\
 &+\left\lVert\begin{aligned}
 &I_{N(s)\lambda(s)}\big((Q_b+\varepsilon)|Q_b+\varepsilon|^{4/d}\big)\\
 &\quad-I_{N(s)\lambda(s)}(Q_b+\varepsilon)
       |I_{N(s)\lambda(s)}(Q_b+\varepsilon)|^{4/d}
 \end{aligned}\right\rVert_{L_s^2L_y^{\frac{2d}{d+2}}(J_{k,j})}
 \left\lVert\phi\right\rVert_{L_s^2L_y^{\frac{2d}{d-2}}(J_{k,j})}.
\end{align*}
Here we used the fact that
\[
 \left\lVert\phi\right\rVert_{L_s^2L_y^{\frac{2d}{d-2}}(J_{k,j})}^2
 \leq |J_{k,j}|\sup_{s\in J_{k,j}}
 \left\lVert\phi(s)\right\rVert_{L_y^{\frac{2d}{d-2}}}^2\lesssim1.
\]
We obtain
\begin{equation}\label{eq:17}
 \int_{J_{k,j}}\sum_{\ell=1}^2
 \bigl|(R_\ell(\varepsilon)-R_\ell(I_{N\lambda}\varepsilon),\phi)\bigr|\,ds
 \lesssim [N(t_k)\lambda(t_k)]^{-(1+\nu)}+[N(t_k)\lambda(t_k)]^{-2},
\end{equation}
provided that
\begin{equation}\label{eq:18}
\begin{aligned}
 &\left\lVert\begin{aligned}
 &I_{N(s)\lambda(s)}\big((Q_b+\varepsilon)|Q_b+\varepsilon|^{4/d}\big)\\
 &\quad-I_{N(s)\lambda(s)}(Q_b+\varepsilon)
       |I_{N(s)\lambda(s)}(Q_b+\varepsilon)|^{4/d}
 \end{aligned}\right\rVert_{L_s^2L_y^{\frac{2d}{d+2}}(J_{k,j})}\\
 \lesssim& [N(t_k)\lambda(t_k)]^{-(1+\nu)}.
\end{aligned}
\end{equation}
Let us assume \eqref{eq:18} and conclude the proof. From \eqref{eq:17} and
\eqref{equ:jk}, we get:
\begin{align*}
 \int_{s_1}^{s^+}F(s)\,ds
 &\lesssim\sum_{k=k_1}^{k^+}\sum_{j=1}^{J_k}
       \bigl([N(t_k)\lambda(t_k)]^{-(1+\nu)}+[N(t_k)\lambda(t_k)]^{-2}\bigr)\\
 &\lesssim\sum_{k=k_1}^{+\infty}k
       \bigl([N(t_k)\lambda(t_k)]^{-(1+\nu)}+[N(t_k)\lambda(t_k)]^{-2}\bigr)
\\
 &\lesssim\lambda(s_1)^{\alpha_2},
\end{align*}
where $N(t_k)\lambda(t_k)=2^{\frac{2s}{3-2s}k}$ and
$\alpha_2=\alpha_2(s)>0$. This proves \eqref{eq:9}.
Here $F$ is a nonnegative majorant of the errors for all the modulation directions.

\noindent\emph{Proof of \eqref{eq:18}.}
Estimate \eqref{eq:18} follows from Lemma~\ref{lem:xianchafa} and
Remark~\ref{remark2} by scaling. This concludes the proof of
Lemma~\ref{lem:4.1}.
\end{proof}

We will also use two consequences of these estimates. First,
\eqref{eq:5} and \eqref{equ:vartass} imply
$0\leq\Xi(s)\lesssim\Gamma_b^{1/8}$. Hence,
\begin{equation}\label{source:roughmodulation}
 \left|\frac{\lambda_s}{\lambda}+b\right|+|b_s|
       +|\widetilde\gamma_s|+\left|\frac{x_s}{\lambda}\right|
 \lesssim\Gamma_b^{1/8}+F(s).
\end{equation}
Second, for every fixed $m\geq0$, the error majorants constructed above
satisfy
\begin{align}\nonumber
 \int_{s_1}^{s^+}b(s)^{-m}F(s)\,ds
 &\lesssim\sum_{k\geq k_1}k(\log k)^m
       \bigl([N(t_k)\lambda(t_k)]^{-(1+\nu)}+[N(t_k)\lambda(t_k)]^{-2}\bigr)\\\label{source:weightedF}
 &\lesssim\lambda(s_1)^{\alpha_2},
\end{align}
after decreasing $\alpha_2$. Indeed, \eqref{equ:ltass} gives
$b^{-1}\lesssim\log k$ on a doubling time interval. The same conclusion
holds with any fixed factor $A^C$, since $A^C\lesssim k^{C'}$.

\subsection{Virial dispersion}
The virial estimates used below are those of
\cite{MR03GAFA,MR06JAMS}, in the form given in
\cite[Section~4.2]{CR09}. The first is the global virial estimate.

\begin{lemma}[Global virial estimate]\label{lem:4.3}
There holds:
\begin{equation}\label{eq:19}
 b_s\geq c_0\left(\Xi(s)+\int|\nabla I_{N(s)\lambda(s)}\varepsilon(s)|^2
                   +\int|\varepsilon(s)|^2e^{-|y|}\right)
       -\Gamma_{b(s)}^{1-C\eta}-F(s)
\end{equation}
where $F$ satisfies \eqref{eq:9}.
\end{lemma}

\begin{proof}
The proof uses the orthogonality conditions and
Assumption~\ref{spectral-property}, as in \cite[Lemma~4.3]{CR09}.
In \eqref{eq:14}, the vector paired with $x_s/\lambda$ is bounded by
$C(\int|\varepsilon|^2e^{-|y|})^{1/2}$. The additional gradient contribution
in \eqref{eq:8} therefore satisfies
\[
 C\delta_0\left\lVert \nabla I_{N\lambda}\varepsilon\right\rVert_{L^2}
              \left(\int|\varepsilon|^2e^{-|y|}\right)^{1/2}
 \leq C\delta_0\left(\int|\nabla I_{N\lambda}\varepsilon|^2
                         +\int|\varepsilon|^2e^{-|y|}\right).
\]
It is absorbed by reducing $\delta_0$. We then insert
\eqref{eq:7} and \eqref{eq:8} to derive:
\begin{align*}
 \frac{\left\lVert yQ\right\rVert_{L^2}^2}{4}b_s
 &\geq2\Xi(s)+\widetilde H(I_{N\lambda}\varepsilon,I_{N\lambda}\varepsilon)\\
 &\quad-\delta_0\left(\Xi(s)+\int|\nabla I_{N(s)\lambda(s)}\varepsilon(s)|^2
                          +\int|\varepsilon(s)|^2e^{-|y|}\right)
 -\Gamma_{b(s)}^{1-C\eta}-F(s)
\end{align*}
with $F(s)$ satisfying \eqref{eq:9}, and  $\widetilde H$ is the modified
quadratic form:
\[
 \widetilde H(\varepsilon,\varepsilon)=H(\varepsilon,\varepsilon)
 -\frac1{|\Lambda Q|_2^2}(\varepsilon_1,L_+\Lambda^2Q)(\varepsilon_1,\Lambda Q).
\]
Assumption~\ref{spectral-property} and the argument in
\cite[proof of Lemma~4.3]{CR09} give, for every $\varepsilon\in H^1$,
\begin{align*}
 \widetilde H(\varepsilon,\varepsilon)
 &\geq c_1\left(\int|\nabla\varepsilon|^2+\int|\varepsilon|^2e^{-|y|}\right)\\
 &\quad-\frac1{\widetilde\delta_1}\bigl\{(\varepsilon_1,Q)^2+(\varepsilon_1,|y|^2Q)^2
 +(\varepsilon_1,yQ)^2+(\varepsilon_2,\Lambda Q)^2+(\varepsilon_2,\Lambda^2Q)^2+(\varepsilon_2,\nabla Q)^2\bigr\},
\end{align*}
for some universal constant $c_1>0$. Applying the coercivity estimates to
$I_{N\lambda}\varepsilon$, it now remains to observe that the negative directions of
$\widetilde H$ are controlled using the orthogonality conditions and the degeneracy
estimates \eqref{eq:5}, \eqref{eq:6}, and \eqref{eq:19} follows.
This concludes the proof of Lemma~\ref{lem:4.3}.
\end{proof}

The next estimate uses the observation of \cite[Section~4.2]{CR09} concerning the term
$-(\varepsilon_1,\operatorname{Re}(\Lambda\Psi_b))
-(\varepsilon_2,\operatorname{Im}(\Lambda\Psi_b))$ in \eqref{eq:14}.
This term comes from the error $\Psi_b$ in \eqref{def:psib}, since $Q_b$
is not an exact self-similar solution. The outgoing radiation $\zeta_b$
of Lemma~\ref{lem:linoutrad} is introduced to estimate this term at the
order $\Gamma_b$.

As explained in \cite{CR09,MR06JAMS}, the profile $Q_b+\zeta_b$ gives a
better approximation in the radiative region. Since $\zeta_b\notin L^2$
by \eqref{equ:gammab}, the radiation is first cut off in space. The cutoff
parameter is
\begin{equation}\label{eq:20} A(t)=e^{\frac{2a}{b(t)}}\end{equation}
for some small parameter $0<a\ll1$. The parameters are chosen with $a$
sufficiently small and then $\eta\ll a$, so that
\[
 \Gamma_b^{-a/2}\leq A\leq\Gamma_b^{-3a/2},\qquad
 3a<\frac1{48},\qquad 1-Ca>\frac{17}{24},\qquad \frac a2>C\eta.
\]
Define
\[
 \widetilde\zeta=\chi\!\left(\frac rA\right)\zeta_b,
\]
where $\chi(r)$ is a smooth nonnegative radial cutoff with $\chi(r)=1$
for $0\leq r\leq\frac32$ and $\chi(r)=0$ for $r\geq2$. Thus
$\widetilde\zeta$ is a Schwartz function. Set
\begin{equation}\label{eq:21} \widetilde\varepsilon=\varepsilon-\widetilde\zeta.\end{equation}
The calculation leading to \eqref{eq:14} is then repeated with
$Q_b+\widetilde\zeta$, as in \cite[Lemma~4.4]{CR09}. The scalar products
involve Schwartz functions constructed from $Q_b$ and $\widetilde\zeta$.
The terms containing $I_{N\lambda}-\mathit{Id}$ are estimated using
\eqref{eq:17}--\eqref{eq:18}. The remaining calculation is the one in
\cite[Lemma~6]{MR06JAMS}.

\begin{lemma}[Virial dispersion in the radiative regime]\label{lem:4.4}
For some universal constants $c_1>0$ and $s\in[s_0,s^+]$, it holds:
\begin{equation}\label{eq:22}
\begin{aligned}
 \{f_1(s)\}_s
 &\geq c_1\left(\Xi(s)+\int|\nabla I_{N(s)\lambda(s)}\widetilde\varepsilon(s)|^2
                  +\int|\varepsilon(s)|^2e^{-|y|}+\Gamma_b\right)\\
 &\quad-\frac1{\delta_1}\int_A^{2A}|\varepsilon|^2-F(s),
\end{aligned}
\end{equation}
with $F(s)$ satisfying \eqref{eq:9} and
\begin{equation}\label{eq:23}
 f_1(s)=\frac b4|yQ_b|_2^2
 +\frac12\operatorname{Im}\!\left(\int(y\cdot\nabla\widetilde\zeta)\overline{\widetilde\zeta}\right)
 +(\varepsilon_2,\Lambda\widetilde\zeta_{\mathrm{re}})
 -(\varepsilon_1,\Lambda\widetilde\zeta_{\mathrm{im}}).
\end{equation}
\end{lemma}

\begin{proof}
We argue as in \cite{SunZheng}.
For the terms containing
$I_{N\lambda}-\mathit{Id}$, we use \eqref{equ:profileapprox}, \eqref{equ:L2tailI},
\eqref{eq:16} and \eqref{eq:18}.
Choose $a$ sufficiently small and then $\eta\ll a$. The radiation bounds give,
for each fixed $m$,
\begin{equation}\label{source:highprofiles}
 \left\lVert \widetilde\zeta\right\rVert_{H^m}
 +\left\lVert \chi(y/A)\partial_b\zeta_b\right\rVert_{H^m}\lesssim A^{C_m}.
\end{equation}
Consequently,
\begin{equation}\label{source:radiationI}
\begin{aligned}
 \left\lVert (I_{N\lambda}-\mathit{Id})\widetilde\zeta\right\rVert_{H^1}
 &\lesssim A^C(N\lambda)^{-2}\lesssim\Gamma_b^{10},\\
 \left\lVert (I_{N\lambda}-\mathit{Id})\varepsilon\right\rVert_{L^2}
 &\lesssim(N\lambda)^{-1}
 \bigl(\left\lVert \nabla I_{N\lambda}\varepsilon\right\rVert_{L^2}+\left\lVert \varepsilon\right\rVert_{L^2}\bigr).
\end{aligned}
\end{equation}
The energy terms are treated as in \cite[Lemma~4.4]{CR09}, using
\eqref{equ:lamenermod}. The expansion of the
modified momentum gives
\begin{equation}\label{source:radiativemomentumidentity}
\begin{aligned}
 \lambda P(I_Nu)
 &=P\big(I_{N\lambda}(Q_b+\widetilde\zeta)
                     +I_{N\lambda}\widetilde\varepsilon\big)\\
 &=-2(I_{N\lambda}\widetilde\varepsilon_2,\nabla Q)
       +P(I_{N\lambda}\widetilde\varepsilon)\\
 &\quad+2\operatorname{Im}\int\nabla I_{N\lambda}\widetilde\varepsilon\,
                  \overline{I_{N\lambda}(Q_b+\widetilde\zeta)-Q}.
\end{aligned}
\end{equation}
Here $P(I_{N\lambda}(Q_b+\widetilde\zeta))=0$ by radiality. Thus
\eqref{equ:lammonmod} implies
\begin{equation}\label{source:radiativemomentumbound}
\begin{aligned}
 2|(I_{N\lambda}\widetilde\varepsilon_2,\nabla Q)|
 &\leq\Gamma_b^{10}
  +\Big(\left\lVert I_{N\lambda}\widetilde\varepsilon\right\rVert_{L^2}
       +2\left\lVert I_{N\lambda}(Q_b+\widetilde\zeta)-Q\right\rVert_{L^2}\Big)
       \left\lVert\nabla I_{N\lambda}\widetilde\varepsilon\right\rVert_{L^2}.
\end{aligned}
\end{equation}
The coefficient of $\left\lVert\nabla I_{N\lambda}\widetilde\varepsilon
\right\rVert_{L^2}$ is small. Indeed, $Q_b\to Q$ in $L^2$, and Sobolev's
inequality and Lemma~\ref{lem:linoutrad} give, for $d\geq3$,
\[
 \left\lVert\widetilde\zeta\right\rVert_{L^2}
 \lesssim A\left\lVert\nabla\zeta_b\right\rVert_{L^2}
 \lesssim A\Gamma_b^{1/2-C\eta}=o(1).
\]
The same estimates give $\left\lVert\Lambda\widetilde\zeta\right\rVert_{L^2}=o(1)$.
By self-adjointness and \eqref{equ:profileapprox},
\[
 \begin{aligned}
 |(I_{N\lambda}\widetilde\varepsilon_2-\widetilde\varepsilon_2,\nabla Q)|
 &=|(\widetilde\varepsilon_2,(I_{N\lambda}-\mathit{Id})\nabla Q)|\\
 &\lesssim(N\lambda)^{-2}
       \left\lVert\widetilde\varepsilon\right\rVert_{L^2}
       \left\lVert\nabla Q\right\rVert_{H^2}
 \lesssim\Gamma_b^{10}.
 \end{aligned}
\]
It follows that
\begin{equation}\label{source:radiativemomentum}
 |(\widetilde\varepsilon_2,\nabla Q)|
 \leq\delta_0\left\lVert\nabla I_{N\lambda}\widetilde\varepsilon\right\rVert_{L^2}
       +C\Gamma_b^{10}.
\end{equation}

Substitute $\varepsilon=\widetilde\varepsilon+\widetilde\zeta$ in
\eqref{source:translationidentity}. Since $\widetilde\zeta$ is radial,
its scalar products with the odd functions in that identity vanish.
The coefficient matrix remains invertible. In the nonlinear terms,
radiality gives
\[
 (R_1(I_{N\lambda}\widetilde\zeta),y_j\Theta)
 =(R_2(I_{N\lambda}\widetilde\zeta),y_j\Sigma)=0,
 \qquad 1\leq j\leq d.
\]
The definitions of $R_1,R_2$ give
\begin{align*}
 &R_1(I_{N\lambda}\varepsilon)-R_1(I_{N\lambda}\widetilde\zeta)\\
 &\quad=(\Sigma+I_{N\lambda}\varepsilon_1)
             |Q_b+I_{N\lambda}\varepsilon|^{4/d}
       -(\Sigma+I_{N\lambda}\widetilde\zeta_{\mathrm{re}})
             |Q_b+I_{N\lambda}\widetilde\zeta|^{4/d}\\
 &\qquad-\left(1+\frac{4\Sigma^2}{d|Q_b|^2}\right)|Q_b|^{4/d}
                 I_{N\lambda}\widetilde\varepsilon_1
       -\frac4d\Sigma\Theta|Q_b|^{\frac4d-2}
                 I_{N\lambda}\widetilde\varepsilon_2,\\[1mm]
 &R_2(I_{N\lambda}\varepsilon)-R_2(I_{N\lambda}\widetilde\zeta)\\
 &\quad=(\Theta+I_{N\lambda}\varepsilon_2)
             |Q_b+I_{N\lambda}\varepsilon|^{4/d}
       -(\Theta+I_{N\lambda}\widetilde\zeta_{\mathrm{im}})
             |Q_b+I_{N\lambda}\widetilde\zeta|^{4/d}\\
 &\qquad-\left(1+\frac{4\Theta^2}{d|Q_b|^2}\right)|Q_b|^{4/d}
                 I_{N\lambda}\widetilde\varepsilon_2
       -\frac4d\Sigma\Theta|Q_b|^{\frac4d-2}
                 I_{N\lambda}\widetilde\varepsilon_1.
\end{align*}
By \eqref{equ:DFholderall} and \eqref{equ:Rpoint},
\begin{align*}
 &|R_1(I_{N\lambda}\varepsilon)-R_1(I_{N\lambda}\widetilde\zeta)|
       +|R_2(I_{N\lambda}\varepsilon)-R_2(I_{N\lambda}\widetilde\zeta)|\\
 &\quad\lesssim
 \big(|Q_b|+|I_{N\lambda}\widetilde\zeta|
                    +|I_{N\lambda}\widetilde\varepsilon|\big)^{\frac4d-\nu}
 \big(|I_{N\lambda}\widetilde\zeta|^\nu
                    +|I_{N\lambda}\widetilde\varepsilon|^\nu\big)
 |I_{N\lambda}\widetilde\varepsilon|.
\end{align*}
Sobolev's inequality, the bootstrap bounds and
Lemma~\ref{lem:linoutrad} therefore imply
\begin{align*}
 &|(R_1(I_{N\lambda}\varepsilon)-R_1(I_{N\lambda}\widetilde\zeta),y_j\Theta)|
       +|(R_2(I_{N\lambda}\varepsilon)-R_2(I_{N\lambda}\widetilde\zeta),y_j\Sigma)|\\
 &\quad\leq\delta_0\left(
       \left\lVert\nabla I_{N\lambda}\widetilde\varepsilon\right\rVert_{L^2}
       +\left(\int|\widetilde\varepsilon|^2e^{-|y|}\right)^{1/2}\right)
       +C\Gamma_b^{10}.
\end{align*}
The differences $R_1(\varepsilon)-R_1(I_{N\lambda}\varepsilon)$ and
$R_2(\varepsilon)-R_2(I_{N\lambda}\varepsilon)$ are estimated by
\eqref{eq:17}. The errors from the linear terms are controlled by
\eqref{equ:profileapprox}, and the remaining modulation terms by
\eqref{source:roughmodulation}. Together with
\eqref{source:radiativemomentum}, these estimates give
\begin{equation}\label{source:radiativetranslation}
 \left|\frac{x_s}{\lambda}\right|
 \leq\delta_0\left(\left\lVert\nabla I_{N\lambda}\widetilde\varepsilon\right\rVert_{L^2}
       +\left(\int|\widetilde\varepsilon|^2e^{-|y|}\right)^{1/2}\right)
       +C\Gamma_b^{10}+F(s).
\end{equation}

The translation term in the derivative of \eqref{eq:23} is
\begin{equation}\label{source:radiativetranslationterm}
 -\frac{x_s}{\lambda}\cdot
 \big\{(\widetilde\varepsilon_2,\nabla\Lambda\Sigma)
       -(\widetilde\varepsilon_1,\nabla\Lambda\Theta)
       +(\widetilde\varepsilon_2,\nabla\Lambda\widetilde\zeta_{\mathrm{re}})
       -(\widetilde\varepsilon_1,\nabla\Lambda\widetilde\zeta_{\mathrm{im}})\big\}.
\end{equation}
For the first two scalar products,
\[
 |(\widetilde\varepsilon_2,\nabla\Lambda\Sigma)|
       +|(\widetilde\varepsilon_1,\nabla\Lambda\Theta)|
 \lesssim\left(\int|\widetilde\varepsilon|^2e^{-|y|}\right)^{1/2}.
\]
For the other two, integration by parts and \eqref{equ:profileapprox} give
\begin{align*}
 |(\widetilde\varepsilon_2,\nabla\Lambda\widetilde\zeta_{\mathrm{re}})|
 &\leq\left\lVert\nabla I_{N\lambda}\widetilde\varepsilon_2\right\rVert_{L^2}
       \left\lVert\Lambda\widetilde\zeta_{\mathrm{re}}\right\rVert_{L^2}
       +CA^C(N\lambda)^{-2},\\
 |(\widetilde\varepsilon_1,\nabla\Lambda\widetilde\zeta_{\mathrm{im}})|
 &\leq\left\lVert\nabla I_{N\lambda}\widetilde\varepsilon_1\right\rVert_{L^2}
       \left\lVert\Lambda\widetilde\zeta_{\mathrm{im}}\right\rVert_{L^2}
       +CA^C(N\lambda)^{-2}.
\end{align*}
Combining these bounds with \eqref{source:radiativetranslation} yields
\begin{equation}\label{source:radiativetranslationabsorb}
\begin{aligned}
 &\left|\frac{x_s}{\lambda}\cdot
 \big\{(\widetilde\varepsilon_2,\nabla\Lambda\Sigma)
       -(\widetilde\varepsilon_1,\nabla\Lambda\Theta)\right.\left.
       +(\widetilde\varepsilon_2,\nabla\Lambda\widetilde\zeta_{\mathrm{re}})
       -(\widetilde\varepsilon_1,\nabla\Lambda\widetilde\zeta_{\mathrm{im}})
       \big\}\right|\\
 &\quad\leq\delta_0\left(\int|\nabla I_{N\lambda}\widetilde\varepsilon|^2
                         +\int|\widetilde\varepsilon|^2e^{-|y|}\right)
             +C\Gamma_b^{1+c_*}+CF(s)
\end{aligned}
\end{equation}
for some fixed $c_*>0$. Moreover, Lemma~\ref{lem:linoutrad} gives
$\int|\widetilde\zeta|^2e^{-|y|}\lesssim\Gamma_b^{1+c_*}$.

All the additional errors involving the profiles and the nonlinear commutator satisfy, on $J_{k,j}$,
\[
 \int_{J_{k,j}}F(s)\,ds
 \lesssim\sup_{t_k\leq t\leq t_{k+1}}A(t)^C
            \bigl([N(t_k)\lambda(t_k)]^{-2}+[N(t_k)\lambda(t_k)]^{-(1+\nu)}\bigr).
\]
Since $A(t)^C\lesssim k^{C'}$ in the bootstrap regime, the same summation as in
\eqref{eq:9} applies. The remaining estimates are those of
\cite[proof of Lemma~3.8]{SunZheng}. This proves \eqref{eq:22}.
\end{proof}
The term $\int_A^{2A}|\varepsilon|^2$ in \eqref{eq:22} is controlled by
the flux of the $L^2$ norm, as in \cite[Lemma~4.5]{CR09}. Let $\psi(r)$
be a smooth nonnegative radial cutoff
such that $\psi(r)=0$ for $r\leq\frac12$, $\psi(r)=1$ for $r\geq3$,
$\frac14\leq\psi'(r)\leq\frac12$ for $1\leq r\leq2$, $\psi'(r)\geq0$.
We then set
\[
 \psi_A(s,r)=\psi\!\left(\frac r{A(s)}\right),
\]
$A(s)$ given by \eqref{eq:20}. Moreover, we restrict  the choice
of the parameters $(\eta,a)$ by assuming $a>C\eta$ (see Remark~8 in~\cite{MR06JAMS}).

\begin{lemma}[$L^2$ dispersion at infinity in space]\label{lem:4.5}
For some universal constants $C,c_3>0$ and $s$ large enough, it holds
\begin{equation}\label{eq:24}
\begin{aligned}
 \left\{\int\psi_A|\varepsilon|^2\right\}_s
 &\geq c_3b\int_A^{2A}|\varepsilon|^2-\Gamma_b^{1+Ca}-\Gamma_b^{a/2}\int|\nabla I_{N\lambda}\varepsilon|^2-F(s)
              -\mathcal H_s(s),
\end{aligned}
\end{equation}
where $F(s)$ satisfies \eqref{eq:9}, and
\begin{equation}\label{source:Hsmall}
 |\mathcal H(s)|\lesssim\lambda(s)^{\alpha_3}
 \qquad\text{for some }\alpha_3>0.
\end{equation}
\end{lemma}

\begin{proof}
We set
\begin{equation}\label{source:Hdefinition}
 \mathcal H(s)=\int\psi_A(y)
 \left(|I_{N(s)\lambda(s)}(Q_b+\varepsilon)(s,y)|^2
                            -|\varepsilon(s,y)|^2\right)\,dy;
\end{equation}
as defined in \cite[Lemma~6.7]{FanMendelson2024}. Since $\psi_AQ_b=0$ (as $b$ is sufficient small), similar to \cite{FanMendelson2024} we have
\begin{align*}
 \mathcal H(s)
 &=\int\psi_A\left(|I_{N\lambda}\varepsilon|^2-|\varepsilon|^2\right)\int\psi_A\left(|I_{N\lambda}(Q_b+\varepsilon)|^2
                            -|Q_b+I_{N\lambda}\varepsilon|^2\right).
\end{align*}
By \eqref{equ:profileapprox}, \eqref{equ:L2tailI} and the bootstrap bounds,
\begin{align*}
 |\mathcal H(s)|
 &\lesssim(N\lambda)^{-1}
       \big(\left\lVert\nabla I_{N\lambda}\varepsilon\right\rVert_{L^2}
           +\left\lVert\varepsilon\right\rVert_{L^2}\big)
            \left\lVert\varepsilon\right\rVert_{L^2}+(N\lambda)^{-2}\lesssim(N\lambda)^{-1}=\lambda(s)^{(1-\beta)/\beta}.
\end{align*}
This proves \eqref{source:Hsmall}.

Recalling that $I_{N(t)}u(t)=I_Nu$ satisfies the equation
\begin{equation}\label{eq:25}
 i\partial_t(I_Nu)+\Delta(I_Nu)
 =i\widetilde I_Nu-I_N[|u|^{4/d}u],
\end{equation}
where
\begin{equation}\label{eq:26}
 I_N=\frac{N_t}{N}\dot I_N,\quad\dot I_N\text{ is defined in }\eqref{equ:Idot}.
\end{equation}
The scaling identity and \eqref{source:Hdefinition} give
\begin{equation}\label{source:masscorrectionidentity}
\begin{aligned}
 \int\psi_A|\varepsilon|^2+\mathcal H(s)
 &=\int\psi_A|I_{N\lambda}(Q_b+\varepsilon)|^2\\
 &=\int\psi\!\left(\frac{x-x(t)}{A\lambda(t)}\right)|I_Nu(t,x)|^2\,dx.
\end{aligned}
\end{equation}
Using \eqref{eq:25}, the derivative is
\begin{equation}\label{source:massderivativeexact}
\begin{aligned}
 &\frac12\frac d{ds}\int
        \psi\!\left(\frac{x-x(t)}{A\lambda(t)}\right)|I_Nu|^2\\
 &\quad=\frac12\int\frac{\partial\psi_A}{\partial s}
                      |I_{N\lambda}(Q_b+\varepsilon)|^2
       +\frac b2\int y\cdot\nabla\psi_A
                      |I_{N\lambda}(Q_b+\varepsilon)|^2\\
 &\qquad+\operatorname{Im}\int\nabla\psi_A\cdot
            \nabla I_{N\lambda}(Q_b+\varepsilon)\,
            \overline{I_{N\lambda}(Q_b+\varepsilon)}\\
 &\qquad-\frac12\left(\frac{\lambda_s}{\lambda}+b\right)
           \int y\cdot\nabla\psi_A|I_{N\lambda}(Q_b+\varepsilon)|^2\\
 &\qquad-\frac12\frac{x_s}{\lambda}\cdot
           \int\nabla\psi_A|I_{N\lambda}(Q_b+\varepsilon)|^2\\
 &\qquad+\lambda^2\operatorname{Re}\int
        \psi\!\left(\frac{x-x(t)}{A\lambda(t)}\right)
              \widetilde I_Nu\,\overline{I_Nu}\\
 &\qquad+\lambda^2\operatorname{Im}\int
        \psi\!\left(\frac{x-x(t)}{A\lambda(t)}\right)
        \big[I_Nu|I_Nu|^{4/d}-I_N(u|u|^{4/d})\big]\overline{I_Nu}.
\end{aligned}
\end{equation}
The terms involving $Q_b$ are controlled by its regularity and support.
Since $\psi'\geq0$ and $\psi'$ is supported in $[1/2,3]$,
\[
 \left|\int\nabla\psi_A|I_{N\lambda}(Q_b+\varepsilon)|^2\right|
 \leq\frac2A\int y\cdot\nabla\psi_A|I_{N\lambda}(Q_b+\varepsilon)|^2.
\]
Estimate \eqref{source:roughmodulation} therefore gives
\begin{equation}\label{source:fluxtranslation}
\begin{aligned}
 \left|\frac{x_s}{\lambda}\cdot
       \int\nabla\psi_A|I_{N\lambda}(Q_b+\varepsilon)|^2\right|
 &\leq\frac b{100}\int y\cdot\nabla\psi_A
                    |I_{N\lambda}(Q_b+\varepsilon)|^2+CF(s).
\end{aligned}
\end{equation}
Moreover,
\[
 \frac{\partial\psi_A}{\partial s}
 =-\frac{A_s}{A}y\cdot\nabla\psi_A,
 \qquad \frac{A_s}{A}=-\frac{2ab_s}{b^2}.
\]
Together with \eqref{source:roughmodulation}, this yields
\begin{align*}
 &\left|\frac{\lambda_s}{\lambda}+b+\frac{A_s}{A}\right|
       \int y\cdot\nabla\psi_A|I_{N\lambda}(Q_b+\varepsilon)|^2\\
 &\quad\leq\frac b{100}\int y\cdot\nabla\psi_A
                    |I_{N\lambda}(Q_b+\varepsilon)|^2+Cb^{-2}F(s).
\end{align*}
For the term containing the spatial derivative, Cauchy--Schwarz gives
\begin{align*}
 &2\left|\operatorname{Im}\int\nabla\psi_A\cdot
       \nabla I_{N\lambda}(Q_b+\varepsilon)\,
       \overline{I_{N\lambda}(Q_b+\varepsilon)}\right|\\
 &\quad\leq\frac b{100}\int y\cdot\nabla\psi_A
                    |I_{N\lambda}(Q_b+\varepsilon)|^2
       +\frac{C}{bA^2}\int|\nabla I_{N\lambda}\varepsilon|^2+\Gamma_b^{10}\\
 &\quad\leq\frac b{100}\int y\cdot\nabla\psi_A
                    |I_{N\lambda}(Q_b+\varepsilon)|^2
       +\Gamma_b^{a/2}\int|\nabla I_{N\lambda}\varepsilon|^2+\Gamma_b^{10}.
\end{align*}
Finally, \eqref{equ:L2tailI} and the support of $Q_b$ imply
\[
 \int y\cdot\nabla\psi_A|I_{N\lambda}(Q_b+\varepsilon)|^2
 \geq\frac18\int_A^{2A}|\varepsilon|^2
       -C(N\lambda)^{-2}\int|\nabla I_{N\lambda}\varepsilon|^2
       -C\Gamma_b^{10}.
\]
It follows that
\begin{equation}\label{eq:27}
\begin{aligned}
 &\frac d{ds}\int
       \psi\!\left(\frac{x-x(t)}{A\lambda(t)}\right)|I_Nu|^2\\
 &\quad\geq c_3b\int_A^{2A}|\varepsilon|^2-\Gamma_b^{1+Ca}
       -\Gamma_b^{a/2}\int|I_{N\lambda}\nabla\varepsilon|^2-Cb^{-2}F(s)\\
 &\qquad-2\lambda^2\left|\operatorname{Im}\int
       \psi\!\left(\frac{x-x(t)}{A\lambda(t)}\right)\overline{I_Nu}
       \big[I_N(u|u|^{4/d})-I_Nu|I_Nu|^{4/d}\big]\right|\\
 &\qquad-2\lambda^2\left|\int
       \psi\!\left(\frac{x-x(t)}{A\lambda(t)}\right)
                      \widetilde I_Nu\,\overline{I_Nu}\right|.
\end{aligned}
\end{equation}
It remains to prove that the last two terms satisfy \eqref{eq:9}.

For the term containing $\widetilde I_N$, estimates \eqref{eq:7},
\eqref{equ:ntass} and \eqref{equ:Idot} give
\begin{equation}\label{source:movingI}
\begin{aligned}
 \lambda^2\left\lVert\widetilde I_Nu\right\rVert_{L^2}
 &\lesssim\left|\frac{N_s}{N}\right|(N\lambda)^{-1}
       \left\lVert\langle\nabla\rangle I_{N\lambda}(Q_b+\varepsilon)
                                      \right\rVert_{L^2}\\
 &\lesssim(1+F(s))(N\lambda)^{-1}.
\end{aligned}
\end{equation}
By conservation of the $L^2$ norm,
\begin{align*}
 &\int_{s_1}^{s^+}\lambda^2\left|\int
       \psi\!\left(\frac{x-x(t)}{A\lambda(t)}\right)
                    \widetilde I_Nu\,\overline{I_Nu}\right|\,ds\\
 &\quad\lesssim\int_{s_1}^{s^+}F(s)\,ds
       +\sum_{k=k_1}^{k^+}\sum_{j=1}^{J_k}
         \int_{\tau_k^j}^{\tau_k^{j+1}}
          \frac1{N(t)\lambda(t)}\frac{dt}{\lambda(t)^2}\\
 &\quad\lesssim\lambda(s_1)^{\alpha_2}
       +\sum_{k=k_1}^{k^+}\sum_{j=1}^{J_k}
         \frac{\tau_k^{j+1}-\tau_k^j}{\lambda(t_k)^2}
         \frac1{N(t_k)\lambda(t_k)}\\
 &\quad\lesssim\lambda(s_1)^{\alpha_2}
       +\sum_{k=k_1}^{+\infty}k[N(t_k)\lambda(t_k)]^{-1}
       \lesssim\lambda(s_1)^{\alpha_2},
\end{align*}
using \eqref{eq:9}, \eqref{equ:ntass}, Lemma~\ref{lem:numberintervals}
and \eqref{eq:16}, after decreasing $\alpha_2$.

For the nonlinear term, Lemma~\ref{lem:xianchafa} and
Remark~\ref{remark2} imply
\begin{align*}
 &\int_{s_1}^{s^+}\lambda^2\left|\operatorname{Im}\int
       \psi\!\left(\frac{x-x(t)}{A\lambda(t)}\right)\overline{I_Nu}
       \big[I_N(u|u|^{4/d})-I_Nu|I_Nu|^{4/d}\big]\right|\,ds\\
 &\quad=\int_{t_1}^{T^+}\left|\operatorname{Im}\int
       \psi\!\left(\frac{x-x(t)}{A\lambda(t)}\right)\overline{I_Nu}
       \big[I_N(u|u|^{4/d})-I_Nu|I_Nu|^{4/d}\big]\right|\,dt\\
 &\quad\lesssim\sum_{k=k_1}^{k^+}\sum_{j=1}^{J_k}
       \left\lVert I_{N(t)}(u|u|^{4/d})
                 -I_{N(t)}u|I_{N(t)}u|^{4/d}\right\rVert_
             {L_t^2L_x^{\frac{2d}{d+2}}([\tau_k^j,\tau_k^{j+1}])}\\
 &\hspace{38mm}\times
       \left\lVert\psi\!\left(\frac{x-x(t)}{A\lambda(t)}\right)I_{N(t)}u
                                  \right\rVert_
             {L_t^2L_x^{\frac{2d}{d-2}}([\tau_k^j,\tau_k^{j+1}])}\\
 &\quad\lesssim\sum_{k=k_1}^{k^+}\sum_{j=1}^{J_k}
                       [N(t_k)\lambda(t_k)]^{-(1+\nu)}\\
 &\quad\lesssim\sum_{k=k_1}^{+\infty}k[N(t_k)\lambda(t_k)]^{-(1+\nu)}
       \lesssim\lambda(s_1)^{\alpha_2}.
\end{align*}
The term $b^{-2}F$ satisfies the same estimate by \eqref{source:weightedF}.
Enlarging $F$ and using \eqref{source:masscorrectionidentity} proves
\eqref{eq:24}.
\end{proof}

\subsection{Proof of the Bootstrap Lemma \ref{lem:bootstrap}}
The bootstrap argument follows \cite[Section~4.3]{CR09}, using the
virial estimates and the flux of the $L^2$ norm established above.

\begin{proof}[Proof of Lemma \ref{lem:bootstrap}]
\par\medskip\noindent\textbf{Step 1:} Lyapounov control of the dynamics.\par\smallskip
As in \cite[Proposition~4]{MR06JAMS} and \cite[Section~4.3]{CR09},
combining \eqref{eq:19}, \eqref{eq:22}, \eqref{eq:24} and conservation
of the $L^2$ norm gives
\begin{equation}\label{eq:28}
\begin{aligned}
 \{\mathcal{J}\}_s
 &\leq-Cb\left[\Gamma_b+\Xi(s)+\int|\nabla I_{N(s)\lambda(s)}\widetilde\varepsilon(s)|^2
    +\int|\varepsilon(s)|^2e^{-|y|}+\int_A^{2A}|\varepsilon(s)|^2\right]\\
 &\quad+F(s)
\end{aligned}
\end{equation}
with $F(s)$ satisfying \eqref{eq:9} and:
\begin{equation}\label{eq:29}
\begin{aligned}
 \mathcal{J}(s)&=\left(\int|Q_b|^2-\int Q^2\right)
 +2(\varepsilon_1,\Sigma)+2(\varepsilon_2,\Theta)+\int(1-\psi_A)|\varepsilon|^2
 -\mathcal H(s)\\
 &\quad-c\left(b\widetilde f_1(b)-\int_0^b\widetilde f_1(v)\,dv
       +b\{(\varepsilon_2,\Lambda\widetilde\zeta_{\mathrm{re}})
          -(\varepsilon_1,\Lambda\widetilde\zeta_{\mathrm{im}})\}\right).
\end{aligned}
\end{equation}
Here $c=\delta_1/800>0$ denotes some small enough universal constant and:
\begin{equation}\label{eq:30}
 \widetilde f_1(b)=\frac b4|yQ_b|_2^2
 +\frac12\operatorname{Im}\!\left(\int(y\cdot\nabla\widetilde\zeta)\overline{\widetilde\zeta}\right).
\end{equation}
By conservation of mass and \eqref{source:masscorrectionidentity},
\begin{equation}\label{source:Jderivative}
\begin{aligned}
 \mathcal J_s
 &=-\frac d{ds}\left(\int\psi_A|\varepsilon|^2+\mathcal H(s)\right)
       -cb(f_1)_s\\
 &\quad-cb_s\big\{(\varepsilon_2,\Lambda\widetilde\zeta_{\mathrm{re}})
                 -(\varepsilon_1,\Lambda\widetilde\zeta_{\mathrm{im}})\big\}.
\end{aligned}
\end{equation}
The last term is estimated as in \cite[Proposition~4]{MR06JAMS}, using
\[
 \big|(\varepsilon_2,\Lambda\widetilde\zeta_{\mathrm{re}})
        -(\varepsilon_1,\Lambda\widetilde\zeta_{\mathrm{im}})\big|
 \lesssim\Gamma_b^{1/2-C\eta}
\]
and \eqref{eq:7}. For the gradient term in \eqref{eq:24},
\begin{equation}\label{source:weightedradiationgradient}
\begin{aligned}
 \Gamma_b^{a/2}\int|\nabla I_{N\lambda}\varepsilon|^2
 &\leq2\Gamma_b^{a/2}\int|\nabla I_{N\lambda}\widetilde\varepsilon|^2
          +C\Gamma_b^{1+a/2-C\eta}.
\end{aligned}
\end{equation}
Since $a/2>C\eta$, the right hand side is absorbed into the
positive $b$-weighted terms, after taking $b(0)$ sufficiently small.
This proves \eqref{eq:28}, using \eqref{source:weightedF} for the remaining
weighted error terms.

Moreover, \eqref{source:Hsmall} and \eqref{equ:ltass} give
\begin{equation}\label{source:masscomparison}
 |\mathcal H(s)|\lesssim\lambda(s)^{\alpha_3}
                       \lesssim\Gamma_b^{10}.
\end{equation}
Now as a consequence of the control of the modified energy \eqref{eq:5} and the
coercivity of the linearized energy under the chosen set of orthogonality conditions,
together with \eqref{source:masscomparison}, we have the bounds:
\begin{equation}\label{eq:31}
\begin{aligned}
 \mathcal{J}(s)-f_2(b(s))
 &\geq-\Gamma_b^{1-Ca}+\frac1C\left(\Xi(s)+\int|I_{N\lambda}\nabla\varepsilon|^2
                                          +\int|\varepsilon|^2e^{-|y|}\right),\\
 \mathcal{J}(s)-f_2(b(s))
 &\leq CA^2\left(\Xi(s)+\int|I_{N\lambda}\nabla\varepsilon|^2+\int|\varepsilon|^2e^{-|y|}\right)
       +\Gamma_b^{1-Ca},
\end{aligned}
\end{equation}
where $f_2$ is given by
\[
 f_2(b)=\left(\int|Q_b|^2-\int Q^2\right)
 -\frac{\delta_1}{800}\left(b\widetilde f_1(b)-\int_0^b\widetilde f_1(v)\,dv\right)
\]
and satisfies
\begin{equation}\label{eq:32}
 \frac{d_0}C<\left.\frac{df_2}{db^2}\right|_{b^2=0}<Cd_0,
\end{equation}
with $d_0$ given by \eqref{equ:qbcrimass}.

\par\medskip\noindent\textbf{Step 2:} Pointwise bound on $\varepsilon$.\par\smallskip

The proof of the pointwise bound \eqref{equ:lamtimpse} follows
\cite[Section~4.3, Step~2]{CR09}. It remains to show that
\begin{equation}\label{eq:33}
 \Xi(s)+\int|I_{N\lambda}\nabla\varepsilon(s)|^2+\int|\varepsilon(s)|^2e^{-|y|}
 \lesssim\Gamma_{b(s)}^{2/3}.
\end{equation}
We may assume that the same function $F$ satisfying \eqref{eq:9} appears on the
right hand side of \eqref{eq:19} and \eqref{eq:28}. Let $s_2\in[s_0,s^+]$ and
consider the function $\widetilde b(s)=b(s)+\int_{s_0}^sF(\sigma)\,d\sigma$.
If $\widetilde b_s(s_2)\leq0$, then \eqref{eq:33} follows directly from \eqref{eq:19}.
If $\widetilde b_s(s_2)>0$, let $s_1\in[s_0,s^+]$ be the first time backwards from
$s_2$ such that $\widetilde b_s(s_1)=0$, then either $s_1$ is attained or $s_1=s_0$.
In both cases, we have using \eqref{equ:ncsma} and \eqref{equ:thetadef}:
\[
 \Xi(s_1)+\int|I_{N(s_1)\lambda(s_1)}\nabla\varepsilon(s_1)|^2
 +\int|\varepsilon(s_1)|^2e^{-|y|}\leq\Gamma_{b(s_1)}^{35/48}
\]
and thus
\begin{equation}\label{eq:34}
 \mathcal{J}(s_1)-f_2(b(s_1))\leq\Gamma_{b(s_1)}^{17/24}
\end{equation}
from \eqref{eq:31} and for $a>0$ in \eqref{eq:20} small enough. Moreover,
$\widetilde b_s\geq0$ on $[s_1,s_2]$ and thus:
\[
 b(s_2)+\int_{s_1}^{s_2}F(s)\,ds
 \geq b(s_1)+\int_{s_1}^{s_1}F(s)\,ds.
\]
This implies from \eqref{eq:9}:
\begin{equation}\label{eq:35}
 b(s_2)\geq b(s_1)-\int_{s_1}^{s_2}F(s)\,ds
 \geq b(s_1)-\lambda^{\alpha_2}(s_1)\geq b(s_1)-\Gamma_{b(s_1)}
\end{equation}
where we used \eqref{equ:ltass} in the last step. We now use the Lyapounov control
\eqref{eq:28} to derive:
\[
 \mathcal{J}(s_2)-\int_{s_1}^{s_2}F(s)\,ds\leq\mathcal{J}(s_1)
\]
and thus
\[
 \mathcal{J}(s_2)\leq\mathcal{J}(s_1)+\int_{s_1}^{s^+}F(s)\,ds
 \leq\mathcal{J}(s_1)+\lambda^{\alpha_2}(s_1)\leq\mathcal{J}(s_1)+\Gamma_{b(s_1)}.
\]
We then inject \eqref{eq:31}, \eqref{eq:35} and \eqref{eq:34} to conclude:
\begin{align*}
 &f_2(b(s_2))+\frac1C\left(\Xi(s_2)+\int|I_{N\lambda}\nabla\varepsilon(s_2)|^2
                                      +\int|\varepsilon(s_2)|^2e^{-|y|}\right)\\
 &\quad\leq\mathcal{J}(s_2)+\Gamma_{b(s_2)}^{1-Ca}
 \leq\mathcal{J}(s_1)+\Gamma_{b(s_2)}^{1-Ca}+\Gamma_{b(s_1)}
 \leq f_2(b(s_1))+\Gamma_{b(s_2)}^{2/3}.
\end{align*}
The monotonicity \eqref{eq:32} of $f_2$ in $b$ and \eqref{eq:35} now imply:
\[
 \Xi(s_2)+\int|I_{N\lambda}\nabla\varepsilon(s_2)|^2+\int|\varepsilon(s_2)|^2e^{-|y|}
 \lesssim\Gamma_{b(s_2)}^{2/3}
\]
which implies \eqref{eq:33} at $s_2$. This concludes the proof of \eqref{equ:lamtimpse}.

\par\medskip\noindent\textbf{Step 3:} Upper bound on blowup rate.\par\smallskip
Let
\[
 \widehat b(s)=b(s)-\int_s^{s^+}F(s)\,ds
\]
then from \eqref{eq:9} and \eqref{equ:ltass},
\[
 \frac9{10}b(s)\leq\widehat b(s)\leq\frac{10}9b(s)
\]
and thus \eqref{eq:19} implies:
\begin{equation}\label{eq:36}
 \widehat b_s\geq-C\Gamma_{\widehat b(s)}^{1-C\eta}.
\end{equation}
This implies:
\[
 \left(e^{\frac\pi{2\widehat b(s)}}\right)_s
 \leq e^{\frac\pi{2\widehat b(s)}}
       \frac{\pi\Gamma_{\widehat b}^{1-C\eta}}{2\widehat b^2}\leq1
\]
as $\Gamma_b\sim e^{-\pi/b}$, and therefore
\[
 e^{\frac\pi{2\widehat b(s)}}\leq e^{\frac\pi{2\widehat b(s_0)}}+s-s_0\leq s
\]
from \eqref{eq:1} and \eqref{eq:36}. Finally, $\forall\,s\in[s_0,s^+)$,
\begin{equation}\label{eq:37}
 \widehat b(s)\geq\frac\pi{2\log s}\quad\text{and thus}\quad
 b(s)\geq\frac{9\pi}{20\log s}.
\end{equation}
We now rewrite the estimate \eqref{eq:7} using \eqref{eq:33}:
\begin{equation}\label{eq:38}
 \left|\frac{\lambda_s}\lambda+b\right|\leq\Gamma_b^{1/4}+F(s).
\end{equation}
We integrate this in time and get: $\forall\,s\in[s_0,s^+]$,
\begin{align*}
 -\log\lambda(s)
 &\geq-\log\lambda(s_0)+\frac23\int_{s_0}^sb-\int_{s_0}^sF(s)\,ds\\
 &\geq-\log\lambda(s_0)+\frac\pi4\left(\frac s{\log s}-\frac{s_0}{\log s_0}\right)-1
\end{align*}
where we used \eqref{eq:9} in the last step. Now from \eqref{equ:b0l0} and \eqref{eq:1},
\[
 -\log\lambda(0)\geq e^{\frac{2\pi}{3b(s_0)}}=s_0^{6/5}
\]
and thus
\begin{equation}\label{eq:39}
 -\log\lambda(s)\geq-\frac23\log\lambda(0)+\frac\pi4\frac s{\log s},
 \quad\text{i.e.}\quad
 \lambda(s)\leq\lambda^{2/3}(0)e^{-\frac\pi4\frac s{\log s}}.
\end{equation}
This also implies: $\forall\,s\in[s_0,s^+)$,
\begin{equation}\label{eq:40}
 -\log\lambda(s)\geq\frac\pi4\frac s{\log s}\geq\sqrt s
\end{equation}
and taking the log of this inequality yields
\begin{equation}\label{eq:41}
 \log|\log\lambda(s)|\geq\frac12\log s,\quad\text{i.e.}\quad
 b\geq\frac\pi{10\log|\log\lambda(s)|}
\end{equation}
using \eqref{eq:37}. Therefore \eqref{equ:lamtimp} is proved.

\par\medskip\noindent\textbf{Step 4:} Monotonicity of $\lambda$.\par\smallskip
We now turn to the proof of \eqref{equ:lmt2imp}, \eqref{equ:tktk-1}. From \eqref{equ:lamtimp}, \eqref{eq:7},
\eqref{eq:33}, there holds:
\[
 -\frac{\lambda_s}\lambda\geq\frac b2-F(s)
\]
and thus: $\forall\,s_1,s_2\in[s_0,s^+]$,
\begin{equation}\label{eq:42}
 -\log\!\left(\frac{\lambda(s_2)}{\lambda(s_1)}\right)
 \geq\frac12\int_{s_1}^{s_2}b-\int_{s_1}^{s_2}F(s)\,ds
 \geq\frac9{20}\int_{s_1}^{s_2}\frac{ds}{\log(s)}-\frac1{1000}
\end{equation}
from \eqref{eq:9} and \eqref{equ:lmt2imp} follows. To prove \eqref{equ:tktk-1}, we let $[t_k,t_{k+1}]$
be a doubling time interval, then from \eqref{eq:42}:
\[
 \log(2)=-\log\!\left(\frac{\lambda(t_{k+1})}{\lambda(t_k)}\right)
 \geq\frac9{20}\int_{t_k}^{t_{k+1}}\frac{dt}{\lambda^2(t)\log(s(t))}-\frac1{1000}
\]
and thus
\[
 1\geq\frac{C(t_{k+1}-t_k)}{\lambda^2(t_k)\log(s(t_{k+1}))}
 \geq\frac{C(t_{k+1}-t_k)}{\lambda^2(t_k)\log|\log(\lambda(t_{k+1}))|}
 \geq\frac{C(t_{k+1}-t_k)}{\lambda^2(t_k)\log k}
\]
and \eqref{equ:tktk-1} follows.

\par\medskip\noindent\textbf{Step 5:} Control of $b$ and $\|\varepsilon(t)\|_{L^2}$.\par\smallskip
The upper bound \eqref{equ:btimp} on $b$ and $\|\varepsilon(t)\|_{L^2}$ is a direct consequence of the
conservation of the $L^2$ norm and \eqref{equ:qbcrimass}, and is left to the reader. The lower bound
$b>0$ follows from \eqref{eq:41}.
This concludes the proof of the bootstrap Lemma~\ref{lem:bootstrap}.
\end{proof}

\subsection{Proof of Proposition \ref{prop:main}}
The proof of Proposition~\ref{prop:main} is completed as in
\cite[Section~4.4]{CR09}.

\par\medskip\noindent\textbf{Step 1:} Finite time blow-up.\par\smallskip
From Lemma~\ref{lem:bootstrap}, the bootstrap estimates hold for all $t\in[0,T^+)$,
$T^+=T\in[0,+\infty]$. Now observe from \eqref{eq:1} and \eqref{eq:40} that
\[
 T=\int_{s_0}^{+\infty}\lambda^2(s)\,ds
 \leq\lambda^{4/3}(0)\int_2^{+\infty}e^{-\frac\pi4\frac s{\log s}}\,ds<+\infty
\]
and hence blowup occurs in finite time. Moreover, from \eqref{equ:lamtimpse},
\[
 \left\lVert u(t)\right\rVert_{H^s}\sim\frac1{\lambda^s(t)}
\]
and thus the local well-posedness theory in $H^s$ ensures $\lambda(t)\to0$ as
$t\to T$. Let us recall the standard scaling lower bound on blowup rate:
\[
 \left\lVert u(t)\right\rVert_{H^s}\geq\frac C{(T-t)^{s/2}}
\]
which implies
\[
 \lambda^2(t)\lesssim(T-t)
\]
for $t$ close enough to $T$, and thus:
\[
 s(t)=\int_0^t\frac{d\tau}{\lambda^2(\tau)}+s_0\to+\infty\quad\text{as }t\to T.
\]

\par\medskip\noindent\textbf{Step 2:} Convergence of the concentration point.\par\smallskip
The convergence of the concentration point is a consequence of \eqref{eq:8},
\eqref{eq:9}, \eqref{eq:40}, \eqref{eq:33}. In particular, \eqref{eq:8} gives
\begin{equation}\label{source:translationfinal}
 \left|\frac{x_s}{\lambda}\right|
        \lesssim\Gamma_b^{1/3}+F(s),\qquad
 x_t=\frac{x_s}{\lambda^2}
        =\frac1\lambda\frac{x_s}{\lambda}.
\end{equation}
These imply:
\[
 \int_{s_0}^{+\infty}|x_s|\,ds
 \leq C\int_{s_0}^{+\infty}\lambda(1+F(s))\,ds<+\infty.
\]

\par\medskip\noindent\textbf{Step 3:} Lower bound on the blow-up rate.\par\smallskip
Observe from \eqref{eq:31}, \eqref{eq:9}, \eqref{eq:33} and \eqref{equ:lamtimp} that:
\[
 \widetilde{\mathcal{J}}(s)=\mathcal{J}(s)+\int_s^{+\infty}F(s)\,ds
\]
satisfies
\[
 \frac{b^2(s)}C\leq\widetilde{\mathcal{J}}(s)\leq Cb^2(s).
\]
Together with \eqref{eq:28}, this implies:
\[
 \{\widetilde{\mathcal{J}}\}_s\leq-e^{-\frac C{\sqrt{\widetilde{\mathcal{J}}}}}.
\]
Integrating this in time yields:
\[
 b(s)\leq C\sqrt{\widetilde{\mathcal{J}}}\leq\frac C{\log s}
\]
for $s$ large enough. Integrating now \eqref{eq:38} in time, we conclude
\[
 -\log\lambda(s)\leq C\int_{s_0}^sb+C\leq C\frac s{\log s}
\]
for $s$ large enough, and thus together with \eqref{eq:41}:
\begin{equation}\label{eq:43}
 \frac1C\leq b\log|\log\lambda|\leq C.
\end{equation}
Next, we compute
\begin{equation}\label{eq:44}
\begin{aligned}
 -\left(\lambda^2\log|\log\lambda|\right)_t
 &=-\lambda\lambda_t\log|\log(\lambda)|
       \left(2-\frac1{|\log\lambda|\log|\log\lambda|}\right)\\
 &=-\left(\frac{\lambda_s}\lambda+b\right)\log|\log\lambda|
       \left(2-\frac1{|\log\lambda|\log|\log\lambda|}\right)\\
 &\quad+b\log|\log\lambda|
       \left(2-\frac1{|\log\lambda|\log|\log\lambda|}\right).
\end{aligned}
\end{equation}
From \eqref{eq:38},
\begin{align*}
 &\int_t^T\left|\left(\frac{\lambda_s}\lambda+b\right)\log|\log\lambda|\right|\\
 &\quad\lesssim\int_t^T\bigl[b^2\log|\log\lambda|+F\log|\log\lambda|\bigr]\,dt\\
 &\quad\lesssim\int_t^Tb^2\log|\log\lambda|
       +\int_s^{+\infty}\lambda^2(\sigma)\log|\log\lambda(\sigma)|F(\sigma)\,d\sigma\\
 &\quad\lesssim\int_t^Tb^2\log|\log\lambda|
       +\lambda^2(s)\log|\log\lambda(s)|\int_s^{+\infty}F(\sigma)\,d\sigma\\
 &\quad\lesssim\int_t^Tb^2\log|\log\lambda|+\lambda^2(t)
\end{align*}
where we used the almost monotonicity of $\lambda$ \eqref{equ:lmt2imp} and \eqref{eq:9} in the
last two steps. Injecting this into \eqref{eq:44} integrated from $t$ to $T$ and
using \eqref{eq:43}, we conclude:
\[
 \frac{T-t}C\leq\lambda^2(t)\log|\log\lambda(t)|\leq C(T-t)
\]
from which
\begin{equation}\label{eq:45}
 \frac1C\left(\frac{T-t}{\log|\log(T-t)|}\right)^{1/2}
 \leq\lambda(t)\leq C\left(\frac{T-t}{\log|\log(T-t)|}\right)^{1/2}
\end{equation}
for $t$ close enough to $T$. The exact convergence \eqref{equ:lamdts}
follows as in \cite[Proposition~6]{MR06JAMS}.

\par\medskip\noindent\textbf{Step 4:} Strong $L^2$ convergence of excess mass outside the blowup point.\par\smallskip
The proof follows \cite{MR05CMP} and \cite[Section~4.4, Step~4]{CR09}.
Let $R>0$ and $x(T)=\lim\limits_{t\to T}x(t)$ as in \eqref{equ:xtconx}.
There exists $u^*\in L^2$ such that
\begin{equation}\label{eq:46}
 u(t)\to u^*\quad\text{in }L^2(\mathbb{R}^d\setminus\{|x-x(T)|\leq R\})
 \quad\text{as }t\to T.
\end{equation}
Indeed, we introduce a smooth radial cut-off function $\chi(x)=0$ for $|x|\leq1$
and $\chi(x)=1$ for $|x|\geq2$, and $\chi_R(x)=\chi(\frac{x-x(T)}R)$.
Let us rewrite the geometrical decomposition \eqref{equ:decass} as:
\begin{equation}\label{eq:47}
\begin{aligned}
 u(t,x)&=\frac1{\lambda(t)^{d/2}}(Q_{b(t)}+\widetilde\zeta)
         \left(\frac{x-x(t)}{\lambda(t)}\right)e^{-i\gamma(t)}+v(t,x),\\
 v(t,x)&=\frac1{\lambda(t)^{d/2}}\widetilde\varepsilon\!\left(t,\frac{x-x(t)}{\lambda(t)}\right)e^{-i\gamma(t)}.
\end{aligned}
\end{equation}
Observe that \eqref{eq:28}, applied to $\sqrt{\widetilde{\mathcal{J}}}$ with
$\widetilde{\mathcal{J}}\sim b^2$, implies the crucial dispersive estimate on the excess of mass:
\begin{equation}\label{eq:48}
 \int_0^T\left\lVert \nabla I_{N(t)}v(t)\right\rVert_{L^2}^2\,dt
 =\int_{s_0}^{+\infty}\left\lVert \nabla I_{N(\sigma)\lambda(\sigma)}\widetilde\varepsilon(\sigma)\right\rVert_{L^2}^2
 \,d\sigma<+\infty.
\end{equation}
We first claim the same dispersive control on $u$ outside of $|x-x(T)|<R$:
\begin{equation}\label{eq:49}
 \forall R>0,\quad\int_0^T\left\lVert \nabla[\chi_R I_{N(\tau)}u(\tau)]\right\rVert_{L^2}^2\,d\tau
 \leq C(R).
\end{equation}

In fact, from the space localization of $Q_b$, $\tilde{\zeta}$ and \eqref{equ:lamtimp}, there is a time $t(R)$ such that
\[
 \forall t\in[t(R),T),\quad u(t,x)=v(t,x)\quad\text{for }|x-x(T)|>R.
\]
We then estimate: $\forall t\in[t(R),T)$,
\begin{align*}
 &\left\lVert \chi_R\nabla I_{N(t)}
 \left[\lambda^{-d/2}(Q_b+\widetilde\zeta)
       \left(\frac{x-x(t)}\lambda\right)e^{-i\gamma(t)}\right]\right\rVert_{L^2}^2\\
 &\qquad\lesssim\lambda^{-2}
      \left\lVert (I_{N\lambda}-\mathit{Id})(Q_b+\widetilde\zeta)\right\rVert_{H^1}^2
 \lesssim\lambda^{-2}A^C(N\lambda)^{-4}
\end{align*}
by \eqref{equ:profileapprox} and \eqref{source:highprofiles}.
This implies from Lemma~\ref{lem:numberintervals}:
\begin{align*}
 &\int_{t(R)}^T\left\lVert \chi_R\nabla I_{N(t)}
 \left[\lambda^{-d/2}(Q_b+\widetilde\zeta)
       \left(\frac{x-x(t)}\lambda\right)e^{-i\gamma(t)}\right]\right\rVert_{L^2}^2\,dt\\
 &\qquad\lesssim\sum_{k\geq k_0}k^{C'}[N(t_k)\lambda(t_k)]^{-4}<+\infty.
\end{align*}
\eqref{eq:49} now directly follows from \eqref{eq:48}, the conservation of the
$L^2$ norm and the definition \eqref{eq:21} of $\widetilde\varepsilon$.

We introduce $w_R(t,x)=\chi_R(x)[I_{N(t)}u(t,x)]$, we compute the equation for $w_R$:
\begin{equation}\label{eq:50}
 i\partial_tw_R+\Delta w_R
 =i\chi_R\widetilde I_Nu+2\nabla\chi_R\cdot\nabla I_Nu
    +\Delta\chi_R I_Nu-\chi_R I_N(u|u|^{4/d})
\end{equation}
where $\widetilde I_N$ is defined as in \eqref{eq:26}.

We now write down Duhamel's formula on $[0,t]$ and estimate each term of the right
hand side in $L_{[0,t]}^\infty L_x^2$. For the first term, we have from the energy
estimate and \eqref{equ:Idot}:
\begin{equation}\label{eq:52}
\begin{aligned}
 &\left\lVert \int_0^te^{i(t-\tau)\Delta}\chi_R\widetilde I_Nu\,d\tau\right\rVert_{L_{[0,t]}^\infty L_x^2}\\
 &\quad\lesssim\left\lVert \chi_R\widetilde I_Nu\right\rVert_{L_{[0,t]}^1L_x^2}
 \lesssim\left\lVert \widetilde I_Nu\right\rVert_{L_{[0,t]}^1L_x^2}\\
 &\quad\lesssim\int_{s_0}^{+\infty}\frac{1+F(s)}{N(s)\lambda(s)}\,ds
 \lesssim\int_{s_0}^{+\infty}F(s)\,ds+
             \sum_{k\geq k_0}k[N(t_k)\lambda(t_k)]^{-1}<+\infty.
\end{aligned}
\end{equation}
where we used \eqref{equ:ntass}, \eqref{eq:38}, \eqref{eq:9} and \eqref{eq:39}.
For the second term, we use \eqref{eq:49} which implies:
\begin{equation}\label{eq:53}
 \left\lVert \int_0^te^{i(t-\tau)\Delta}\nabla\chi_R\cdot\nabla I_Nu\,d\tau\right\rVert_{L_{[0,t]}^\infty L_x^2}
 \lesssim\left\lVert \nabla\chi_R\cdot\nabla I_Nu\right\rVert_{L_{[0,t]}^1L_x^2}<+\infty.
\end{equation}
The third term is even easier to estimate and left to the reader. We now estimate
the nonlinear term and first estimate using Strichartz estimate:
\[
 \left\lVert \int_0^te^{i(t-\tau)\Delta}\chi_R I_N(u|u|^{4/d})\,d\tau\right\rVert_{L_{[0,t]}^\infty L_x^2}
 \lesssim\left\lVert \chi_R I_N(u|u|^{4/d})\right\rVert_{L_{[0,t]}^2L_x^{\frac{2d}{d+2}}}.
\]
We now claim:
\begin{equation}\label{eq:54}
 \left\lVert \chi_R I_N(u|u|^{4/d})\right\rVert_{L_{[0,T]}^2L_x^{\frac{2d}{d+2}}}<+\infty.
\end{equation}

\noindent\emph{Proof of \eqref{eq:54}.}
We first have:
\begin{align}\label{eq:55}
 &\left\lVert \chi_R I_N(u|u|^{4/d})\right\rVert_{L_{[0,T]}^2L_x^{\frac{2d}{d+2}}}\\\nonumber
\lesssim&\left\lVert  I_{N(t)}(u|u|^{4/d})-I_{N(t)}u|I_{N(t)}u|^{4/d}\right\rVert_{L_{[0,T]}^2L_x^{\frac{2d}{d+2}}}
  +\left\lVert \chi_R I_Nu|I_Nu|^{4/d}\right\rVert_{L_{[0,T]}^2L_x^{\frac{2d}{d+2}}}.
\end{align}
The first term is estimated using Lemma~\ref{commutor} and
Remark~\ref{remark2}. Let $(\tau_k^j)_{1\leq j\leq J_k}$
be local well-posedness intervals given by Lemma~\ref{lem:hslwp}, then:
\begin{align*}
 &\left\lVert I_{N(t)}(u|u|^{4/d})-I_{N(t)}u|I_{N(t)}u|^{4/d}\right\rVert_{L_{[0,T]}^2L_x^{\frac{2d}{d+2}}}^2\\
 &\qquad\lesssim1+\sum_{k\geq k_0}k[N(t_k)\lambda(t_k)]^{-2(1+\nu)}<+\infty.
\end{align*}
The second term in \eqref{eq:55} is controlled using the dispersive estimate
\eqref{eq:49}:
\[
 \left\lVert \chi_R I_Nu|I_Nu|^{4/d}\right\rVert_{L_x^{\frac{2d}{d+2}}}
 \lesssim_{R,M(u)}\left\lVert \nabla[\chi_{R/2}I_Nu]\right\rVert_{L^2}.
\]
This concludes the proof of \eqref{eq:54}.

From \eqref{eq:52}, \eqref{eq:53} and \eqref{eq:54}, we conclude from standard
argument that $w_R=\chi_R I_Nu$ has a strong $L^2$ limit as $t\to T$.
We now have from rescaling:
\begin{equation}\label{eq:56}
\begin{aligned}
 \left\lVert w_R-\chi_Ru\right\rVert_{L^2}
 &\lesssim\left\lVert (\mathit{Id}-I_N)u\right\rVert_{L^2}\\
 &\lesssim(N\lambda)^{-1}
     \left\lVert \langle\nabla\rangle I_{N\lambda}(Q_b+\varepsilon)\right\rVert_{L^2}
       \longrightarrow0\quad\text{as }t\to T
\end{aligned}
\end{equation}
where we used $\left\lVert \varepsilon(t)\right\rVert_{H^s}\lesssim1$ from \eqref{equcont} and \eqref{equ:vartass}, and
$N(t)\lambda(t)\to+\infty$ as $t\to T$ from \eqref{equ:nlamd}. This concludes the proof of
the existence of an $L^2$ limit $u^*$ outside the blowup point. Note that
$u^*\in L^2(\mathbb{R}^d)$ from the conservation of the $L^2$ norm.
This concludes the proof of \eqref{eq:46}.

\par\medskip\noindent\textbf{Step 5:} Nonconcentration of the $L^2$ norm at the blowup point.\par\smallskip
The nonconcentration argument follows
\cite[Section~4.4, Step~5]{CR09}. The conclusion is in $L^2$; the
nonsmoothness statement in Theorem~\ref{thm:main} gives
$u^*\notin L^p$ for every $p>2$, as in \cite{MR05CMP}.

Let $u^*$ be the $L^2$ limit of $u$ (and hence $\widetilde u$) outside the blowup
point given by \eqref{eq:46}, then $\widetilde u\rightharpoonup u^*$ in $L^2(\mathbb{R}^d)$
as $t\to T$, and we thus need to prove:
\[
 \left\lVert \widetilde u(t)\right\rVert_{L^2}\to\left\lVert u^*\right\rVert_{L^2}\quad\text{as }t\to T.
\]
Now arguing as for the proof of \eqref{eq:56}, we see that it suffices to prove:
\begin{equation}\label{eq:57}
 \left\lVert I_{N(t)}\widetilde u(t)\right\rVert_{L^2}\to\left\lVert u^*\right\rVert_{L^2}\quad\text{as }t\to T.
\end{equation}
Let us recall the smooth radial cut-off function $\chi(x)=0$ for $|x|\leq1$,
$\chi(x)=1$ for $|x|\geq2$. Let
\begin{equation}\label{eq:58} R(t)=10A(t)\lambda(t)\end{equation}
with $A$ given by \eqref{eq:20}. We first claim that:
\begin{equation}\label{eq:59}
 \int|u^*|^2=\lim_{t\to T}\int\chi\!\left(\frac{x-x(t)}{R(t)}\right)|I_Nu(t)|^2.
\end{equation}
Let us assume \eqref{eq:59} and conclude the proof of \eqref{eq:57}.
From the space localization \eqref{eq:58}, \eqref{eq:59} easily implies as before:
\begin{equation}\label{eq:60}
 \int|u^*|^2=\lim_{t\to T}\int\chi\!\left(\frac{x-x(t)}{R(t)}\right)|I_N\widetilde u(t)|^2.
\end{equation}
Observe now that:
\[
 \int_{|x-x(t)|\leq2R(t)}|I_N\widetilde u(t)|^2
 \lesssim\int_{|y|\leq20A(t)}|I_{N\lambda}\varepsilon|^2.
\]
Since $d\geq3$, H\"older's inequality and Sobolev embedding give
\begin{equation}\label{eq:61}
\begin{aligned}
 \int_{|y|\leq20A(t)}|I_{N\lambda}\varepsilon|^2
 &\lesssim_d A^2
 \left\|I_{N\lambda}\varepsilon\right\|_{L^{\frac{2d}{d-2}}}^2\\
 &\lesssim_d A^2\int|\nabla I_{N\lambda}\varepsilon|^2\\
 &\lesssim_d \Gamma_b^{\frac23-3a}\longrightarrow0
 \qquad\text{as }t\to T,
\end{aligned}
\end{equation}
where we used \eqref{eq:20}, \eqref{equ:lamtimpse}, and
$3a<\frac1{48}$. We conclude that
\[
 \int_{|x-x(t)|\leq2R(t)}|I_N\widetilde u(t)|^2
 \longrightarrow0
 \qquad\text{as }t\to T,
\]
and \eqref{eq:60} then implies \eqref{eq:57}.

It therefore only remains to prove \eqref{eq:59}.

\noindent\emph{Proof of \eqref{eq:59}.}
We apply the operator $I_{N(t)}$ to \eqref{equ:nls} and compute:
\begin{equation}\label{eq:62}
 i\partial_t(I_{N(t)}u)+\Delta(I_{N(t)}u)=i\widetilde I_{N(t)}u-I_{N(t)}[u|u|^{4/d}],
\end{equation}
where $\tilde{I}_{N(t)}$ is defined as in \eqref{eq:26}. we then compute the flux of $L^2$ norm at the fixed time $t$:
\begin{align*}
 &\frac12\frac d{d\tau}\int\chi\!\left(\frac{x-x(\tau)}{R(t)}\right)|I_{N(\tau)}u(\tau)|^2\\
 &\quad=\frac1{R(t)}\operatorname{Im}\!\left(\int\nabla\chi\!\left(\frac{x-x(\tau)}{R(t)}\right)
           \cdot\nabla I_{N(\tau)}u\,\overline{I_{N(\tau)}u}\right)\\
 &\qquad-\frac{x^{\prime}(\tau)}{2R(t)}\cdot\int\nabla\chi\!\left(\frac{x-x(\tau)}{R(t)}\right)|I_{N(\tau)}u(\tau)|^2\\
 &\qquad+\operatorname{Re}\!\left(\int\chi\!\left(\frac{x-x(\tau)}{R(t)}\right)
          \widetilde I_{N(\tau)}u\,\overline{I_{N(\tau)}u}\right)\\
 &\qquad+\operatorname{Im}\!\left(\int\chi\!\left(\frac{x-x(\tau)}{R(t)}\right)
          \bigl(I_{N(\tau)}u|I_{N(\tau)}u|^{4/d}-I_{N(\tau)}(u|u|^{4/d})\bigr)\overline{I_{N(\tau)}u}\right)
\end{align*}
and integrate from $t$ to $T$. Using also \eqref{eq:56}, we obtain that
\begin{equation}\label{eq:63}
\begin{aligned}
 &\left|\int\chi\!\left(\frac{x-x(T)}{R(t)}\right)|u^*|^2
       -\int\chi\!\left(\frac{x-x(t)}{R(t)}\right)|I_Nu(t)|^2\right|\\
 &\quad\lesssim\frac1{A(t)\lambda(t)}\int_t^T\left\lVert \nabla I_{N(\tau)}u(\tau)\right\rVert_{L^2}\,d\tau\\
 &\qquad+\frac1{A(t)\lambda(t)}\int_t^T\left|\frac{x_s}\lambda\right|\frac{d\tau}{\lambda(\tau)}\\
 &\qquad+\int_t^T\left|\int\chi\!\left(\frac{x-x(\tau)}{R(t)}\right)
                          \widetilde I_Nu\,\overline{I_Nu}\,dx\right|\,d\tau\\
 &\qquad+\int_t^T\int\left|I_Nu\bigl(I_Nu|I_Nu|^{4/d}-I_N(u|u|^{4/d})\bigr)\right|\,dx\,d\tau.
\end{aligned}
\end{equation}
We now claim that all terms on the right hand side of \eqref{eq:63} go to zero as
$t\to T$. For the first one, we first observe from \eqref{eq:43} and \eqref{eq:20} that
\[
 A(t)\geq|\log(T-t)|^{Ca}
\]
and thus from \eqref{equ:vartass}:
\begin{equation}\label{eq:64}
\begin{aligned}
 &\frac1{A(t)\lambda(t)}\int_t^T\left\lVert \nabla I_{N(\tau)}u(\tau)\right\rVert_{L^2}\,d\tau
 \lesssim\frac1{\lambda(t)A(t)}\int_t^T\frac{d\tau}{\lambda(\tau)}\\
 &\quad\lesssim\frac1{|\log(T-t)|^{Ca}}
 \sqrt{\frac{\log|\log(T-t)|}{T-t}}
 \int_t^T\sqrt{\frac{\log|\log(T-\tau)|}{T-\tau}}\,d\tau\\
 &\quad\to0\quad\text{as }t\to T.
\end{aligned}
\end{equation}
For the second term, we have from \eqref{source:translationfinal}:
\[
 \frac1{A(t)\lambda(t)}\int_t^T\left|\frac{x_s}\lambda\right|\frac{d\tau}{\lambda(\tau)}
 \lesssim\frac1{A(t)\lambda(t)}\int_t^T\frac{1+F(s(\tau))}{\lambda(\tau)}\,d\tau.
\]
The first term is estimated as for the proof of \eqref{eq:64}, and for the second
one we use \eqref{eq:9} and \eqref{equ:almons}:
\begin{align*}
 \frac1{A(t)\lambda(t)}\int_t^T\frac{F(s(\tau))}{\lambda(\tau)}\,d\tau
 &=\frac1{A(t)\lambda(t)}\int_s^{+\infty}\lambda(s)F(s)\,ds\leq\frac{\lambda^{\alpha_2}(t)}{A(t)}\to0\quad\text{as }t\to T.
\end{align*}
For the third term, we use \eqref{source:movingI} and \eqref{eq:52}:
\[
 \int_t^T\left\lVert \widetilde I_{N(\tau)}u(\tau)\right\rVert_{L^2}\,d\tau
 \lesssim\sum_{k\geq k_t}k[N(t_k)\lambda(t_k)]^{-1}
       +\int_{s(t)}^{+\infty}F(s)\,ds\longrightarrow0.
\]
For the nonlinear term, we let $k_t$ such that $\lambda(t)\sim\frac1{2^{k_t}}$
and split $[t,T)$ into LWP intervals $(\tau_k^j)_{1\leq j\leq J_k}$ given by
Lemma~\ref{lem:hslwp}. We then estimate from Lemma~\ref{commutor} and Remark~\ref{remark2}:
\begin{align*}
 &\int_t^T\int\left|
 \big[I_{N(\tau)}(u|u|^{4/d})-I_{N(\tau)}u|I_{N(\tau)}u|^{4/d}\big]
 I_{N(\tau)}u\right|\,dx\,d\tau\\
 &\quad\lesssim\sum_{k\geq k_t}\sum_{j=1}^{J_k}
 \left\lVert I_{N(\tau)}(u|u|^{4/d})
                -I_{N(\tau)}u|I_{N(\tau)}u|^{4/d}\right\rVert_
        {L_\tau^2L_x^{\frac{2d}{d+2}}([\tau_k^j,\tau_k^{j+1}])}\\
 &\hspace{32mm}\times
 \left\lVert I_{N(\tau)}u\right\rVert_
        {L_\tau^2L_x^{\frac{2d}{d-2}}([\tau_k^j,\tau_k^{j+1}])}\\
 &\quad\lesssim\sum_{k\geq k_t}k[N(t_k)\lambda(t_k)]^{-(1+\nu)}\longrightarrow0
       \quad\text{as }k_t\to+\infty.
\end{align*}
We now pass to the limit $t\to T$ in \eqref{eq:63} and \eqref{eq:59} follows.
The nonsmoothness of $u^*$ follows as in~\cite{MR05CMP}: the errors
involving $I_N-\mathit{Id}$ are smaller than any fixed power of $b$,
so that the radiative mass estimate gives
\[
 \int_{|x-x(T)|<R}|u^*(x)|^2\,dx
 \gtrsim(\log|\log R|)^{-2}
\]
for $R$ sufficiently small, which excludes $u^*\in L^p$ for $p>2$.
This concludes the proof of Proposition~\ref{prop:main} and so Theorem~\ref{thm:main} follows.

\begin{center}

\end{center}

\end{document}